\documentclass[10pt]{article}  
\usepackage{amssymb,upgreek,amsmath}
\usepackage{makecell, boldline}
\usepackage{booktabs}
\usepackage{mathtools, bm,bbm}
\usepackage{esint}
\usepackage{subfigure, tabularx, amsthm}
\usepackage{enumitem, cite, color}
\usepackage{multirow}
\usepackage{enumerate}

\usepackage{geometry}

\usepackage[OT2,T1]{fontenc}
\DeclareSymbolFont{cyrletters}{OT2}{wncyr}{m}{n}
\DeclareMathSymbol{\Sha}{\mathalpha}{cyrletters}{"58}

\newtheorem{proposition}{Proposition}
\newtheorem{theorem}{Theorem}
\newtheorem{lemma}{Lemma}
\newtheorem{corollary}{Corollary}

\newtheorem{definition}{Definition}

\newcommand{\C}{ \mathbb{C}}
\newcommand{\R}{ \mathbb{R}}
\newcommand{\T}{ \mathbb{T}}
\newcommand{\Z}{ \mathbb{Z}}
\newcommand{\N}{ \mathbb{N}}

\newcommand{\Lop}{{\rm L}}
\newcommand{\Kop}{{\rm L^{\dagger}}}

\newcommand{\Dop}{{\rm D}}
\newcommand{\Id}{\mathrm{Id}}

\newcommand{\Span}{\mathrm{Span}}

\newcommand{\Sp}{ \mathcal{S}}

\newcommand{\ek}{{e}_{\bm{k}}}

\newcommand{\ML}{\mathcal{M}_{\Lop}}
\newcommand{\KL}{K_{\Lop}}
\newcommand{\CL}{\mathcal{C}_{\Lop}}
\newcommand{\NL}{\mathcal{N}_{\Lop}}
\newcommand{\Proj}{\mathrm{Proj}}

\def\One{\mathbbm{1}}

\def\mytitle{TV-based Reconstruction of  Periodic Functions}

\begin{document}

\title{\mytitle}
\author{Julien Fageot and Matthieu Simeoni}

\maketitle

\begin{abstract}
We introduce a general framework for the reconstruction of periodic multivariate functions from finitely many and possibly noisy linear measurements.
The reconstruction task is formulated as a penalized convex optimization problem, taking the form of a sum between a convex data fidelity functional and a sparsity-promoting total variation based penalty involving a suitable spline-admissible regularizing operator $\Lop$.
In this context, we establish a periodic representer theorem, showing that the extreme-point solutions are periodic $\Lop$-splines with less knots than the number of measurements.
The main results are specified for the broadest classes of measurement functionals, spline-admissible operators, and convex data fidelity functionals.
We exemplify our results for various regularization operators and  measurement types (\emph{e.g.}, spatial sampling, Fourier sampling, or square-integrable functions). We also consider the reconstruction of both univariate and multivariate periodic functions.
\end{abstract}

\textit{Keywords.} Periodic operators, splines, total variation norm, optimization on measure spaces, representer theorems, native spaces.
 
\section{Introduction}\label{sec:intro}
	\subsection{Total Variation Regularization for Periodic Inverse Problems} \label{sec:problem}

	The development of optimization-based methods for the reconstruction of functions from finitely many linear  measurements has been an important subject of recent investigation. This paper participates to this effort by considering the special case of periodic functions defined over the $d$-dimensional torus. More specifically, our goal is to recover an unknown periodic function $f_0$ from finitely many observations $y_1, \ldots , y_M \in \R$ of the latter.
	The real function $f_{0}$ is defined over the torus $\T^d = \R^d / 2 \pi \Z^d$ with ambiant dimension $d\geq 1$.  
	The $M$ observations $y_1, \ldots , y_M$ are possibly noise-corrupted versions of  noiseless measurements  depending linearly on $f_{0}$. This linear relationship can be modelled in terms of $M$ linear measurement functionals $f \mapsto \langle \nu_m ,f \rangle \in \R$, $m=1,\ldots,M$.
	 Note that since the unknown function $f_0$ is infinite-dimensional and  the data finite-dimensional, the reconstruction task is dramatically ill-posed and must therefore be regularized. This can be achieved by considering a convex penalized optimization problem of the form 
	\begin{equation} \label{eq:optipb}
		\tilde{f}\in\arg\min_f E( \bm{y} , \bm{\nu}(f)) + \lambda \lVert \Lop f \rVert_{\mathcal{M}},
	\end{equation}
	where:
	\begin{itemize}
	\item $E$ is a suitable convex cost functional encouraging  the measurement vector $\bm{\nu}(f) = (\langle \nu_m , f \rangle)_{m=1\ldots M}$ to be close to the observation vector $\bm{y} = (y_1, \ldots , y_M)$. 
	\item $\Lop$ is a suitable regularizing operator acting on periodic functions. 
	\item $\lVert \cdot \rVert_{\mathcal{M}}$ is the total variation (TV) norm of a periodic Radon measure.
	\item $\lambda$ is a positive constant defining the penalty strength.
\end{itemize}
	The penalty term $ \lambda \lVert \Lop f \rVert_{\mathcal{M}}$ in \eqref{eq:optipb} helps regularizing the ill-posed reconstruction problem. Moreover, the total variation norm is known to promote sparse spline-like solutions~\cite{Unser2017splines}, similarly to the $\ell_1$ norm in finite dimension.
	
	\subsection{Comparison with Previous Works} \label{sec:comparison}
	
		\paragraph{Discrete $\ell_1$ reconstruction methods.}
	The problem of reconstructing an unknown physical quantity, or signal, from incomplete measurements has a rich history in applied mathematics, dealing both with discrete and continuous settings. In the former, the signal is modeled as a vector (finite dimensional setting) or a sequence (infinite dimensional discrete setting). The problem is then ill-posed in the sense that we do not have enough measurements to  uniquely characterize  the unknown vector (underdetermined system of equation). Regularization methods have been introduced in order to make the problem well-posed, starting with the Tikhonov regularization based on the $\ell_2$-norm~\cite{Tikhonov1963solution}, also known as ridge regression in statistics~\cite{Hoerl1962ridge}.
	In the 90's, it has been recognized that $\ell_1$ regularizations  are much better for providing sparse and interpretable reconstructions, leading to the LASSO~\cite{tibshirani1996regression} and the basis pursuit~\cite{chen2001atomic} and then inspiring the field of compressed sensing~\cite{Donoho2006,Candes2006sparse,chandrasekaran2012convex,eldar2012compressed,Foucart2013mathematical}.
	These ideas have been initially developed in finite dimension, and have been extended to infinite-dimensional settings by several authors~\cite{Adcock2015generalized,Adcock2017breaking,eldar2008compressed,unser2016representer,Traonmilin2017compressed,Bodmann2018compressed,marz2020sampling}.
	
		\paragraph{Dirac recovery and optimization over measure spaces.}

	Many physical quantities are not adequately described using the discrete formalism introduced above. This is typically the case for applications where the physical phenomenon of interest involves point sources lying in a continuum, often  modeled as Dirac streams $w_0 = \sum a_k \delta_{\bm{x}_k}$, \emph{i.e.} weighted sums of Dirac impulses. Such generalized functions are characterized by a finite rate of innovation, which is a continuous-domain generalization of the classically discrete notion of sparsity~ \cite{Vetterli2002FRI,Maravic2005sampling}. A key problem is then to reconstruct the unknown Dirac stream  from finitely many observations.
	
	In \cite{candes2014towards,Candes2013super}, Candès and Fernandez-Granda considered the super-resolution problem, aiming at recovering $w_0$ from low-pass Fourier measurements. Remarkably, they showed that it is possible to reconstruct \emph{perfectly} the infinite-dimensional measure from finitely many Fourier measurements as soon as the Dirac impulses are sufficiently far apart.
  	In this framework, the reconstructed Dirac stream is the solution of an optimization task over Radon measures, where the TV norm is used as a regularization, which corresponds to \eqref{eq:optipb} with $\mathrm{L} = \mathrm{Id}$.  Optimization problems over measure spaces can be traced back to the 20th century~\cite{zukhovitskii1956approximation,Fisher1975}, and have been the topic of many works  in the recent years~\cite{deCastro2012exact,Bredies2013inverse,Duval2015exact,Azais2015Spike,FernandezGranda2016Super,Chambolle2016geometric,duval2017sparseI,duval2017sparseII,Denoyelle2017support,flinth2020linear,chi2020harnessing,garcia2020approximate}. 
	These include algorithmic schemes adapted to Dirac recovery, with recent applications to super-resolution microscopy~\cite{denoyelle2019sliding} and cloud tracking~\cite{courbot2019sparse}. 
		
		\paragraph{TV beyond measures: the splines realm.}
		Although relevant in some applications, Dirac stream recovery remains quite specific. In particular, it cannot be applied to physical quantities admitting pointwise evalutations. Generalizations of the framework to the reconstruction of such objects started with~\cite{Unser2017splines}, with already some roots in~\cite{Fisher1975}. By adding a pseudo-differential operator to the regularization term, similarly to \eqref{eq:optipb}, this new reconstruction paradigm was able to reconstruct smooth solutions while maintaining the sparsity-promoting effect of the TV norm. In this setting, the Dirac streams are replaced by splines in the solution set. Splines are piecewise-smooth functions with  finite rates of innovation, whose  smoothness can be adapted by adequately choosing the differential operator~\cite{Schoenberg1973cardinal,Unser2017splines}.
		Since then, algorithmic schemes have been developed~\cite{gupta2018continuous,flinth2019exact,Debarre2019,debarre2020sparsest}, and important extensions have been proposed, generalizing the framework to other Banach spaces such as measure spaces over spherical domains~\cite{simeoni2020functionalpaper,simeoni2020functional}, hybrid spaces~\cite{debarre2019hybrid}, multivariate settings~\cite{aziznejad2018l1}, or more general abstract settings~\cite{bredies2018sparsity,boyer2019representer,unser2019native}. Applications include geophysical and astronomical data reconstruction~\cite{simeoni2020functionalpaper,simeoni2020functional}, neural networks~\cite{unser2019representer,aziznejad2020deep}, and image analysis~\cite{novosadova2018image}. 		
		
		\paragraph{Optimization in periodic function spaces.}
		
		Several works for Dirac recovery have been developed over the torus, and are therefore tailored for periodic Dirac streams~\cite{Vetterli2002FRI,Fisher1975,deCastro2012exact,candes2014towards,Duval2015exact,Denoyelle2017support, simeoni2020cpgd}.  Contrarily to the non periodic setting, the extension to arbitrary periodic functions has only received a limited attention so far, mostly in the works of Simeoni~\cite{simeoni2020functionalpaper,simeoni2020functional}. There, the author considers functions over the $d$-dimensional sphere $\mathbb{S}^d$, which coincides with the univariate periodic case for $d=1$. The proposed reconstruction framework is however limited to invertible regularization operators, hence excluding important standard differential operators. 
		There are strong motivations to develop a periodic framework. Periodic reconstruction methods can be used for the parametric representation of closed curves~\cite{Delgado2012ellipse,Uhlmann2016hermite} or the interpolation of periodic functions~\cite{light1992interpolation,Jacob2002sampling}. They can also be used to model 2D acoustical and electromagnetic wave fields,  encountered in the field array signal processing \cite{simeoni2019deepwave,pan2017frida, krim1996two}.
		Inverse problems for periodic functions have been considered with Tikhonov $L_2$-regularizations, a complete treatment being proposed in~\cite{Badou}. The present paper can be seen as the TV adaptation of this work.
	
	\paragraph{Generalized measurements.}

		One specificity of the functional setting in comparison with its discrete counterpart is that the measurement process, modelled via measurement functionals $\nu_m$, deserves a special attention.
		For infinite-dimensional optimization problems, the space of linear measurement functionals is indeed intimately linked to the search space of the optimization problem. 
		For instance, the search space for non periodic Dirac recovery is  the space of Radon measures, which can  only be sensed by continuous functions vanishing at infinity~\cite{Fisher1975,Unser2017splines}, possibly with additional smoothness technical conditions for specific tasks such as support recovery~\cite{Duval2015exact}.
		
		This problem has been addressed in various ways for generalizations of the Dirac recovery problem.
		The main motivation is to determine whether some practical measurement procedures are adapted to the considered optimization task. This includes spatial sampling or Fourier sampling, that were considered for instance in~\cite{Debarre2019,gupta2018continuous}, and that we shall also consider.  
		Most of the time, theoretical works provide sufficient conditions on the measurement functionals so that the optimization problem is well-posed and admits well-characterized solutions via representer theorems~\cite{Unser2017splines}. However, to the best of our knowledge, the only framework providing necessary and sufficient conditions over the measurement functionals to achieve this goal is proposed in the non periodic setting~\cite{unser2019native}. As we shall see, our present work provides complete answers to such questions in the periodic setting.	
	
	\subsection{Contributions and Outline} \label{sec:contributions}

	This paper introduces a very generic total variation-based optimization framework for the reconstruction of \emph{periodic} functions. Our main contributions are the following:
	
	\begin{itemize}
	
	\item We provide a rigorous and exhaustive functional-analytic framework for optimization problems of the form~\eqref{eq:optipb}. This requires (i) to identify the Banach space, called the \emph{native space}, on which the optimization problem is well-posed, and (ii) to characterize the space of the linear measurement functionals $\nu_m$, called the \emph{measurement space}, such that the measurement process $f \mapsto \langle \nu_m , f \rangle$ is well-defined and has relevant topological properties. To the best of our knowledge, our work is the first to date to provide a definitive answer to both questions for arbitrary spline-admissible   operators $\Lop$.
	
	\item We demonstrate a representer theorem for the extreme-point solutions of the optimization problem \eqref{eq:optipb}. The latter are periodic $\Lop$-splines whose number of knots is smaller or equal to the number of measurements $M$. This is the periodic counterpart of recent representer theorems obtained in non periodic settings \cite{Unser2017splines,boyer2019representer,bredies2018sparsity}. Our representer theorem is moreover not directly deducible from these results.
	
	\item We give necessary and sufficient conditions for the admissibility of several measurement procedures often used in practice to sense the unknown function $f_0$. These conditions  only depend  on the properties of the regularizing operator $\Lop$. We consider notably the case of spatial sampling, Fourier sampling, and square-integrable filtering.
	
	\item We exemplify our general framework on various classes of pseudo-differential operators and their corresponding periodic splines. This includes classical differential operators such as the derivative and the Laplacian, together with polynomial or fractional generalizations of the latter. We moreover introduce other operators whose splines, called Mat\'ern and Wendland splines,  are  characterized by their excellent localization properties.
	All the results, including the sampling-admissibility, are examplified on these operators.
	To the best of our knowledge, this is the first characterization of the classical pseudo-differential operators that are sampling-admissible.
	In addition, we provide Python routines for efficiently generating and manipulating the various multivariate periodic $\Lop$-splines considered in this paper on the public GitHub repository \cite{periodispline2020}. We give illustrative examples in dimension $d=1,2,3$, prevailing in practice. 
		
	\end{itemize}

	The paper is organized as follows.
	In Section \ref{sec:opandsplines}, we introduce the class of spline-admissible periodic operators and their corresponding $\Lop$-splines.
	The native space and the measurement space of the optimization problem~\eqref{eq:optipb} are constructed in Section \ref{sec:constructionNativeSpace}. 
	The periodic representer theorem associated to the optimization problem \eqref{eq:optipb} is derived in Section \ref{sec:RT}.
	Examples of admissible operators and of linear functionals are given in Sections \ref{sec:differentialop} and \ref{sec:criteria} respectively. 
	Finally, we conclude in Section \ref{sec:discussion}.

\section{Periodic Functions, Operators, and Splines} \label{sec:opandsplines}

	This section is dedicated to the introduction of periodic spline-admissible operators. We first provide some definitions and results on functions, operators, and splines in the periodic setting.

	\subsection{Periodic  Function Spaces and Generalized Fourier Series}\label{sec:vocabulary}
	
	\paragraph{Generalized Periodic Functions.} The \emph{Schwartz space} of \emph{infinitely differentiable} periodic functions is denoted by $\Sp(\T^d)$. It is endowed with its usual nuclear Fr\'echet topology~\cite{Treves1967} associated to the familly of norms
	\begin{equation}
		\lVert \varphi \rVert_{N} := \left( \sum_{0 \leq \lvert \bm{n} \rvert \leq N} \lVert \mathrm{D}^{\bm{n}} \{\varphi\} \rVert_{2}^2 \right)^{1/2}, \qquad \forall \varphi \in \Sp(\T^d),
	\end{equation}	
	where $N \in \N$,  $\bm{n} := (n_1,\ldots, n_d) \in \N^d$, $\lvert \bm{n} \rvert := n_1 + \cdots + n_d$, and $\mathrm{D}^{\bm{n}} := \Pi_{i=1}^d \partial_i^{n_i}$.
	The topological dual of $\Sp(\T^d)$, denoted by $\Sp'(\T^d)$, is the space of continuous linear functionals $f:\Sp(\T^d)\rightarrow\R$, called \emph{generalized periodic functions}. The bilinear map $\langle\cdot,\cdot\rangle:\Sp(\T^d) \times \Sp'(\T^d) \rightarrow \R, \,(\varphi,f)\mapsto\langle f,\varphi\rangle:=f(\varphi)$ is called the  \emph{Schwartz duality product}. The inner product notation is not fortuitous since when $f \in L_1(\T^d)$, we have  $\langle f, \varphi \rangle = \frac{1}{(2\pi)^d}\int_{\T^d} f(\bm{x}) \varphi(\bm{x}) \mathrm{d}\bm{x}.$
The space $\Sp'(\T^d)$ can be endowed with the \emph{weak* topology}, induced by the family of semi-norms $\{\|f\|_\varphi:=|\langle f,\varphi\rangle|, \varphi\in\Sp(\T^d)\}$. This is the topology of \emph{pointwise convergence}: a sequence of generalized periodic functions $f_n \in \Sp'(\T^d)$ converges to $f\in \Sp'(\T^d)$ if $\lim_{n\rightarrow\infty}\langle f_n , \varphi \rangle=\langle f, \varphi \rangle$ for any test function $\varphi \in \Sp(\T^d)$.
	
\paragraph{Generalized Fourier Series.} 			For $\bm{k}\in \Z^d$, we define the sinusoid function  $e_{\bm{k}} : \bm{x} \mapsto \mathrm{e}^{  \mathrm{i}  \langle \bm{x},\bm{k}\rangle}, $	which is trivially in $\Sp(\T^d)$. Any generalized periodic function $f \in \Sp'(\T^d)$ can then be uniquely decomposed as $f = \sum_{\bm{k}\in\Z^d} \widehat{f}[\bm{k}] e_{\bm{k}}$,
	where the convergence of the series holds in $\Sp'(\T^d)$.
	The sequence $(\widehat{f}[\bm{k}])_{\bm{k}\in \Z^d} \in \C^{\Z^d}$, called the \emph{Fourier sequence}  of $f$, is   \emph{slowly growing}~\cite[Chapter VII]{Schwartz1966distributions}---\emph{i.e.}, such that $|\widehat{f} [\bm{k}]|=\mathcal{O}(\|\bm{k}\|^n)$ for some $n\in\N$. The \emph{Fourier coefficients} are given by $\widehat{f} [\bm{k}] := \langle f , e_{\bm{k}} \rangle$ for any $f \in \Sp'(\T^d)$. A periodic generalized function $\varphi \in \Sp'(\T^d)$ is in $\Sp(\T^d)$ if and only if its Fourier sequence is \emph{rapidly decaying}---\emph{i.e.},   $|\widehat{\varphi} [\bm{k}]|=o(\|\bm{k}\|^{-n}),\, \forall n\in\N$. 
	
\paragraph{Dirac Comb.} 	The \emph{Dirac comb} $\Sha\in\Sp'(\T^d)$ is characterized by the relation $\langle \Sha , \varphi \rangle := \varphi(0)$ for any $\varphi \in \Sp(\T^d)$. The Fourier sequence of $\Sha$ is hence $\widehat{\Sha} [\bm{k}] = \langle \Sha , e_{\bm{k}} \rangle = e_{\bm{k}}(0) = 1$, yielding $\Sha = \sum_{\bm{k}\in\Z^d} e_{\bm{k}}$. Seen as a generalized function over $\R^d$, we recover the usual definition of the Dirac comb, \emph{i.e.}, $\Sha = \sum_{\bm{n} \in \Z^d} \delta( \cdot - 2\pi\bm{n})$, with $\delta$ the Dirac impulse.

\paragraph{Periodic Sobolev Spaces.} 	The \emph{periodic Sobolev space of smoothness $\tau \in \R$} is defined by
	\begin{equation} \label{eq:sobolevspace}
		\mathcal{H}_{2}^{\tau}(\T^d) := \left\{  f \in \Sp'(\T^d) , \ \lVert f \rVert_{\mathcal{H}_2^\tau}^2 := \sum_{\bm{k}\in \Z^d} ( 1 + \lVert \bm{k} \rVert^2)^{\tau} \lvert \widehat{f} [\bm{k}] \rvert^2 < \infty  \right\}.
	\end{equation}
	It is a Hilbert space for the Hilbertian norm $\lVert \cdot \rVert_{\mathcal{H}_2^\tau}$. Moreover, we have for any $\tau_1, \tau_2 \in \R$ with $\tau_1 \leq \tau_2$ the dense  topological embeddings
	\begin{equation}
		\Sp(\T^d) = \cap_{\tau \in \R} \mathcal{H}_{2}^{\tau}(\T^d)  \subseteq \mathcal{H}_{2}^{\tau_2}(\T^d) \subseteq \mathcal{H}_{2}^{\tau_1}(\T^d) \subseteq \cup_{\tau \in \R} \mathcal{H}_{2}^{\tau}(\T^d)  = \Sp'(\T^d). 
	\end{equation}

	\subsection{Periodic Spline-Admissible Operators and their Periodic Green's Function}
		
	We denote by $\mathcal{L}_{\mathrm{SI}} (\Sp'(\T^d))$ the space of \emph{linear}, \emph{shift-invariant} operators that are \emph{continuous} from $\Sp'(\T^d)$ to itself.
	The following characterization in terms of Fourier sequences is well-known \cite[Section II.A]{Badou}.
	
\begin{proposition} \label{prop:firsttrivialstuff}
	Let $\Lop \in \mathcal{L}_{\mathrm{SI}} (\Sp'(\T^d))$. Then, the $e_{\bm{k}}$ are eigenfunctions of $\Lop$ and the sequence of eigenvalues $(\widehat{L}[\bm{k}])_{\bm{k}\in\Z^d}\subset \C$ such that $\Lop e_{\bm{k}} = \widehat{L}[\bm{k}] e_{\bm{k}}$ is slowly growing. Moreover, we have that, for any $f \in \Sp'(\T^d)$, 
	\begin{equation} \label{eq:Lf}
		\Lop \{f\} = \sum_{\bm{k}\in\Z^d} \widehat{L}[\bm{k}] \widehat{f}[\bm{k}] e_{\bm{k}},
	\end{equation}	
	where the convergence holds in $\Sp'(\T^d)$. 
	Conversely, any slowly growing sequence $(\widehat{L}[\bm{k}])_{\bm{k}\in\Z^d}$ specifies an operator $\Lop \in  \mathcal{L}_{\mathrm{SI}} (\Sp'(\T^d))$  via the relation \eqref{eq:Lf}.
\end{proposition}	
\textit{Remark.} The relation \eqref{eq:Lf} indeed specifies  an element of $\Sp '(\T^d)$ since the sequence $(\widehat{L}[\bm{k}] \widehat{f}[\bm{k}])_{\bm{k}\in \Z^d}$ is slowly growing as the element-wise product between two slowly growing sequences. \\

	 {Spline-admissible operators} are operators $\Lop\in\mathcal{L}_{\mathrm{SI}} (\Sp'(\T^d))$ for which the notion of $\Lop$-spline is well-defined.
	They include classical differential operators, together with their fractional versions~\cite{Unser2014sparse}. 
	Splines are usually considered over the complete real line~\cite{Unser2017splines}. We adapt here the construction to the periodic case. 
	
	\begin{definition}[Pseudoinverse] \label{def:pseudoinverse}
	Let $\Lop \in  \mathcal{L}_{\mathrm{SI}} (\Sp'(\T^d))$. 
	We say that $\Kop \in \mathcal{L}_{\mathrm{SI}} (\Sp'(\T^d))$ is a  \emph{pseudoinverse} of $\Lop$  if  it  satisfies the four relations $
	\Lop \Kop \Lop = \Lop$, $\Kop \Lop \Kop = \Kop$, $(\Lop \Kop)^* = \Lop \Kop$, and $(\Kop \Lop)^* = \Kop \Lop$.
	\end{definition}
	
		\textit{Remark.} The operator $\Kop$  is also called the Moore-Penrose pseudoinverse~\cite{campbell2009generalized}. When it exists, the pseudoinverse is known to be unique~\cite{ben2003generalized}. 
		In our case, we shall see thereafter that the two relations $\Lop \Kop \Lop = \Lop$ and $\Kop \Lop \Kop = \Kop$ are sufficient to characterize the pseudoinverse, and that they imply the self-adjoint relations $(\Lop \Kop)^* = \Lop \Kop$ and $(\Kop \Lop)^* = \Kop \Lop$. 
		
	\begin{definition}[Spline-Admissible Operator] \label{def:splineadmissible}
		An   operator $\Lop  \in  \mathcal{L}_{\mathrm{SI}} (\Sp'(\T^d))$ is said to be \emph{spline-admissible} if 
		\begin{itemize}
			\item it has a finite-dimensional null-space and
			\item it admits a pseudoinverse operator $\Kop  \in  \mathcal{L}_{\mathrm{SI}} (\Sp'(\T^d))$. 
		\end{itemize}
	\end{definition}
	
	Spline-admissible operators, their null space, and their pseudoinverse can be readily characterized by their Fourier sequence as follows.
	
	\begin{proposition} \label{prop:finitedimNL}
		Let $\Lop \in   \mathcal{L}_{\mathrm{SI}} (\Sp'(\T^d))$ and set $K_\Lop := \{\bm{k} \in \Z^d, \ |\widehat{L}[\bm{k}]|  \neq 0\}$ and $N_\Lop := \{\bm{k} \in \Z^d, \   |\widehat{L}[\bm{k}]|  = 0\}  = \Z^d \backslash K_{\Lop}$.
		Then,
		\begin{equation} \label{eq:NLKL}
		\NL =  \{ f \in \Sp'(\T^d):\, \widehat{f}[\bm{k}] = 0, \,\forall \bm{k} \in K_{\Lop} \} = \overline{\Span} \{e_{\bm{k}}, \  \bm{k} \in N_{\Lop} \}, 
		\end{equation}
		where $ \overline{\Span} \  A$ is the closure of the span of $A$ for the topology of $\Sp'(\T^d)$. 
		Hence, $\NL$ is finite dimensional if and only if finitely many $\widehat{L}[\bm{k}] $ are null. 
		Moreover, $\Lop$ admits a pseudoinverse $\Kop\in\mathcal{L}_{\mathrm{SI}} (\Sp'(\T^d))$ if and only if the sequence
	\begin{equation} \label{eq:defL+seq}
		 \widehat{L^\dagger}[\bm{k}] = 
		\begin{cases} 
		\widehat{L}[\bm{k}]^{-1}  &\text{ if } \bm{k} \in K_\Lop, \\
		0 & \text{ otherwise}
		\end{cases}
	\end{equation}
	is slowly growing. In that case, the Fourier sequence of the pseudoinverse $\Lop^\dagger$ is given by $( \widehat{L^\dagger}[\bm{k}] )_{\bm{k} \in \Z^d}$.
			
	\end{proposition}
	
		The first part of Proposition \ref{prop:finitedimNL}  on the null space of $\Lop$ is well-known in the non periodic setting with $d=1$, due to the following result: any shift-invariant and finite-dimensional linear subspace of $\Sp'(\R)$ is constituted of exponential polynomial functions~\cite{Unser2014sparse}. 	
	This result has been adapted in the periodic setting in~\cite[Proposition 1]{Badou}. We provide here a simple proof for the general case $d\geq 1$.
	
\begin{proof}[Proof of Proposition  \ref{prop:finitedimNL}]
We first observe that, for any  $\Lop \in  \mathcal{L}_{\mathrm{SI}} (\Sp'(\T^d))$, the null space $\NL$ is the $\Sp'(\T^d)$-closure of  $\Span \{ e_{\bm{k}}, \ \bm{k} \notin \KL\}$. This is simply due to the relation $\Lop f = \sum_{\bm{k} \in\KL} \widehat{L}[\bm{k}] \widehat{f}[\bm{k}] e_{\bm{k}}$, from which one deduces that $f \in \NL$ if and only if $\widehat{f}[\bm{k}] = 0$ for every $\bm{k} \notin \KL$, giving \eqref{eq:NLKL}.
It is then obvious that $\mathrm{dim} \NL = \mathrm{Card} ( \Z^d \backslash \KL) $ and is therefore finite dimensional if and only if $\Z^d \backslash \KL$ is finite.  

If the sequence $ (\widehat{L^\dagger}[\bm{k}])_{\bm{k}\in \Z^d}$ defined in \eqref{eq:defL+seq}  is slowly growing and denoting by $\Kop$ the corresponding operator, then, for any $f \in \Sp  '(\T^d)$, 
	\begin{equation}
	\Kop \mathrm{L} \Kop \{f\} = \sum_{\bm{k}\in K_{\mathrm{L}}}  \widehat{L}[\bm{k}]^{-1} \widehat{L}[\bm{k}] \widehat{L}[\bm{k}]^{-1}  \widehat{f}[\bm{k}] \ek
	=
	\sum_{\bm{k}\in K_{\mathrm{L}}}  \widehat{L}[\bm{k}]^{-1}  \widehat{f}[\bm{k}] \ek
	= \Kop \{f\}, 
	\end{equation} 
	which implies that that $\Kop \mathrm{L} \Kop = \Kop$. We show identically that $\Lop \Kop \Lop = \Lop$. Finally, we have $\Lop^\dagger\Lop f=\Lop\Lop^\dagger f=\sum_{\bm{k}\in K_\Lop} \widehat{f}[\bm{k}]$ yielding trivially $(\Lop^\dagger\Lop)^\ast=\Lop^\dagger\Lop$ and $(\Lop\Lop^\dagger)^\ast=\Lop\Lop^\dagger$. Therefore, $\Kop$ is indeed the pseudoinverse of $\Lop$.
	Conversely, if the pseudoinverse exists, we have in particular that
	\begin{equation}
		 \widehat{L}[\bm{k}] \ek = \Lop \ek = \Lop \Kop \Lop \ek =  \widehat{L}[\bm{k}]^2 \widehat{L^\dagger}[\bm{k}] \ek.
	\end{equation}
	Hence, $ \widehat{L^\dagger}[\bm{k}] = \widehat{L}[\bm{k}]^{-1}$ as soon as $\widehat{L}[\bm{k}] \neq 0$. Similarly, 
	\begin{equation}
		 \widehat{L^\dagger}[\bm{k}] \ek = \Kop \ek = \Kop \Lop  \Kop \ek =  \widehat{L^\dagger}[\bm{k}]^2 \widehat{L }[\bm{k}] \ek,
	\end{equation}	
	implying that $ \widehat{L^\dagger}[\bm{k}] = 0$ if  $ \widehat{L}[\bm{k}] = 0$. This shows that the sequence $ \widehat{L^\dagger}$ is given by \eqref{eq:defL+seq}, implying in particular that this sequence, coming from a LSI continuous operator, is slowly growing. 
\end{proof}
	
	Thanks to Proposition  \ref{prop:finitedimNL}, $\Lop$ is spline-admissible if and only if it vanishes on finitely many $e_{\bm{k}}$ and its Fourier sequence does not vanish faster than any rational function. This excludes notably  convolution operators $\Lop f= h * f$ where the impulse response $h \in \mathcal{S}(\T^d)$ is such that the sequence $(\widehat{h}[\bm{k}])_{\bm{k}\in \Z^d}$ never vanishes and decays exponentially fast. For such operators indeed, the sequence $(\widehat{L}[\bm{k}]^{-1} :=  \widehat{h}[\bm{k}]^{-1})_{\bm{k}\in \Z^d}$ is not slowly growing and is therefore not the Fourier sequence of a LSI operator from $\Sp'(\T^d)$ to itself.
	With our notation, the pseudoinverse of a spline-admissible $\Lop$ is given by
		\begin{equation}	
	\Kop f =  \sum_{\bm{k} \in \KL} \frac{\widehat{f}[\bm{k}]}{\widehat{L}[\bm{k}]} e_{\bm{k}}, \qquad \forall f\in\mathcal{S}'(\mathbb{T}^d).
	\end{equation}
	The orthogonal projector on the null space $\NL$ of a spline-admissible operator satisfies
	\begin{align} \label{projdef}
		\Proj_{\NL} f &= \sum_{n=1}^{N_0} \widehat{f}[\bm{k}_n] e_{\bm{k}_n}, \qquad \forall f\in\mathcal{S}'(\mathbb{T}^d),
	\end{align}
	where $\{k_1 ,\ldots, k_{N_0}\} = N_{\mathrm{L}} = \Z^d \backslash K_{\mathrm{L}}$. 
	From the definition of the pseudoinverse, it is moreover easy to obtain
	\begin{equation} \label{eq:LopKop}
		\Proj_{\NL}= \mathrm{Id} -\Lop \Kop  = \mathrm{Id} -\Kop \Lop.	\end{equation}
 We recover from \eqref{eq:LopKop} that a spline-admissible operator $\Lop$ is invertible, with inverse $\Lop^{-1} = \Kop$, if and only if its null space is trivial.

		\textit{Remark.} Spline-admissible operators are sometimes refer to as \emph{Fredholm operators}, which are operators between Banach spaces whose kernel (null space) and co-kernel (quotient of the output space with the range of the operators) are finite-dimensional. This concept can be extended to operators between $\Sp'(\T^d)$ to itself, and a spline admissible operator is therefore a Fredholm operator.
	\begin{definition}[Spectral Growth]\label{def:growth}
	Let $\Lop$ be a spline-admissible and $\widehat{L}$ its Fourier sequence. If there exists a parameter $\gamma \geq 0$ and some constants $0<A \leq B < \infty$ and $k_0 \geq 0$ such that
	\begin{equation} \label{eq:AGforglop}
	A   \lVert \bm{k} \rVert^{\gamma} \leq |\widehat{L} [ \bm{k} ]| \leq B  \lVert \bm{k} \rVert^{\gamma}, \qquad \forall  \lVert \bm{k} \rVert \geq k_0 \in \Z^d,
	\end{equation}
	then, we call $\gamma$ the \emph{spectral growth} of $\Lop$. For brevity, we write \eqref{eq:AGforglop} as $|\widehat{L} [ \bm{k} ]|=\Theta( \lVert \bm{k} \rVert^{\gamma}).$
	\end{definition}
	
	If it exists, the spectral growth of a spline-admissible operator is unique. It measures the impact of $\Lop$ on the smoothness of the input functions. For instance, for any $\tau \in \R$, $f \in \mathcal{H}_2^\tau(\T^d)$ if and only if $\Lop f \in \mathcal{H}_2^{\tau - \gamma} (\T^d)$, where the periodic Sobolev spaces are defined in \eqref{eq:sobolevspace}. Note that the existence of $\gamma$ is not guaranteed. 
	This is typically the case if the Fourier sequence of $\Lop$ does not behave purely polynomially (\emph{e.g.}, when $d = 1$ and $\widehat{\Lop}[{k}] =  \lvert {k} \rvert  \log ( 1 +  \lvert {k} \rvert )$) or alternates between two polynomial behaviors (\emph{e.g.}, $\widehat{\Lop}[2k] = k$ and $\widehat{\Lop}[2k+1] = k^2$). However, most of the usual pseudo-differential operators admit a spectral growth, as will be exemplified in Section \ref{sec:differentialop}.   
	
	As we shall see, an important ingredient for the definition of periodic splines is the notion of Green's function of an operator $\Lop$. It is defined as the $\Lop$-primitive of the Dirac comb $\Sha$.
	\begin{definition}[Green's Function] \label{def:green}
		Let $\Lop$ be a spline-admissible operator with pseudoinverse $\Kop$. The generalized function $g_{\Lop} := \Kop \Sha \in \Sp'(\T^d)$ is called the \emph{Green's function} of $\Lop$. 
	\end{definition}
	
	The Fourier sequence $(\widehat{g}_\Lop [\bm{k}])_{\bm{k}\in\Z^d}$ of the Green's function of $\Lop$ coincides with the Fourier sequence of $\Kop$. The proposed notion of Green's function differs from certain convention for pseudo-differential operators for functions in $\R^d$. It is conventional to define Green's functions as fundamental solutions to $\Lop g_{\Lop}=  \delta$. This is however not adapted to the torus (for which $\delta$ is replaced by the Dirac comb $\Sha$). Indeed, there exists no periodic generalized function $g_\Lop$ as soon as the null space of $\Lop$ is non trivial, because the generalized function $\Lop g_{\Lop}$ has vanishing Fourier coefficients at null space frequencies, contrary to $\Sha$. Definition \ref{def:green} is the periodic adaptation to usual spherical Green's function constructions, such as the ones used in \cite[Chapter 4]{freeden2008spherical} and \cite[Definition 5]{simeoni2020functionalpaper}. Note that the different notions coincides when the operator $\Lop$ is invertible \cite[Proposition 4.3]{simeoni2020functional}.  \\
		 
		\textbf{Running example $\Lop = \Dop^N$, $N\geq 1$.} 
	The Fourier sequence of $\Lop = \Dop^N$ is $\widehat{D^N} [k] = (\mathrm{i} k)^N$.  	
	Then, the null space of $\Dop^N$ consists in the constant functions and is of dimension $N_0 = 1$ with $N_\Lop = \{0\}$.   
	The pseudoinverse of $\Dop^N$ is the operator with Fourier sequence $\widehat{(D^N)^{\dagger}} [{k}] = \frac{1-\delta[{k}]}{(\mathrm{i} k)^N}$ with $\delta[\cdot]$ the Kronecker delta, which is therefore the Fourier sequence of $g_{\Dop^N}$.
	Having a finite-dimensional null space and a pseudoinverse, the operator $\Dop^N$ is  spline-admissible in the sense of Definition \ref{def:splineadmissible}.  Moreover, it admits a spectral growth $\gamma = N$. 
	
	\subsection{Periodic $\Lop$-splines}
	
	The class of spline-admissible operators acting on generalized periodic functions allow us to adapt the classical notion of splines~\cite{Unser1999splines} to the periodic setting in full generality. 	
		
	\begin{definition}[Periodic $\Lop$-spline] \label{def:periodicsplines}
		Let $\Lop$ be a spline-admissible operator. 
		A \emph{periodic $\Lop$-spline} is a function $f \in \Sp'(\T^d)$ such that 
		\begin{equation} \label{eq:Lopfspline}
		 	\Lop 	f = \sum_{k=1}^K a_k \Sha( \cdot - \bm{x}_k),	
		\end{equation}
		with $K \geq 0$ is the number of knots ($K=0$ corresponds to the case where $\Lop f = 0$, \emph{i.e.}, $f$ is in the null space of $\Lop$), $\bm{x}_k \in \T^d$ are the distinct \emph{knots},  and $a_k \in \R \backslash\{0\}$ are the weights of $f$. The pairs $\{(a_k,\bm{x}_k), \, k=1,\ldots,K \}$ are called the \emph{innovations} of the periodic $\Lop$-spline.
	\end{definition}		
	In words, a periodic $\Lop$-spline is a generalized periodic function such that its $\Lop$-derivative is a Dirac stream with finite rate of innovation \cite{Vetterli2002FRI} --i.e. it has a finite number of innovations per period. 
		It is worth noting that the weights of a periodic $\Lop$-spline fulfill a linear system as soon as the null space of $\Lop$ is non trivial. For instance, if $ \widehat{L}[\bm{0}]  = \Lop e_{\bm{0}} =  0$, as is the case for $\Lop = \mathrm{D}$ in dimension $d=1$,  then, any $\Lop$-spline $f$  with weights $a_k$ satisfies  $ \sum_{k=1}^K a_k = \widehat{\mathrm{L} f }[\bm{0}]  = \widehat{L}[\bm{0}] \widehat{f}[\bm{0}] = 0$. 
		We generalize this idea to any spline-admissible operator. 
				
	\begin{proposition} \label{prop:Mnkmatrix}
	Let $\Lop   \in  \mathcal{L}_{\mathrm{SI}} (\Sp'(\T^d))$ be a spline-admissible operator with Green's function $g_{\Lop}\in\mathcal{S}'(\mathbb{T}^d)$ and null space $\NL$ with null space frequencies $N_{\Lop}:=\{\bm{k}_1,\ldots,\bm{k}_{N_0}\}\subset\mathbb{Z}^{d}$. 
	There exists an $\Lop$-spline of the form \eqref{eq:Lopfspline} if and only if the weight vector $\bm{a}:=(a_1,\ldots,a_K)\in\R^K$ satisfies $\bm{\mathrm{M}}  \bm{a} = \bm{0}$ with $\bm{\mathrm{M}} \in \R^{N_0 \times K }$ the matrix whose entries are
	\begin{equation} \label{eq:matrixM}
		M[n,k] = \mathrm{e}^{- \mathrm{i} \langle  \bm{k}_n, \bm{x}_k \rangle } , \qquad \forall n=1,\ldots , N_0, \forall k =1, \ldots , K.
	\end{equation}

	In that case, the generic form of a periodic $\Lop$-spline $f$ satisfying \eqref{eq:Lopfspline} is
		\begin{equation} \label{eq:splineform}
	f = \sum_{k=1}^K a_k g_{\Lop}( \cdot -\bm{x}_k) + p
	\end{equation}
	with $\bm{\mathrm{M}}  \bm{a} = \bm{0}$ and $p \in \NL$. 
	\end{proposition}
	
	\begin{proof}
	Assume first that $f$ satisfies \eqref{eq:Lopfspline}. Then,  we have that, for any $1 \leq  n \leq N_0$,
	\begin{equation}
	0 = \widehat{L}[\bm{k}_n ] \widehat{f}[\bm{k}_n] = \langle \Lop f , e_{\bm{k}_n} \rangle = \sum_{k=1}^K a_k \mathrm{e}^{- \mathrm{i}  \langle \bm{x}_k , \bm{k}_n\rangle}= (\bm{\mathrm{M}}  \bm{a})_n,
	\end{equation}
	or equivalently, $\bm{\mathrm{M}}  \bm{a} = \bm{0} \in \R^{N_0}$. 
	Assume now that $\bm{a} \in \R^K$ satisfies $\bm{\mathrm{M}}  \bm{a} = \bm{0}$. We set $w = \sum_{k=1}^K a_k \Sha( \cdot - \bm{x}_k)$. 
	Then, 
	\begin{equation}
	\Proj_{\NL} \{w\} = \sum_{n=1}^{N_0} \widehat{w} [\bm{k}_n] e_{\bm{k}_n} =  \sum_{n=1}^{N_0} \sum_{k=1}^K a_k \mathrm{e}^{- \mathrm{i}  \langle \bm{x}_k , \bm{k}_n\rangle} = 0.
	\end{equation}
	This implies, using \eqref{eq:LopKop}, that 
	\begin{equation}
	\Lop \Kop w = w -  \Proj_{\NL} w  = w.
	\end{equation}
	Then, $f = \Kop w=\sum_{k=1}^K a_k g_{\Lop}( \cdot -\bm{x}_k)$ satisfies $\Lop f = \Lop \Kop w = w = \sum_{k=1}^K a_k \Sha( \cdot - \bm{x}_k)$, and is therefore a periodic $\Lop$-spline.
	Moreover, two periodic $\Lop$-spline  solutions of \eqref{eq:Lopfspline} only differs from a null space component $p \in \NL$, implying \eqref{eq:splineform}. 
	Finally, when $K \leq N_0$, because the matrix $M$ imposes $N_0$ independent conditions on the vector $\bm{a}$ of size $K$, we deduce that $\bm{a}=\bm{0}$ and $f = p \in \NL$.
	\end{proof}
	
	Proposition \ref{prop:Mnkmatrix} essentially tells us that periodic $\Lop$-splines can be written as sums of a linear combination of shifts of the Green's functions and a trigonometric polynomial in the null space of $\Lop$. The shifts are moreover given by the spline knots and the weights are constrained to verify a certain annihilation equation.    In particular, an important consequence of Proposition \ref{prop:Mnkmatrix} is that the Green's function $g_\Lop = \Kop \Sha$ is \emph{not} a periodic $\Lop$-spline when the null space of $\Lop$ is non trivial. Indeed, we have that $\Lop \Kop \Sha = \Sha - \Proj_{\NL} \Sha$, which is not of the form \eqref{eq:Lopfspline} due to the trigonometric polynomial $\Proj_{\NL} \Sha = \sum_{n=1}^{N_0} e_{\bm{k}_n}$. \\
	
	\textbf{Running example $\Lop = \Dop^N$, $N\geq 1$.} 
	Let $f$ be a periodic $(\Dop^N)$-spline with knots $x_1, \ldots , x_K$. Then, $f$ is a piecewise-polynomial. More precisely, $f$ is a polynomial of degree at most $(N-1)$ on each intervals $[x_{k+1} , x_k]$, $k = 1, \ldots, K$ (with the convention that $x_{K+1} = x_1$). Moreover, for $N \geq 2$, $f$ has continuous derivatives up to order $(N-2)$. 
	 In particular, a periodic $\Dop$-spline is piecewise constant and a periodic $(\Dop^2)$-spline is piecewise linear and continuous.
	A non constant periodic $(\Dop^N)$-spline $f$ has at least $K = 2$ knots. In particular, there is no periodic $(\Dop^N)$-spline with only $1$ knots. Indeed, such a spline would be such that $\mathrm{D}^N f = a_1 \Sha(\cdot - x_1)$ and would satisfy, according to Proposition \ref{prop:Mnkmatrix}, $\mathrm{e}^{- \mathrm{i} x_1} a_1 = 0$, hence $a_1 = 0$ and $f$ would be constant, which we excluded.	

	\section{Periodic Native Spaces and Measurement Spaces} \label{sec:constructionNativeSpace}
		
		The goal of this section is to construct the Banach functions spaces associated to the optimization problem \eqref{eq:optipb}. More precisely, we shall introduce:
		\begin{itemize}
			\item The native space $\ML(\T^d)$: the generalized functions $f$ for which the regularization term $\lVert \Lop  f \rVert_{\mathcal{M}}$ is finite. 			\item The measurement space $\CL(\T^d)$: the generalized  functions $\nu$ that can be used as linear functionals over the native space $\ML(\T^d)$.
		\end{itemize}
		The measurement space and the native space form a dual pair, in the same way the space of periodic continuous functions $\mathcal{C}(\T^d)$ and the space of periodic Radon measures $\mathcal{M}(\T^d)$ do. We remind this fact and other useful ones in Section \ref{sec:CandM} before constructing the native and measurements spaces in Sections \ref{sec:nativespace} and \ref{sec:measurementspace}, respectively. 
		
	\subsection{The spaces $\mathcal{M}(\T^d)$ and $\mathcal{C}(\T^d)$} \label{sec:CandM}
		
		The space of continuous periodic function is denoted by $\mathcal{C}(\T^d)$.
		It is a Banach space when endowed with the supremum norm $\lVert  \varphi \rVert_{\infty} = \sup_{\bm{x}\in\T^d} \lvert \varphi(\bm{x}) \rvert$. 
		The space of periodic finite Radon measure is denoted by $\mathcal{M}(\T^d)$. According to the Riesz-Markov theorem~\cite{Gray1984shaping}, $\mathcal{M}(\T^d)$ is isometric to the space of continuous and linear functionals over $\mathcal{C}(\T^d)$. 
		As is classical, we therefore make a complete identification between the space of periodic finite Radon measures and the dual of $\mathcal{C}(\T^d)$, \emph{i.e.},
		\begin{equation} \label{eq:normM}
		\mathcal{M}(\T^d) = (\mathcal{C}(\T^d) , \lVert \cdot \rVert_\infty)'.
		\end{equation}
		Then, $\mathcal{M}(\T^d)$ is a Banach space for the \emph{total variation (TV)} norm 
		\begin{equation} \label{eq:Mnorm}
			\lVert w \rVert_{\mathcal{M}} = \sup_{\varphi \in \mathcal{C}(\T^d), \ \lVert \varphi \rVert_\infty = 1} \langle w , \varphi \rangle,	
		\end{equation}
		with $ \langle w ,\varphi \rangle = \int_{\T^d} \varphi (\bm{x}) w(\mathrm{d}\bm{x})$. When $w \in L_1(\T^d)$, we have that $\lVert w \rVert_{\mathcal{M}} = \lVert w \rVert_{1}$.
		For $w = \sum_{k=1}^K a_k \Sha(\cdot - \bm{x}_k)$ with distinct $\bm{x}_k \in \T$, we have $\lVert w \rVert_{\mathcal{M}} = \sum_{k=1}^K \lvert a_k\rvert = \lVert \bm{a}\rVert_1$. 	
		Since the Schwartz space $\Sp(\T^d)$ is dense in $\mathcal{C}(\T^d)$ for the norm $\lVert \cdot \rVert_\infty$ \cite[Theorem 4.25]{rudin2006real}, we can restrict the supremum 	
		over functions in the Schwartz space $\Sp(\T^d)$ in \eqref{eq:Mnorm}. This allows us to extend the total variation norm over $\Sp'(\T^d)$ and to deduce that
		\begin{equation}
			\mathcal{M}(\T^d) = \left\{ w \in \Sp'(\T^d), \  \lVert w \rVert_{\mathcal{M}}= \sup_{\varphi \in \mathcal{S}(\T^d), \ \lVert f \rVert_\infty = 1} \langle w , \varphi \rangle < \infty \right\}.
		\end{equation}
		Note that we have the continuous embeddings
		\begin{equation} \label{eq:easyembed}
			\Sp(\T^d) \subseteq \mathcal{C}(\T^d) \subseteq \Sp'(\T^d)  \quad \text{and}  \quad \Sp(\T^d) \subseteq \mathcal{M}(\T^d) \subseteq \Sp'(\T^d).
		\end{equation}
		In what follows, $\mathcal{M}(\T^d)$ will be endowed with the \emph{weak* topology} defined in Section \ref{sec:vocabulary}. The latter is indeed more convenient for our purposes than the Banach topology induced by the TV norm \eqref{eq:Mnorm}. 
		
	\subsection{The Native Space of a Spline-admissible Operator}		 \label{sec:nativespace}
		
		We define the native space on which the optimization problem \eqref{eq:optipb} is well-defined, and identify its structure.
		
		\begin{definition}[Native Space]
			The \emph{native space} associated to the spline-admissible operator $\Lop$ is defined as
			\begin{equation} \label{eq:ML}
				\mathcal{M}_{\Lop} (\T^d) := \left\{ f \in \Sp'(\T^d), \ \Lop f \in \mathcal{M}(\T^d) \right\}. 
			\end{equation}
		\end{definition}
		
		\begin{theorem}[Banach structure of the native space] \label{theo:whatisML}
		Let $\Lop$ be a spline-admissible operator with finite dimensional null space $\NL$ and pseudoinverse $\Kop$. 
		We also fix $p\in [1,\infty]$. 
		Then, $\ML(\T^d)$ is the direct sum
			\begin{equation}
			\ML (\T^d) = \Kop ( \mathcal{M} (\T^d) ) \oplus \NL,
			\end{equation}
			where $\Kop (\mathcal{M}(\T^d)) = \{\Kop w , \ w \in \mathcal{M}(\T^d) \}$. 	
			It is  a Banach space for the norm 
			\begin{equation} \label{eq:normML}
				\lVert f \rVert_{\ML,p} = \left(  \lVert \Lop f \rVert_{\mathcal{M}}^p  + \lVert \Proj_{\NL} f \rVert_{2}^p \right)^{1/p},
			\end{equation}
			with $p\in [1, +\infty]$ and the usual adaptation for $p = \infty$. 
			Moreover, we have the topological embeddings
			\begin{equation}\label{eq:embeddingML}
			\Sp(\T^d) \subseteq \ML(\T^d) \subseteq \Sp'(\T^d). 
			\end{equation}
			The operator $\Lop$ is continuous from $\ML(\T^d)$ to $\mathcal{M}(\T^d)$. Moreover, any periodic $\Lop$-spline is in $\ML(\T^d)$. 
			Finally, the norms $\lVert \cdot \rVert_{\ML,p}$ are all equivalent on $\ML(\T^d)$. 
		\end{theorem}
		
		\begin{proof}		
		\textit{Direct sum.}
		Let $f \in \ML(\T^d)$. Then, from \eqref{eq:LopKop}, and since $\Lop f \in \mathcal{M}(\T^d)$ by assumption,
		$$f = \Kop \{ \Lop f \} + \Proj_{\NL} f \in \Kop (\mathcal{M}(\T^d)) + \NL.$$
		Now, let $f = \Kop w + p \in \Kop (\mathcal{M}(\T^d)) + \NL$. Then, $\Lop f = \Lop \Kop w + \Lop p = w - \Proj_{\NL} w$, where we used \eqref{eq:LopKop} again. Since $\NL \subset \Sp(\T^d) \subset \mathcal{M}(\T^d)$, we deduce that $\Lop f = w - \Proj_{\NL} w \in \mathcal{M}(\T^d)$. This shows that $\ML (\T^d) = \Kop ( \mathcal{M} (\T^d) ) + \NL$. 
		
		If now $f \in \Kop ( \mathcal{M} (\T^d) ) \cap \NL$, then $\widehat{f} [\bm{k}] = 0$ for $\bm{k}\in \KL$ because $f = \Kop w$ for some $w$ and $\widehat{f}[\bm{k}] = 0$ for $\bm{k}\notin \KL$ because $\Lop f = 0$. Hence, $f=0$ and the sum $ \Kop ( \mathcal{M} (\T^d) ) \oplus \NL$ is direct. 
		
		\textit{Banach space structure.}		
			Clearly, \eqref{eq:normML} is a semi-norm on $\ML(\T^d)$. Moreover, the relation $\lVert f \rVert_{\ML,p} = 0$ implies that $\Lop f =  \Proj_{\NL} f = 0$, which is equivalent to $f=0$.
Finally, $\Kop ( \mathcal{M} (\T^d) )$ inherits the completeness of $\mathcal{M}(\T^d)$ and the finite-dimensional space $\NL$ is also a Banach space for the $L_2$-norm. Therefore, the direct sum $\ML(\T^d)$ is also a Banach space for the direct sum norm \eqref{eq:normML}. 

	\textit{Embedding relations.}	
	We remark that \eqref{eq:LopKop} implies that $\varphi = \Kop \{ \Lop \varphi \} + \Proj_{\NL} \varphi$ for any $\varphi \in \Sp(\T^d)$. Moreover, $\Lop \varphi \in \Sp(\T^d) \subseteq \mathcal{M}(\T^d)$ because $\Lop$ is continuous from $\Sp(\T^d)$ to itself (as any  LSI  operator continuous from $\Sp'(\T^d)$ to itself) and $\Proj_{\NL} \varphi \in \NL$. This shows that $\Sp(\T^d) \subset \ML (\T^d)$ (set inclusion). 
	The identity is moreover continuous from $\Sp(\T^d)$ to $\ML (\T^d)$. Indeed, if $\varphi_k$ converges to $\varphi$ in $\Sp(\T^d)$, then $\Lop \varphi_k$ converges to $\Lop \varphi$ in $\Sp(\T^d)$, hence in $\mathcal{M}(\T^d)$. Moreover, $\Proj_{\NL} \varphi_k$ also converges to $\Proj_{\NL} \varphi$ in $\Sp(\T^d)$, hence in $L_2(\T^d)$. We then have that $$\lVert \varphi_k - \varphi\rVert_{\ML,p}^p = \lVert \Lop \varphi_k - \Lop \varphi\rVert_{\mathcal{M}}^p + \lVert \Proj_{\NL} \varphi_k - \Proj_{\NL} \varphi \rVert_{2}^p \rightarrow 0,$$ as expected. This demonstrates the embedding $\Sp(\T^d) \subseteq  \ML(\T^d)$. 
		
	For the other embedding, we observe that $f = \Kop w + p \in \ML(\T^d) = \Kop ( \mathcal{M} (\T^d) ) \oplus \NL$ has the Fourier sequence $\widehat{f}[\bm{k}] = \widehat{L^\dagger}[\bm{k}]\widehat{w}[\bm{k}]+ \widehat{p}[\bm{k}]$, which is clearly slowly growing as the sum and product of slowly growing sequences, hence $\ML(\T^d) \subset \Sp'(\T^d)$ (set inclusion).
	Now, there exists some $\tau < 0$ such that $\mathcal{M}(\T^d) \subseteq \mathcal{H}_{2}^{\tau}(\T^d)$, where $\mathcal{H}_{2}^{\tau}(\T^d)$ is the Sobolev space of smoothness $\tau$. Then, $\Kop$ is continuous from $\Sp'(\T^d)$ to itself, and hence there exists $\tau'$ such that $\Kop : \mathcal{H}_{2}^\tau(\T^d) \rightarrow \mathcal{H}_{2}^{\tau'} (\T^d)$ continuously~\cite{Simon2003distributions}. By restriction, $\Kop$ is also continuous from $\mathcal{M}(\T^d)$ to $\mathcal{H}_{2}^{\tau'}(\T^d)$. Then, one has that $\ML(\T^d) \subseteq \mathcal{H}_{2}^{\tau'}(\T^d)$ due to the isometry property between $\mathcal{M}(\T^d)$ and $\ML(\T^d)$. Finally, since $\mathcal{H}_{2}^{\tau'}(\T^d) \subseteq \Sp'(\T^d)$, this concludes the proof. 
	
	\textit{Continuity of $\Lop$.}
	We simply remark that $\lVert \Lop f \rVert_{\mathcal{M}} \leq \lVert f \rVert_{\ML,p}$, implying the continuity of $\Lop$ from $\ML(\T^d)$ to $\mathcal{M}(\T^d)$.
	
	\textit{$\Lop$-splines are in $\ML(\T^d)$.} This simply follows from the fact that a $\Lop$-splines is such that $\Lop f = \sum_{k=1}^K a_k \Sha(\cdot - \bm{x}_k) \in \mathcal{M}(\T^d)$. 
	
	\textit{Norm equivalence.} The equivalence of the norms $\lVert \cdot \rVert_{\ML,p}$ for $p\in [1,\infty]$ simply follows from the equivalence of the $\ell_p$-norms over $\R^2$.
			\end{proof}
	
	\subsection{The Measurement Space of a Spline-admissible Operator} \label{sec:measurementspace}
		
		A linear functional $\nu$ can be used for the linear measurements in \eqref{eq:optipb} under the condition that $\nu(f)$ is well-defined for any $f \in \ML(\T^d)$. Moreover, as we shall see in Section \ref{sec:RT}, to ensure the existence of extreme point solutions of \eqref{eq:optipb}, $\nu$ should be in the dual of $\ML(\T^d)$ when the latter is equipped qith the weak* topology. In this section, we introduce the measurement space $\CL(\T^d)$, specify its Banach space structure and show its relation with the native space $\ML(\T^d)$. 
		We recall that the adjoint of $\Lop  \in  \mathcal{L}_{\mathrm{SI}} (\Sp'(\T^d)) $ is the unique operator $\Lop^*  \in  \mathcal{L}_{\mathrm{SI}} (\Sp'(\T^d))$ such that 
		\begin{equation}
			\langle \Lop \varphi_1, \varphi_2 \rangle = \langle \varphi_1 , \Lop^* \varphi_2\rangle
		\end{equation}
		for every $\varphi_1, \varphi_2 \in \Sp (\T^d)$. The Fourier sequence of $\Lop^*$ is given by $\widehat{L^*}[\bm{k}] = \overline{\widehat{L}[\bm{k}]}$ for $\bm{k}\in \T^d$ and the operators $\Lop$ and  $\Lop^*$ share the same null space $\NL$.
		
		\begin{definition}[Measurement Space]
		We define the \emph{measurement space} associated to the spline-admissible operator $\Lop$ as 
		\begin{equation}\label{eq:CL}
			\CL(\T^d) = \{ g \in \Sp'(\T^d) , \ \Kop^* g \in \mathcal{C}(\T^d) \},
		\end{equation}
		where $\Kop^*$ is the adjoint of $\Kop$.
		\end{definition}
		
		\begin{theorem}[Banach structure of the measurement space] \label{theo:whatisCL}
		Consider a spline-admissible operator $\Lop$ with finite-dimensional null space $\NL$. We also fix $q\in [1,\infty]$. 
		Then, $\CL(\T^d)$ is the direct sum
		\begin{equation}\label{eq:CL}
			\CL(\T^d) = \Lop^* ( \mathcal{C} (\T^d) )  \oplus \NL.
		\end{equation}
		The measurement space $\CL(\T^d)$ is a Banach space for the norm
		\begin{equation} \label{eq:normCL}
			\lVert g \rVert_{\CL,q} =  \left( \lVert \Kop^* g \rVert_{\infty}^q + \lVert \Proj_{\NL} g \rVert_{2}^q \right)^{1/q},
		\end{equation}
		with $q\in [1, +\infty]$ the usual adaptation when $q=\infty$. 
		Moreover, we have the topological embeddings
			\begin{equation}\label{eq:embeddingCL}
			\Sp(\T^d) \subseteq \CL(\T^d) \subseteq \Sp'(\T^d),
			\end{equation}
			and the space $\Sp(\T^d)$ is dense in $\CL(\T^d)$.	
		Finally, the norms $\lVert \cdot \rVert_{\CL,q}$ are all equivalent on $\CL(\T^d)$ for $1 \leq q \leq \infty$. 
	
		\end{theorem}

		\begin{proof}
		The direct sum \eqref{eq:CL}, the Banach structure with norm \eqref{eq:normCL} (remembering that $\lVert \cdot \rVert_\infty$ is a Banach norm on $\mathcal{C}(\T^d)$), and the embeddings \eqref{eq:embeddingCL} are obtained with identical arguments than in Theorem \ref{theo:whatisML}. 
		
		The only remaining part is the denseness of $\Sp(\T^d)$ in $\CL(\T^d)$. Let $g  = \Lop^* h + p \in \CL(\T^d)$, with $h \in \mathcal{C}(\T^d)$ and $p \in \NL$. The space $\Sp(\T^d)$ is dense in $ \mathcal{C}(\T^d)$, hence there exists a sequence $(\varphi_k)_{n\geq 1}$ of functions in $\Sp(\T^d)$ such that $\lVert h - \varphi_k \rVert_\infty \rightarrow 0$ when $k\rightarrow \infty$. We set $\psi_k = \Lop^* \varphi_k + p$. Then, $p \in \NL \subset \Sp(\T^d)$ and $\Lop^*$, as for any operator in $\mathcal{L}_{\mathrm{SI}} (\Sp'(\T^d))$,  is continuous from $\Sp(\T^d)$ to itself. We therefore have that $\psi_k \in \Sp(\T^d)$ for any $k \geq 1$. 
		Moreover, $g - \psi_k = \Lop^* \{ h - \varphi_k\} \in \CL(\T^d)$ and we have that
		\begin{equation} \label{eq:gminuspsi} 
		\lVert g - \psi_k \rVert_{\CL, q} = \lVert \Kop^* \Lop^* (h-\varphi_k) \rVert_\infty =  \lVert (\mathrm{I} - \mathrm{Proj}_{\NL})^*  (h-\varphi_k) \rVert_\infty = \lVert (\mathrm{I} - \mathrm{Proj}_{\NL})  (h-\varphi_k) \rVert_\infty ,
		\end{equation}
		where we used \eqref{eq:LopKop} and $\mathrm{Proj}_{\NL} = \mathrm{Proj}_{\NL}^*$. 
		If $f \in \mathcal{C}(\T^d)$, then for any $\bm{k}\in \Z^d$, $\lvert \widehat{f} [\bm{k}] \rvert \leq \lVert f \rVert_\infty$. Applied to $f = h - \varphi_k$ and denoting by $\bm{k}_1, \ldots , \bm{k}_{N_0}$ the frequencies of the finite-dimensional null space of $\Lop$ (see Proposition~\ref{prop:finitedimNL}), we deduce from \eqref{eq:gminuspsi} that
		\begin{equation}
		\lVert g - \psi_k \rVert_{\CL,q} \leq \lVert h - \varphi_k \rVert_\infty + \sum_{n=1}^{N_0} \lvert \widehat{(h -\varphi_k)}[\bm{k}_n]\rvert \leq (1+N_0)  \lVert h - \varphi_k \rVert_\infty.
		\end{equation}
		Since $ \lVert h - \varphi_k \rVert_\infty \rightarrow 0$, we deduce that $\lVert g - \psi_k \rVert_{\CL,q}$ vanishes and the denseness is proved. Finally, the equivalence of the norms $\lVert \cdot \rVert_{\CL,q}$ for $q\in [1,\infty]$ simply follows from the equivalence of the $\ell_q$-norms over $\R^2$.
		\end{proof}
		
		\begin{theorem}[Generalized Riesz-Markov Representation Theorem] 
		\label{theo:RieszMarkovgeneralized}
		Let $\Lop$ be a spline-admissible operator and $1 \leq p, q \leq \infty$ such that $1/ p + 1/q = 1$. 
		The topological dual of the Banach space $( \CL(\T^d) , \lVert \cdot \rVert_{\CL,q} )$  is isometric to the native space $(\ML(\T^d) ,\lVert \cdot \rVert_{\ML,p} ))$; that is,
		\begin{equation} \label{eq:dualidentification}
			(\CL'(\T^d), \lVert \cdot \rVert_{\CL',q}) = (\ML(\T^d), \lVert \cdot \rVert_{\ML,p}).
		\end{equation} 		
		Moreover, the topological dual of $\ML(\T^d)$ endowed with the weak* topology inherited from $\CL(\T^d)$ is $\CL(\T^d)$ itself. Hence, $(\CL(\T^d),\ML(\T^d))$ is a dual pair. 
		\end{theorem}

		The relation \eqref{eq:dualidentification} is a representation theorem, since it identifies the topological dual of the measurement space as being the native space $\ML(\T^d)$. It is therefore a generalization of  the Riesz-Markov representation theorem, stating that $(\mathcal{C}'(\T^d) , \lVert \cdot \rVert_{\mathcal{C}'} ) = (\mathcal{M}(\T^d), \lVert \cdot \rVert_{\mathcal{M}})$ (isometric identification). The proof of Theorem \ref{theo:RieszMarkovgeneralized} is based on the following lemma, which recalls elementary results on Banach spaces.
		
		\begin{lemma} \label{lemma:XYZ}
			Let $\mathcal{X}, \mathcal{Y}, \mathcal{Z}$ be three Banach spaces with norm $\lVert \cdot \rVert_{\mathcal{X}}, \lVert \cdot \rVert_{\mathcal{Y}},\lVert \cdot \rVert_{\mathcal{Z}}$, respectively. Their topological duals $\mathcal{X}', \mathcal{Y}', \mathcal{Z}'$ are Banach spaces for their dual norms, denoted respectively by $\lVert \cdot \rVert_{\mathcal{X}'}, \lVert \cdot \rVert_{\mathcal{Y}'},\lVert \cdot \rVert_{\mathcal{Z}'}$.
			Then, the following statements hold.
			\begin{itemize}
				\item For any $p \in [1,\infty]$, The product space $\mathcal{X}\times \mathcal{Y}$ is a Banach space for the norm, defined for any $(x,y) \in \mathcal{X}\times \mathcal{Y}$
				\begin{equation} 
					\lVert (x,y) \rVert_{\mathcal{X}\times \mathcal{Y}} = \left( \lVert x \rVert_{\mathcal{X}}^p + \lVert y \rVert_{\mathcal{Y}}^p\right)^{1/p},
				\end{equation}
				with the usual adaptation for $p = \infty$.
				\item The topological dual $(\mathcal{X}\times \mathcal{Y})'$ of $\mathcal{X}\times \mathcal{Y}$ associated with the dual norm is isometric to the Banach space $\mathcal{X}' \times \mathcal{Y}'$ endowed with the norm defined for $(x',y') \in \mathcal{X}' \times \mathcal{Y}'$ by
				\begin{equation} \label{eq:dualproduct}
					\lVert (x',y') \rVert_{\mathcal{X}'\times \mathcal{Y}'} = \left( \lVert x' \rVert_{\mathcal{X}'}^q + \lVert y' \rVert_{\mathcal{Y}'}^q\right)^{1/q},
				\end{equation}
				where $q \in [1,\infty]$ is the conjugate of $p$ satisfying $\frac{1}{p} + \frac{1}{q} = 1$, with the usual adaptation for $q=\infty$ (\emph{i.e.}, $p=1$) in \eqref{eq:dualproduct}.
				\item Assume that $\Phi : \mathcal{X} \rightarrow \mathcal{Z}$ is an isometry. Then, the adjoint $\Phi^*$ of $\Phi$ is an isometry between $\mathcal{Z}'$ and $\mathcal{X}'$ endowed with their dual norms, and we have
				\begin{equation} \label{eq:ZXPhi}
					\mathcal{Z}' = (\Phi^*)^{-1} \mathcal{X}'. 
				\end{equation}
			\end{itemize}
		\end{lemma}
		
		Lemma \ref{lemma:XYZ} is elementary and let to the reader. Note that the relation \eqref{eq:dualproduct} simply uses that the topological dual of $(\R^2, \lVert \cdot \rVert_p)$ is $(\R^2, \lVert \cdot \rVert_q)$. 
		
		
		\begin{proof}[Proof of Theorem \ref{theo:RieszMarkovgeneralized}]

		We recall that the projector $\mathrm{Proj}_{\NL}$ is defined in \eqref{projdef}, and set $\mathrm{Proj}_{\NL^\perp} = \mathrm{Id} - \mathrm{Proj}_{\NL}$. Then, the spaces $\mathrm{Proj}_{\NL^\perp} (\mathcal{C}(\T^d))$ and $\mathrm{Proj}_{\NL^\perp} (\mathcal{M}(\T^d))$
inherit the Banach space structures of $\mathcal{C}(\T^d)$ and $\mathcal{M}(\T^d)$ for the restriction of the norms $\lVert \cdot \rVert_\infty$ and $\lVert \cdot \rVert_{\mathcal{M}}$, respectively. 
	According to the Riesz-Markov theorem, the space $\mathcal{M}(\T^d)$ is isometric to the topological dual of $(\mathcal{C}(\T^d), \lVert \cdot \rVert_\infty)$ endowed with the dual norm, and this property is transmit to the projections of those spaces,  which implies the isometric identification 
	\begin{equation}\label{eq:dualonprojnlprop}
	( (\mathrm{Proj}_{\NL^\perp} (\mathcal{C}(\T^d)))' = \mathrm{Proj}_{\NL^\perp} (\mathcal{M}(\T^d)).
	\end{equation}
	Due to Theorems  \ref{theo:whatisML} and  \ref{theo:whatisCL}, we have that
	\begin{align}
		\CL(\T^d) &=  \Lop^* ( \mathrm{Proj}_{\NL^\perp} (\mathcal{C} (\T^d) ) )  \oplus \NL \quad \text{ and }  \quad 
		\ML(\T^d) =  \Lop^\dagger ( \mathrm{Proj}_{\NL^\perp} (\mathcal{M} (\T^d) ) )  \oplus \NL.		
	\end{align}
	Moreover, the operator $\Lop^*$ is an isometry between $\mathrm{Proj}_{\NL^\perp} (\mathcal{C} (\T^d) )$ and $\mathrm{Proj}_{\NL^\perp} (\CL(\T^d) )$, and its inverse's adjoint is $  \Kop$. 
	This can be easily seen from the definition of the norm \eqref{eq:normCL} and the fact that $\mathrm{Proj}_{\NL^\perp}$ simply set to zero the Fourier coefficients associated to the finitely many null space frequencies. Thanks to \eqref{eq:ZXPhi} in Lemma \ref{lemma:XYZ}, we therefore deduce the isometric identification $\left( \Lop^* ( \mathrm{Proj}_{\NL^\perp} (\mathcal{C} (\T^d) ) ) \right)' =  \Lop^\dagger ( \mathrm{Proj}_{\NL^\perp} (\mathcal{M} (\T^d) ) )$. We observe moreover that $ \Lop^*   \mathrm{Proj}_{\NL^\perp} = \Lop^*$ and $ \Lop^\dagger  \mathrm{Proj}_{\NL^\perp} = \Kop$, implying that 
	\begin{equation}
		\left( \Lop^*    (\mathcal{C} (\T^d)  ) \right)' =  \Lop^\dagger (  \mathcal{M} (\T^d)  ).
	\end{equation}
	The finite-dimensional space $\NL$ has an orthonormal basis for the $L_2$ scalar product, given by $\{e_{\bm{k}_n}, n=1 \ldots N_0\}$, with $\bm{k}_1, \ldots, \bm{k}_{N_0}$ the null space frequencies of $\Lop$. Hence, $\NL'$ is isometrically identified to $\NL$. Applying \eqref{eq:dualproduct}, we deduce \eqref{eq:dualidentification}. The second statement of Theorem \ref{theo:RieszMarkovgeneralized} then follows directly follows.
		\end{proof}
		
%
		
	\textbf{Running example $\Lop = \Dop^N$, $N\geq 1$.} 
	The native space of $\Dop^N$ is the space of functions $f$ such that $\Dop^N f \in \mathcal{M}(\T^d)$. According to Theorem \ref{theo:whatisML}, a function of this space can be written as 
$f = (\Dop^N)^\dagger w + \alpha $ 	with $w \in \mathcal{M}(\T^d)$ and $\alpha \in \R$.
	The norm of $f$ is then $\lVert f \rVert_{\mathcal{M}_{\Dop^N},p} = \left( \lVert w \rVert_{\mathcal{M}}^p + \alpha^p \right)^{1/p}$. 
	
	A function $g$ of the measurement space $\mathcal{C}_{\Dop^N}(\T^d)$ is of the form $g = \Dop^N \{ h \} + \beta$ with $h \in \mathcal{C}(\T^d)$ and $\beta \in \R$. Note that the adjoint of $\Dop^N$ is $(-1)^N \Dop^N$ but the constant $(-1)^N$ can be absorbed in $h$. The norm of $g$ is then $\lVert g \rVert_{\mathcal{C}_{\Dop^N, q}} = \left( \lVert \varphi \rVert_{\infty}^q + \beta^q \right)^{1/q}$.

	\section{Periodic Representer Theorem} \label{sec:RT}
			
		Assume that we want to reconstruct an unknown periodic function $f_0$ from its possibly noisy linear measurements $\bm{y} \approx  \bm{\nu}(f_0)\in\R^M$. Typically, $\bm{y}$ is a random perturbation of $\bm{\nu}(f_0)$ such that $\mathbb{E} [\bm{y} ] = \bm{\nu} ( f_0 )$. This mild assumption allows to consider additive noise models, but also more general ones. We shall not discuss further the model for the data acquisition on this paper and we refer \cite[Chapter 7.5]{simeoni2020functional} for more details on this topic.
		To achieve our reconstruction goal, we consider the penalized optimization problem:
		\begin{equation} \label{eq:theoptibeforeRT}
		\tilde{f}\in{\arg \min_f}  \quad  E(\bm{y}, \bm{\nu} (f ) )  +   \lambda \lVert \Lop f \rVert_{\mathcal{M}}  ,
		\end{equation}
		where $E(\bm{y}, \bm{\nu} (f ) )$ is a data-fidelity term that enforces the fidelity of the predicted measurements $\bm{\nu} (f )$ to the observed one $\bm{y}$. A typical choice for $E$  is the quadratic cost function $E(\bm{y}, \bm{\nu} (f ) ) = \lVert \bm{y} - \bm{\nu}(f) \rVert^2$.
		The regularization $ \lVert \Lop f \rVert_{\mathcal{M}}$ promotes functions with certain smoothness properties.
		The spline-admissible operator $\Lop$ typically characterizes the smoothness of the reconstruction from \eqref{eq:theoptibeforeRT}, while the choice the $\mathcal{M}$-norm is know to promote sparse reconstructions~\cite{Unser2017splines,gupta2018continuous}. 
		An optimizer $\tilde{f}$ of \eqref{eq:theoptibeforeRT} is expected to adequately approximate the function $f_0$.

		In Section \ref{sec:constructionNativeSpace}, we have introduced the function spaces required to give a clear meaning to \eqref{eq:theoptibeforeRT}. The functions $f$ should typically live in the native space $\ML(\T^d)$, for which $\lVert \Lop f \rVert_{\mathcal{M}} < \infty$, while the measurement functionals $\nu_m$ are taken in  the measurement space $\CL(\T^d)$. 
		
		\begin{theorem}[Periodic Representer Theorem] \label{theo:RT}
			Consider the following assumptions: 
			\begin{itemize}
				\item a spline-admissible operator $\Lop$ with null space $\NL$ of finite dimension $N_0 \geq 0$ and pseudoinverse $\Kop$;
				\item a linearly independent family of $M \geq N_0$ linear functionals $\nu_m \in \CL(\T^d)$ such that $\bm{\nu}= (\nu_1,\ldots , \nu_M)$ is injective on the null space of $\Lop$; that is,  such that the condition $\langle \bm{\nu}, p \rangle = \bm{0}$ for  $p\in \NL$ implies that $p = 0$;
				\item a cost function $E(\cdot, \cdot) : \R^M\times\R^M \rightarrow \R^+ \cup \{\infty\}$ such that $E(\bm{z}, \cdot)$ is a  lower semi-continuous,  convex, coercive, and proper function for any fixed $\bm{z} \in \R^M$;
				\item a measurement vector $\bm{y}\in \R^M$; and
				\item a tuning parameter $\lambda > 0$.
			\end{itemize}
			Then, the set of minimisers
			\begin{equation}\label{eq:optiwellstated}
			\mathcal{V} := \underset{f \in \ML(\T^d) }{\arg \min}  E(\bm{y}, \bm{\nu} (f ) )  +   \lambda \lVert \Lop f \rVert_{\mathcal{M}}  
			\end{equation}
			is non empty, convex, and compact with respect to the weak* topolgy on $\ML(\T^d)$.
			Moreover, the extreme points of  \eqref{eq:optiwellstated} are  periodic $\Lop$-splines of the form
			\begin{equation} \label{eq:LopfRT}
				  f_{\mathrm{ext}} = \sum_{k=1}^K a_k \Kop \{ \Sha\} ( \cdot -\bm{x}_k) + p = \sum_{k=1}^K a_k g_{\Lop} ( \cdot -\bm{x}_k) + p
			\end{equation}
			for some $\bm{x}_k \in \T^d$, $a_k \in \R$ with $\bm{\mathrm{M}} \bm{a} = \bm{0}$ where $\bm{\mathrm{M}}$ is given by \eqref{eq:matrixM},  $K \leq M$ knots, and $p \in \NL$.
		\end{theorem}

	The periodic representer theorem reveals the form of the solutions of the optimization problem \eqref{eq:optiwellstated} in the following sense: (i) The extreme point solutions are periodic $\Lop$-splines with at most $M$ knots, with $M$ the number of measurements. (ii) The other spline solutions are finite convex combinations of the extreme points. (iii) Any solution is the (weak*) limit of spline solutions.
	Note that an extreme point solution is such that $\Lop  f_{\mathrm{ext}} = \sum_{k=1}^K a_k \Sha( \cdot - \bm{x}_k) $
		is a finite sum of Dirac combs. \\
		
	\textbf{Running example $\Lop = \Dop^N$, $N\geq 1$.}  Theorem \ref{theo:RT} can be applied to the periodic spline-admissible operator $\Dop^N$. The second condition is then equivalent to the existence of $1 \leq m \leq M$ such that $\langle \nu_m , 1 \rangle = \widehat{\nu}_m [0] \neq 0$, or equivalently, the existence of a linear functional with nonzero mean, which is a mild requirement.
	Then, under the conditions of Theorem \ref{theo:RT}, the extreme points of the solution set $\mathcal{V}$ are periodic $\Dop^N$-splines. 
	 	\\

	The seminal work of Fisher and Jerome~\cite{Fisher1975} was developed over compacts domains, and we will see that Theorem \ref{theo:RT} can be deduced from their main result, that we recall here with our notation. We present here a slight adaptation of \cite[Theorem 1]{Fisher1975}.
		
		We fix a set of distinct Fourier frequencies $\{\bm{k}_n, \,n=1,\ldots , N_0\}$ and introduce $\mathcal{N} = \mathrm{Span}\{e_{\bm{k}_n}, n=1,\ldots , N_0\}$ which is a space of dimension $N_0$. We define the orthogonal projector $\mathrm{Proj}_{\mathcal{N}}$ over $\mathcal{N}$ as in \eqref{projdef}, and set $\mathrm{Proj}_{\mathcal{N}^\perp} = \mathrm{Id} - \mathrm{Proj}_{\mathcal{N}}$. Then, the spaces $\mathrm{Proj}_{\mathcal{N}^\perp} (\mathcal{C}(\T^d))$ and $\mathrm{Proj}_{\mathcal{N}^\perp} (\mathcal{M}(\T^d))$
inherit the Banach space structures of $\mathcal{C}(\T^d)$ and $\mathcal{M}(\T^d)$ for the restriction of the norms $\lVert \cdot \rVert_\infty$ and $\lVert \cdot \rVert_{\mathcal{M}}$, respectively. Moreover, as we have seen in the proof of Theorem \ref{theo:RieszMarkovgeneralized}, $(\mathrm{Proj}_{\mathcal{N}^\perp} (\mathcal{C}(\T^d)))' = \mathrm{Proj}_{\mathcal{N}^\perp} (\mathcal{M}(\T^d))$.
		
		\begin{lemma}[Fisher-Jerome Theorem] \label{theo:FJ}
		Let $\mathcal{N} = \mathrm{Span}\{e_{\bm{k}_n}, n=1,\ldots , N_0\}$as above and $M \geq N_0$. 
		Let $(f_m,q_m) \in \mathrm{Proj}_{\mathcal{N}^\perp} (\mathcal{C}(\T^d)) \times \mathcal{N}$, $m=1,\ldots , M$, be a set of linearly independent couples of functions. We assume moreover that  $\bm{q}(p) = (\langle q_1, p \rangle , \ldots , \langle q_M, p \rangle ) = \bm{0}$ if and only if $p=0$ for $p\in \mathcal{N}$. 
		Let $\bm{z}_0 \in \R^M$ be such that there exists $(w, p) \in \mathrm{Proj}_{\mathcal{N}^\perp} (\mathcal{M}(\T^d))\times \mathcal{N}$ with 
		\begin{equation} \label{eq:conditionwp}
			\bm{\mu}((w,p)) := (\langle f_1 , w \rangle + \langle q_1, p \rangle , \ldots , \langle f_M , w \rangle + \langle q_M, p \rangle ) = \bm{z}_0.
		\end{equation}
		Then, the set of minimizers 
		\begin{equation} \label{eq:theFJadapted}
			\underset{\bm{\mu}(w,p) = \bm{z}_0}{\arg \min} \lVert w \rVert_{\mathcal{M}}
		\end{equation}	
		is non empty, convex, weak* compact in $\mathrm{Proj}_{\mathcal{N}^\perp} (\mathcal{M}(\T^d)) \times  \mathcal{N}$,  and its extreme points are of the form
		\begin{equation} \label{eq:extremewp}
			(w_{\mathrm{ext}}, p_{\mathrm{ext}}) = \left( \sum_{k=1}^K a_k \Sha( \cdot - \bm{x}_k) , p_{\mathrm{ext}}\right),
		\end{equation}
		where $p_{\mathrm{ext}} \in \mathcal{N}$, $a_k \in \R$ with $\bm{\mathrm{M}}\bm{a} = \bm{0}$ where $\bm{\mathrm{M}}$ is defined in \eqref{eq:matrixM}, $\bm{x}_k \in \T^d$, and $K \leq M$.	
		\end{lemma}		

		We call condition \eqref{eq:conditionwp} the \emph{feasibility assumption}, it means that one can achieve the measure $\bm{z}_0$ and is obviously needed so that \eqref{eq:theFJadapted} has a solution.  Lemma \ref{theo:FJ} is an adaptation of \cite[Theorem 1]{Fisher1975}, where the authors work with a compact metric space $X$ that we specialize with $\T^d$, and where we simply work with $\mathrm{Proj}_{\mathcal{N}^\perp} (\mathcal{C}(\T^d))$ instead of $\mathcal{C}(X)$. 
	The proof is identical. For a more recent treatment, we refer the reader to \cite[Theorem 7]{Unser2017splines} which considers the real line $\R^d$.
	Note that the relation $\bm{\mathrm{M}}\bm{a} = \bm{0}$ comes from the fact that $w_{\mathrm{ext}} \in \mathrm{Proj}_{\mathcal{N}^\perp} (\mathcal{M}(\T^d))$ (see Proposition \ref{prop:Mnkmatrix}).
					
		\begin{proof}[Proof of Theorem \ref{theo:RT}]
		The proof has two parts. We first prove that the solution set $\mathcal{V}$ in \eqref{eq:optiwellstated} is non empty, convex, and weak* compact, and then obtain the form of the extreme points using Lemma \ref{theo:FJ}. \\
		
		\textit{Properties of $\mathcal{V}$.}
		The first part of the proof is classical, we briefly mention the key steps since it follows exactly the line of the one of \cite[Theorem 4]{gupta2018continuous}.
		For $f \in \ML(\T^d)$, we define $J(f) := E(\bm{y}, \bm{\nu} (f ) )  +   \lambda \lVert \Lop f \rVert_{\mathcal{M}}$. The functional $J : \ML(\T^d) \rightarrow \R^+ \cap \{\infty\}$ is proper,  convex, coercive, and weak*-lower semi-continuous (see \cite[Appendix B]{gupta2018continuous} for a full proof). 
		Then, we are exactly in the conditions of \cite[Proposition 8]{gupta2018continuous}, revealing that $\mathcal{V} = \arg\min J$ is effectively non empty, convex, and weak* compact. 
		As such, from the Krein-Milman theorem \cite[p. 75]{RudinFA}, the set $\mathcal{V}$ admits extreme points and is the weak* closure of those extreme points. \\

		\textit{Form of the extreme points.}
		We fix an extreme point solution $f^* \in \mathcal{V}$, set $\bm{z}_0 := \bm{\nu}(f^*) \in \R^M$, and consider the optimization problem 
		\begin{equation} \label{eq:newpb}
			\tilde{\mathcal{V}}  := \underset{f \in \ML(\T^d),  \ \bm{\nu}(f) = \bm{z}_0 }{\arg \min}   \lVert \Lop f \rVert_{\mathcal{M}}.
		\end{equation}	
		Clearly, we have that $\tilde{\mathcal{V}} \subset \mathcal{V}$ since $\lVert \Lop g^* \rVert_{\mathcal{M}} = \lVert \Lop f^* \rVert_{\mathcal{M}}$ and $\bm{\nu}(g^*) = \bm{\nu}(f^*)$ for any $g^* \in \mathcal{V}_{f^*}$. Moreover, $f^*$ is an extreme point of $\mathcal{V}_{f^*}$ (being an extreme point of the bigger set $\mathcal{V}$). We can therefore focus on the optimization problem  \eqref{eq:newpb} and show that its extreme points have the expected representation. 
		
		With Theorem \ref{theo:whatisML}, we know that any $f \in \ML(\T^d)$ has a unique representation as $f = \Kop w + p$ with $w \in \mathrm{Proj}_{\NL^\perp} (\mathcal{M}(\T^d))$ and $p \in \NL$. In particular, we have the equivalence
		\begin{equation} \label{eq:equivalentproblems}
		f^* \in \tilde{\mathcal{V}} \Longleftrightarrow (w^* , p^*) \in \mathcal{W} :=  \underset{(w,p) \in\mathrm{Proj}_{\NL^\perp} (\mathcal{M}(\T^d))\times\NL, \      \bm{\nu}(\Kop w + p) = \bm{z}_0}{\arg \min} \lVert w \rVert_{\mathcal{M}},
		\end{equation}	
		where $f^* = \Kop w^* + p^*$, $w^* \in \mathrm{Proj}_{\NL^\perp} (\mathcal{M}(\T^d))$ and $p^* \in \NL$. The equivalence also holds for the extreme points of both problems: $f_{\mathrm{ext}}$ is an extreme point of $\tilde{\mathcal{V}}$ if and only if $(w_{\mathrm{ext}}, p_{\mathrm{ext}})$ is an extreme point of $\mathcal{W}$, with $f_{\mathrm{ext}} = \Kop w_{\mathrm{ext}} + p_{\mathrm{ext}}$.
		
		We observe that, for any $f = \Kop w + p$ with $(w,p) \in \mathrm{Proj}_{\NL^\perp} (\mathcal{M}(\T^d)) \times \NL$, we have
		\begin{equation}
			\nu_m(f) = \langle \nu_m , \Kop w \rangle + \langle \nu_m , p \rangle 
			= \langle (\Kop)^* \nu_m , w \rangle + \langle \mathrm{Proj}_{\NL} \nu_m , p \rangle.
		\end{equation}
		Then, $f_m := (\Kop)^* \nu_m \in \mathrm{Proj}_{\NL}(\mathcal{C}(\T^d))$ and $\mathrm{Proj}_{\NL} \nu_m \in \NL$. We define $\bm{\mu}$ as in \eqref{eq:conditionwp}.
		
		We are then in the conditions of Lemma \ref{theo:FJ} with $\mathcal{N} = \NL$ for those functionals. 
		Indeed, the condition $\bm{q} (p) = \bm{0}$ if and only if $p=0$ for $p\in \NL$ comes from the second assumption in Theorem \ref{theo:RT}. 
		Note that the feasibility condition is satisfied because $\tilde{\mathcal{V}}$,  hence $\mathcal{W}$, are non empty.
		We deduce that $\tilde{\mathcal{V}}$ inherits the properties of $\mathcal{W}$, and is therefore convex and weak* compact in $\ML(\T^d)$. 
		Moreover, the extreme points are such that 
		$f_{\mathrm{ext}} = \Kop w_{\mathrm{ext}} + p_{\mathrm{ext}}$, where $(w_{\mathrm{ext}}, p_{\mathrm{ext}})$ are given by \eqref{eq:extremewp}.
		This shows that $f_{\mathrm{ext}}$ has the expected form, the relation $\bm{\mathrm{M}} \bm{a} = \bm{0}$ coming from the condition over $\bm{a}$ in \eqref{eq:extremewp} and Theorem \ref{theo:RT} is proved. 
		\end{proof}
		
	\section{Pseudo-Differential Periodic Spline-Admissible Operators} \label{sec:differentialop}
	
	The goal of this section is to provide examples of periodic spline-admissible operators and to highlight their main properties.
	The considered operators are pseudo-differential: they have a \emph{roughening} behaviour---\emph{i.e.}, they reduce the smoothness of the input function. This roughening behaviour is quantified by the spectral growth $\gamma > 0$ (see Definition \ref{def:growth}). 
	We include both univariate (ambiant dimension $d=1$, Section \ref{sec:univariate}) and multivariate ($d \geq 1$, Section \ref{sec:multivariate}) spline-admissible operators. Python routines for efficiently generating and manipulating the various multivariate periodic $\Lop$-splines considered in this section are available on the public GitHub repository \cite{periodispline2020}.
	 The code is compatible with any ambiant dimension $d \geq 1$, and only requires to specify the knots and the amplitudes of the splines. We only use it to represent periodic spines with minimum number of knots (one or two, depending on the operators) but it can be used for any number of knots. 
	Our periodic spline generator leverages truncated Fourier series expansions, implemented efficiently via FFTs using the routines from the GitHub repository  \cite{pyffs}. We moreover make use of F\'ejer kernels for a fast convergence of the truncated Fourier series to the spline function.
	
	\subsection{Univariate Splines-admissible Operators} \label{sec:univariate}
	
	In ambiant dimension $d=1$, we 
	consider classical differential operators and their fractional versions. 
	Table \ref{table:operators1d} provides the list of the considered univariate spline-admissible operators, together with their Fourier sequence and their null space via its set $N_\Lop$ of Fourier frequencies. We recall that $\NL = \mathrm{Span} \{ e_k, \  k \in N_\Lop \}$. 
		
	\begin{table*}[h!] 
\centering
\caption{Families of univariate spline-admissible operators}

\begin{tabular}{ccccccc} 
\hline
\hline\\ 
Spline's type & Operator & Parameter & $\widehat{L}[k]$ & Spectral growth & $N_\Lop$ 
\\
\hline\\[-1ex]
Polynomial splines & $\Dop^N$ & $N\in \N$ & $ (  \mathrm{i} k)^N $ & $N$ & $0$    \\
Exponential splines    & $\Dop +  \alpha \Id$ & $\alpha \notin \mathrm{i} \Z$ &  $  \mathrm{i} k + \alpha$ & $1$  & $\emptyset$     \\
 & $\Dop - \mathrm{i}  k_0 \Id$ & $k_0 \in \Z$ &  $ \mathrm{i} (k - k_0)$ & $1$   &  $k_0$      \\
& $\Dop^2 + k_0^2  \Id$ & $k_0 \in \Z$ &  $k_0^2 - k^2$ & $2$  & $k_0, -k_0$  &    \\
Fractional splines & $\Dop^\gamma$ & $\gamma > 0$ &  $( \mathrm{i} k)^\gamma$ & $\gamma$  & $0$   \\
Fractional \\
 exponential splines & $(\Dop + \alpha \mathrm{Id})^\gamma$ & $\gamma > 0$, $\alpha \in \R$ &  $( \mathrm{i} k + \alpha)^\gamma$   & $\gamma$ & $\emptyset$   \\
Fractional \\
polyharmonic splines & $(-\Delta)^{\gamma/2}$ & $\gamma > 0$ & $\lvert k \rvert^\gamma$ & $\gamma$  & $0$ \\
Sobolev splines & $ (\alpha^2 \Id - \Delta)^{\gamma/2}$ & $\alpha \neq 0$, $\gamma \in \R$ &  $(\alpha^2 +  k^2)^{\gamma/2}$ & $\gamma$  & $\emptyset$    \\
Mat\'ern splines & $\mathrm{M}_\epsilon^\beta$ & $\epsilon > 0$ &  see Definition \ref{def:matern} & $2(\beta - 1/2) $  & $\emptyset$   \\
& &  $\beta \in \mathbb{N}_{\geq 1} + 1/2$ & & & \\
Wendland splines & $\mathrm{W}_{\epsilon,\mu}^\beta$ & $\epsilon > 0$, $\nu \in \mathbb{N}$ &  see Definition \ref{def:wendland}   & $2(\beta - 1/2)$  & $\emptyset$   \\
& &  $\beta \in \mathbb{N}_{\geq 2} $ & & & \\
\hline
\hline
\end{tabular} \label{table:operators1d}
\end{table*}	

	\paragraph{Periodic Polynomial Splines.} 
	We already considered the derivative operator $\Lop = \Dop^N$ of order $N\geq 1$, for which $(\Dop^N)$-splines are  periodic piecewise-polynomial functions of degree less than $(N-1)$ and are $(N-2)$ times continuously differentiable for $N\geq 2$ at their junctions.
	We have seen that a $(\Dop^N)$-spline has at least 2 knots, and that the Green's function of $\Dop^N$ is not a periodic $(\Dop^N)$-spline. However, the function 
	\begin{equation}
		\rho_{\Dop^N} (x) =(\mathrm{D}^N)^\dagger \{\Sha\}(x) - (\mathrm{D}^N)^\dagger \{\Sha\}(x-\pi) = (\mathrm{D}^N)^{\dagger} \{ \Sha - \Sha(\cdot - \pi) \} (x)
	\end{equation}
	is a $(\Dop^N)$-spline, its weights $(a_1, a_2) = (1, -1)$ satisfying the system of equation $\bm{\mathrm{M}}  \bm{a} = \bm{0}$ of Proposition \ref{prop:Mnkmatrix}. 
	We represent $\rho_{\Dop^N}$ over two periods in Figure \ref{fig:Dgamma} for integer values of the parameter $N = \gamma$. 
	
\begin{figure}[h!]
\centering
	\includegraphics[width=0.7\textwidth]{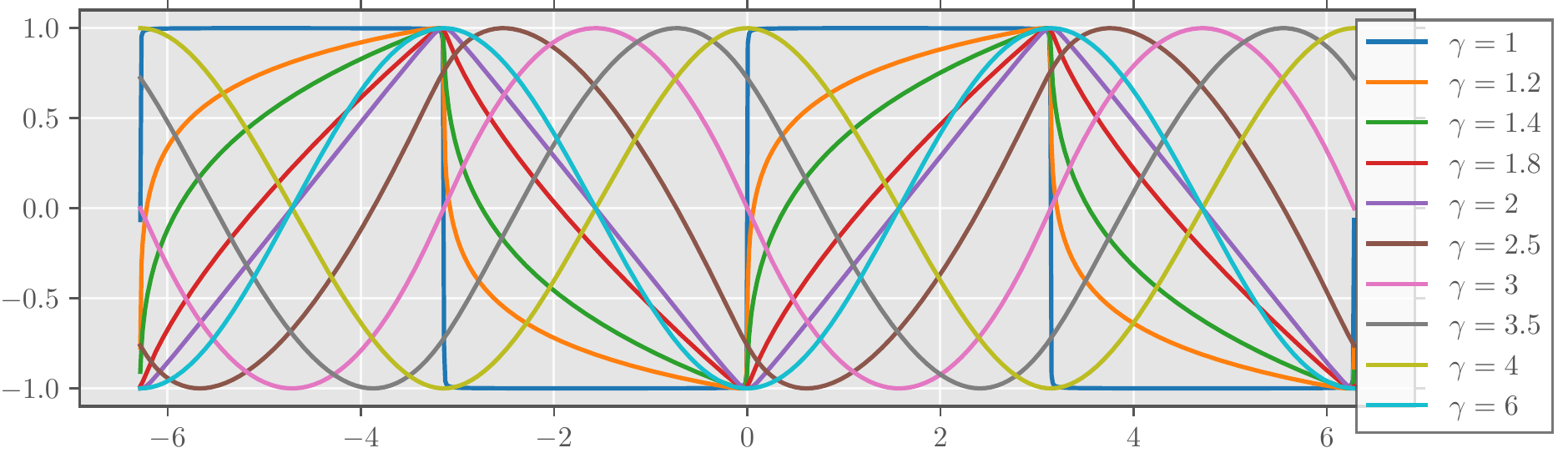}
	\caption{Periodic $\Lop$-splines $\rho_{\Dop^\gamma}$ defined in \eqref{eq:rhoforDgamma} associated to $\Lop = \mathrm{D}^\gamma$ over two periods for various values of $\gamma > 0$. 
	The splines are normalized so that the maximum value is $1$.}
	\label{fig:Dgamma}
\end{figure}

	

	\paragraph{Periodic Exponential Splines.} 
	We fix $\alpha \in \C$ and consider the operator $\Lop = (\Dop + \alpha \mathrm{Id})^N$ with $\alpha \in \C$. 
	We distinguish two cases, depending if $\alpha$ is in $\mathrm{i} \mathbb{Z}$ or not. 	
	Assume first that $\alpha \notin \mathrm{i} \mathbb{Z}$. Then,  $(\Dop + \alpha \mathrm{Id})^N$ has a trivial and therefore finite dimensional null space ($N_0=0$) and is invertible, with Fourier sequence $(\mathrm{i} k + \alpha)^{N}$, which corresponds to a spectral growth of $\gamma = N$. It is therefore spline-admissible. 
	A periodic $(\Dop - \alpha \mathrm{Id})^N$-spline $f$ with knots $x_1, \ldots , x_K$ is a piecewise-exponential-polynomial. More precisely, $f$ is a exponential-polynomial of the form $x \mapsto P(x) \exp(- \alpha x)$, with $P$ a polynomial of degree at most $(N-1)$ on each intervals $[x_{k+1} , x_k]$, $k = 1, \ldots, K$ (with the convention that $x_{K+1} = x_1 + 1$). For $N \geq 2$, $f$ has continuous derivatives   rder $(N-2)$.
	Moreover, the Green's function of $(\Dop + \alpha \mathrm{Id})$ satisfies for $x \in \T=[0,1]$ that $g_{(\Dop + \alpha \mathrm{Id})} (x) = (\Dop + \alpha \mathrm{Id})^{-1} \{ \Sha \} (x) = \frac{\mathrm{e}^{- \alpha x}}{1 - \mathrm{e}^{-\alpha}}$, or equivalently, for $x\in \R$,
	\begin{equation} \label{eq:formulaforDalphaIdGreen}
		g_{(\Dop + \alpha \mathrm{Id})} (x)  = \frac{\mathrm{e}^{- \alpha  (x - \lfloor x \rfloor )}}{1 - \mathrm{e}^{-\alpha}},
	\end{equation}
	where $\lfloor x \rfloor \in \Z$ is the largest integer smaller or equal to $x$. 
	To show this, it suffices to apply $(\Dop + \alpha \mathrm{Id})$ on both sides of \eqref{eq:formulaforDalphaIdGreen}. Similar formulas can be obtained for any $N \geq 1$. 
	We represent the periodic exponential splines $(\Dop - \alpha \mathrm{Id})^{-1} \{ \Sha \}$ for various values of $\alpha > 0$ and integer values $N= \gamma$ in Figure \ref{fig:expgamma}. 

\begin{figure}[h!]
\centering
	\includegraphics[width=0.7\textwidth]{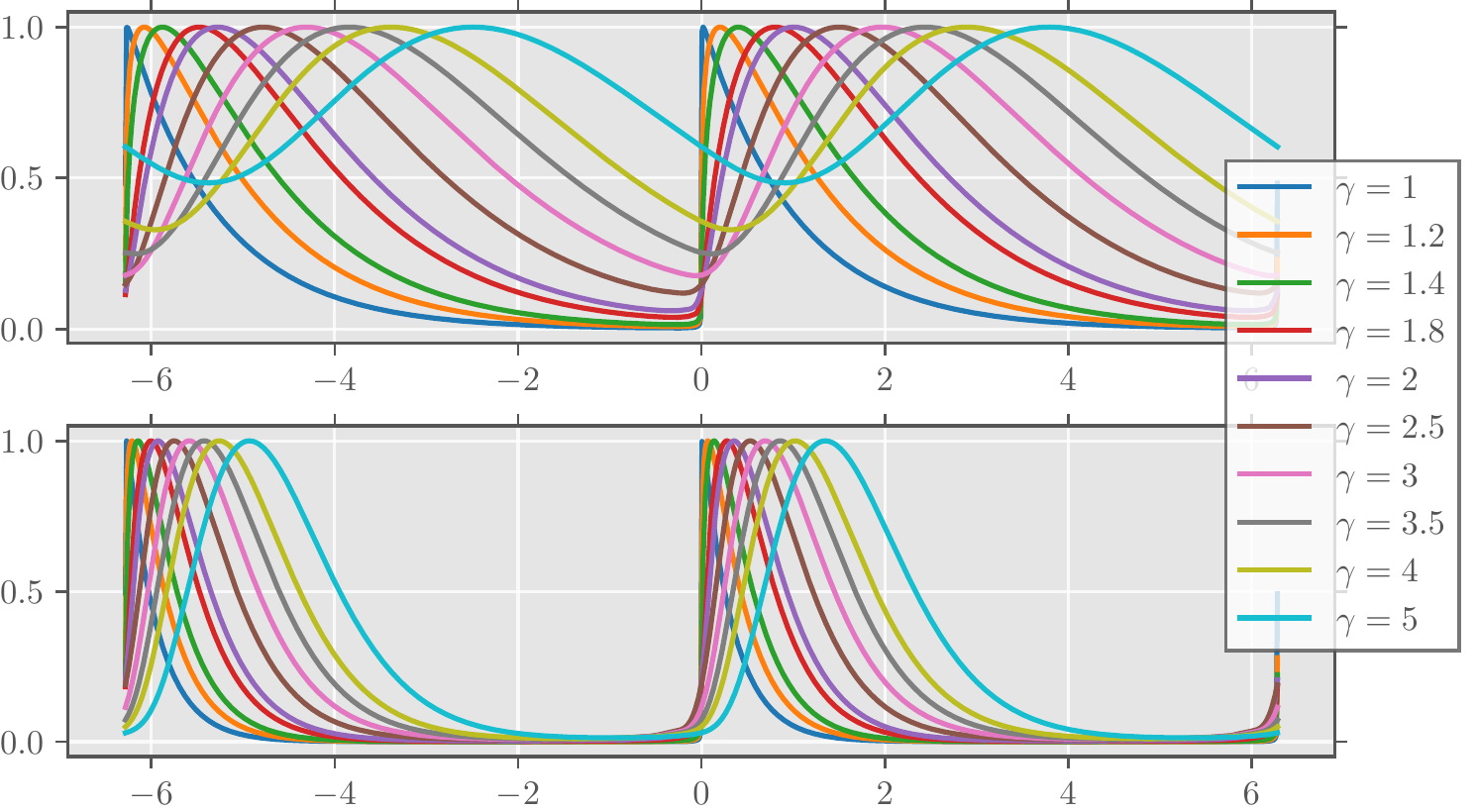}
	\caption{Periodic Green's functions $g_\Lop$ associated to $\Lop = (\mathrm{D} + \alpha \mathrm{Id})^\gamma$ given by \eqref{eq:formulaforDalphaIdGreen} over two periods for various values of $\gamma > 0$ and $\alpha = 1$ (top) or $\alpha = 3$ (bottom). The splines are normalized so that the maximum value is $1$.}
	\label{fig:expgamma}
\end{figure}

	Assume now that $\alpha = \mathrm{i} k_0 \in \mathrm{i} \mathbb{Z}$. In order to define a real operator, we consider instead $\Lop = (\Dop^2 + k_0^2 \mathrm{Id})^N = (\Dop - \mathrm{i} k_0 \mathrm{Id})^N  (\Dop + \mathrm{i} k_0 \mathrm{Id})^N$. This defines a spline-admissible operator whose null space of dimension $N_0 = 2$ is generated by $\cos(k_0 \cdot)$ and $\sin(k_0 \cdot)$, and whose pseudo-inverse has Fourier sequence $\One_{|k| \neq |k_0|} \cdot ( k_0^2 - k^2)^{-N}$. 
	
	More generally, we can consider $\Lop = P(\Dop)$ with $P = X^N + a_{N-1} X^{N-1} + \cdots + a_0$ a polynomial function. By decomposing $P(\Dop) = \prod_{n=1}^{N} (\Dop - \alpha_n \mathrm{Id})$ with $N$ the degree and $\alpha_n$ the complex roots of $P$, 	
	the general case reduces to the convolution between the periodic polynomial and/or periodic exponential splines considered above. The corresponding spectral growth is again $\gamma = N$. 

	\paragraph{Periodic Fractional Splines.}	
	There are several ways of considering fractional versions of integro-differential operators~\cite{samko1993fractional}.
	In our case, we follow the traditional approach for the periodic setting and consider Weyl's fractional derivative~\cite[Chapter XII, Section 8]{zygmund2002trigonometric}.
	For $\gamma > 0$, we define the operator $\Dop^{\gamma}$ from its Fourier sequence given by
	\begin{equation}
		\widehat{D^\gamma}[k] = (  \mathrm{i} k)^\gamma = \lvert  k \rvert^\gamma \exp\left(  - \frac{ \mathrm{i} \gamma \ \mathrm{sign} k}{4}  \right).
	\end{equation}
	When $\gamma = N \in \mathbb{N}$, we recover the classical $N$th order derivative. 
	
	The fractional derivative operator is spline-admissible.
	Its null space of $\Dop^{\gamma}$ is made of constant functions. The pseudoinverse $(\Dop^{\gamma})^\dagger$ of $\Dop^{\gamma}$ has Fourier sequence $\One_{k\neq 0} / (\mathrm{i} k)^\gamma$. Clearly, this corresponds to a spectral growth of $\gamma$. 
	As for the $N$th order derivative, a non constant periodic spline has at least two knots. In Figure \ref{fig:Dgamma}, we represent the function
		\begin{equation} \label{eq:rhoforDgamma}
		\rho_{\Dop^\gamma} (x) =(\mathrm{D}^\gamma)^\dagger \{\Sha\}(x) - (\mathrm{D}^\gamma)^\dagger \{\Sha\}(x-\pi) = (\mathrm{D}^\gamma)^{\dagger} \{ \Sha - \Sha(\cdot - \pi) \} (x)
	\end{equation}
	for various values of $\gamma > 0$. 
	
	We   also consider the fractional versions of the exponential splines, corresponding to the operator $\Lop = (\mathrm{D}+ \alpha \mathrm{Id})^\gamma$ with $\alpha \in \R$ and $\gamma >0$. In Figure \ref{fig:expgamma}, we represent such functions for various values of $\gamma$ and for $\alpha = 1$ and $\alpha = 3$. 
	
	\paragraph{Periodic Polyharmonic Fractional Splines.} 	
	The periodic fractional operator $(-\Delta)^{\gamma/2}$ is defined for $\gamma > 0$ via its Fourier response $\widehat{(-\Delta)^{\gamma/2}} [k] = \lvert k \rvert^\gamma$.
	Its impact on the smoothness of the input function is identical to the fractional derivative $\Dop^\gamma$ (and is actually equal for even integers $\gamma = 2n \geq 2$) but leads to different notions of periodic splines and admit, contrarily to the fractional derivative, non separable multivariate extensions (see Section \ref{sec:multivariate}).
	The operator $(-\Delta)^{\gamma/2}$ is spline-admissible, with a $1$-dimensional null space made of constant functions. In Figure \ref{fig:Lapgamma}, we plot the function
		\begin{equation} \label{eq:splinepolyharm}
		\rho_{(-\Delta)^{\gamma/2}} (x) =((-\Delta)^{\gamma/2})^\dagger \{\Sha\}(x) - ((-\Delta)^{\gamma/2})^\dagger \{\Sha\}(x-\pi) = (-\Delta)^{\gamma/2})^{\dagger} \{ \Sha - \Sha(\cdot - \pi) \} (x)
	\end{equation}
	for different values of $\gamma>0$. 

\begin{figure}[h!]
\centering
	\includegraphics[width=0.7\textwidth]{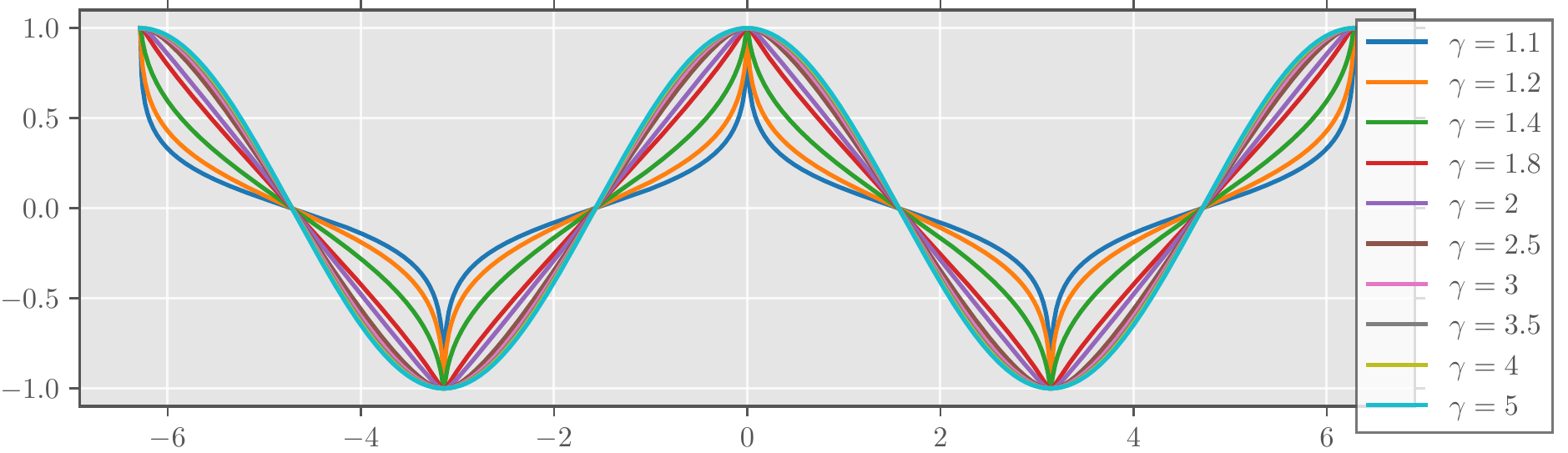}
	\caption{Periodic $\Lop$-splines \eqref{eq:splinepolyharm} associated to $\Lop = (-\Delta)^{\gamma/2}$ over two periods for various values of $\gamma > 0$. The splines are normalized so that the maximum value is $1$.}
	\label{fig:Lapgamma}
\end{figure}

	\paragraph{Periodic Sobolev Splines.} 	
Fix $\alpha > 0$ and $\gamma > 0$.
	We consider the periodic Sobolev operator  $\Lop_{\gamma,\alpha} = (\alpha^2 \mathrm{Id} - \Delta)^{\gamma/2}$, whose Fourier sequence is given by $\widehat{L}_{\gamma,\alpha} [{k}] = (\alpha^2 +k^2 )^{\gamma/2}$. 
	Then, $\Lop_{\gamma,\alpha}$  has a trivial null space and admits a periodic inverse operator: it is therefore spline-admissible.  Moreover, it has a spectral growth of $\gamma$. 
	We represent the periodic Green's function $g_{\Lop_{\gamma,\alpha}} = \Lop^{-1}_{\gamma,\alpha} \Sha$ in Figure  \ref{fig:sobolev1d}.
	\begin{figure}[h!]
\centering
	\includegraphics[width=0.7\textwidth]{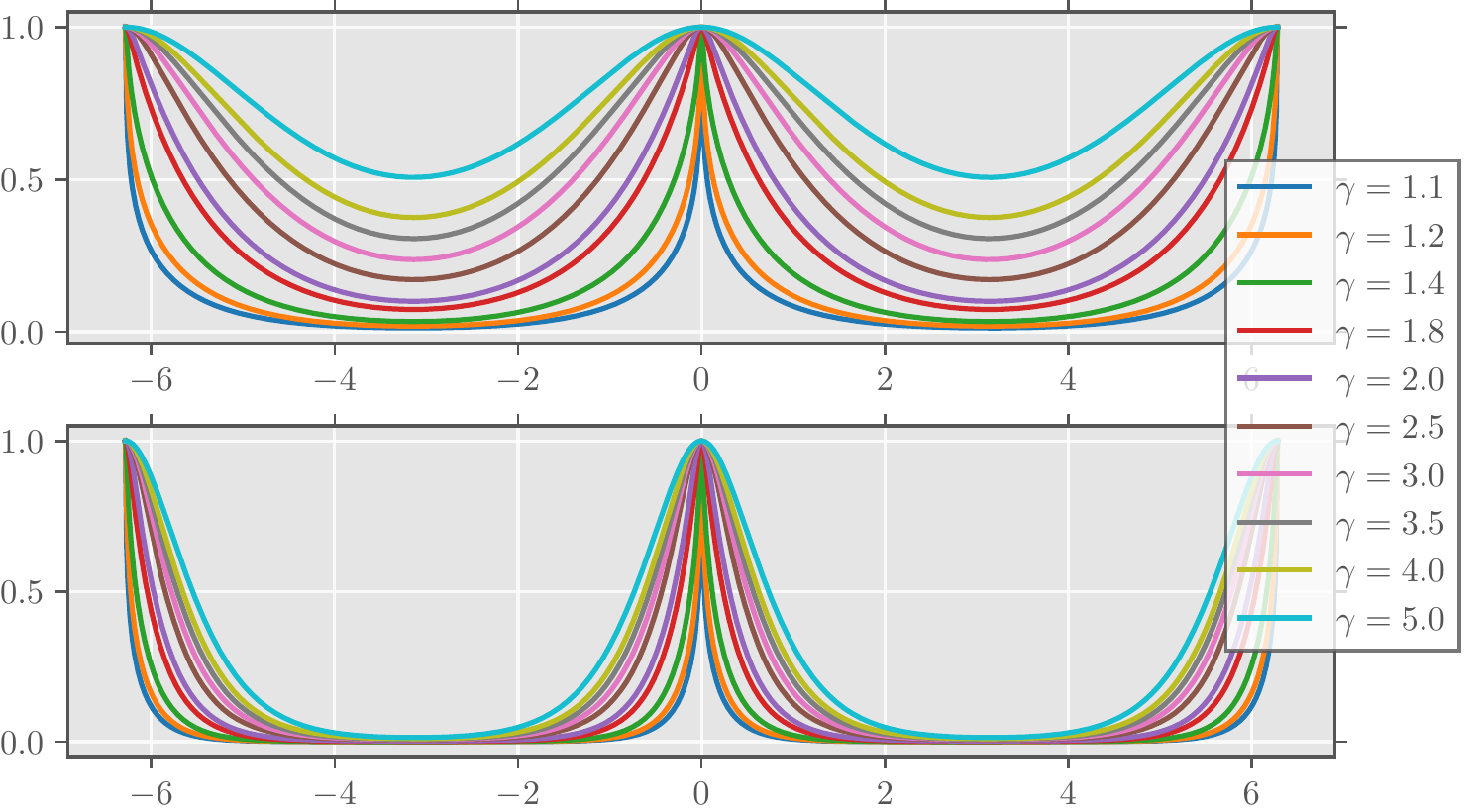}
	\caption{Periodic Green's functions associated to $\Lop = ( \alpha^2 \mathrm{Id} -\Delta)^{\gamma/2}$ over two periods for various values of $\gamma > 0$ and $\alpha\in \R$. The splines are normalized so that the maximum value is $1$.}
	\label{fig:sobolev1d}
\end{figure}

	\paragraph{Spline-admissible Operators via their Green's Functions.} 
	The previous spline-admissible operators have been defined via their Fourier sequences. It is convenient to readily recognize their null space and their spectral growth (see Definition \ref{def:growth}) but also has  some limitations.
	We now propose an alternative construction, which has been described more extensively in  \cite[Chapter 8]{simeoni2020functional}, for which the operator is specified from the construction of its Green's function. In particular, the operators we will consider have the following desirable properties. 
	
	\begin{itemize}	
	\item  They are invertible, self-adjoint, and symmetric in the sense that $g_{\Lop} (x) = g_{\Lop} (-x)$ for any $x \in \T$. 
	\item  They are well-localized: $g_\Lop( \cdot - x_0)$ is concentrated around $x_0$. This is especially appealing for applications, what can be leveraged in practice to design well-conditioned and parsimonious discretization schemes for the gTV-penalized optimisation problem \eqref{eq:optiwellstated} (see \cite[Chapter 8]{simeoni2020functional} for more details). 
	\item  Their Green's function $g_\Lop$ admits a closed form expression in spatial domain. 
	\item They have a well-identified spectral growth. 
	\end{itemize}
	
	The Sobolev operators typically share these properties with the exception of the third one. Moreover, the localization, which is visible in  Figure \ref{fig:sobolev1d} for large values of $\alpha>0$, can again be improved . 
	The general principle for constructing the Green's functions and their corresponding periodic operator is as follows. 
	Consider a function $G^\beta : \R^+ \rightarrow \R$, where $\beta > 1$ will play the role of a smoothness parameter. For $\epsilon > 0$, we set 
	\begin{equation}  
	g^\beta_\epsilon(x)=G^\beta\left(\epsilon^{-1}\sqrt{2-2\cos x }\right),\qquad \forall x\in \mathbb{T},
	\label{restrictions_rbf}
	\end{equation}
	Note that, upon identification of $\mathbb{T}=[0,2\pi)$ with the circle $\mathbb{S}^1=\{\bm{x}\in\R^2: \|\bm{x}\|=1\}\subset \R^2,$ it is possible to interpret \eqref{restrictions_rbf} as the  restriction on $\mathbb{S}^1$ of the radial function $\bm{x} \mapsto G^\beta( \epsilon^{-1} \lVert \bm{x}\rVert)$. 
	The following Lemma is a reformulation of \cite[Lemma 2.1]{gia2012multiscale} where the authors consider the general case of the $d$-dimensional sphere, that we particularize with $d=1$ with our notations.  Consider a function $G^\beta : \R^+ \rightarrow \R$ defining a radial function $G^\beta(\|\cdot\|) : \R^2 \rightarrow \R$.
	 Then, the Fourier transform of $G^\beta(\|\cdot\|)$ is itself radial, and given by $\widehat{G}^\beta( \lVert \cdot\rVert):\R^2\rightarrow \R$ where $\widehat{G}^\beta:\R_+\rightarrow \R$ is the \emph{Hankel transform of order zero} of ${G}^\beta$: $\widehat{G}^\beta(p):=\int_0^{+\infty} r {G}^\beta(r) J_0(rp) dr, \, p\in\R_+.$ 
	 
	 Then, we say that the radial function $G^\beta(\|\cdot\|)$ \emph{reproduces} the Sobolev space
	\begin{equation}\label{eq:soboR2}
	\mathcal{H}_2^{\beta}(\R^2)=\left\{f\in\mathcal{S}'(\R^2) , \quad  (\mbox{Id}-\Delta_{\R^2})^{\beta/2}f\in\mathcal{L}_2(\R^2)\right\},
	\end{equation}
	if $\widehat{G}^\beta$  is strictly positive and the bilinear map  $$(f, g) \mapsto \int_{\R^2}  \frac{ \widehat{f}(\bm{\omega} )\overline{\widehat{g}(\bm{\omega})}}{\widehat{G}^\beta( \lVert \bm{\omega}\rVert)} \mathrm{d} \bm{\omega},$$ is a well-defined inner-product for every element of $\mathcal{H}_2^{\beta}(\R^2)$, whose induced norm is equivalent to the canonical Hilbertian norm $\|(\mbox{Id}-\Delta_{\R^2})^{\beta/2}f\|_2$ on $\mathcal{H}_2^{\beta}(\R^2)$. If $G^\beta(\|\cdot\|)$ does {reproduce} the Sobolev space $\mathcal{H}_2^{\beta}(\R^2)$, we have in particular the set equality: 
	\begin{equation*}
	\mathcal{H}_2^{\beta}(\R^2) =\left\{ f \in \mathcal{S}'(\R^d), \ \int_{\R^2}  \frac{\lvert  \widehat{f}(\bm{\omega} ) \rvert^2}{\widehat{G}^\beta( \lVert \bm{\omega}\rVert)} \mathrm{d} \bm{\omega} < \infty \right\}.
	\end{equation*}

	\begin{lemma} \label{lemma:gia}
	Let $\beta > 1$ and  $G^\beta(\|\cdot\|): \R^2 \rightarrow \R$ be a radial function reproducing $\mathcal{H}_2^\beta(\R^2)$. Then, for each $\epsilon > 0$, there exists constants $0< c_1 \leq c_2 < \infty$ such that
 \begin{equation}  
 c_1(1+\epsilon | k |)^{-2(\beta-1/2)}\leq \widehat{g}_\epsilon^\beta[k] \leq c_2(1+\epsilon  | k |)^{-2(\beta-1/2)}, \quad \forall k\in\Z,
 \label{tight_bound}
  \end{equation}
  where the $ \widehat{g}_\epsilon^\beta[k]$ are the Fourier coefficients of the periodic function $g_\epsilon^\beta$ in \eqref{restrictions_rbf}.  
	\end{lemma} 
	
	One can use the function $g_\epsilon^\beta$ to specify a spline-admissible operator, as detailed in the next elementary proposition.
	
	\begin{proposition} \label{prop:usefulcriterion}
	Let $g_\epsilon^\beta$ be as above. 
	Then, the Fourier sequence $(1/ \widehat{g}_\epsilon^\beta[k])_{k\in \Z}$ defines an operator in $\mathcal{L}_{\mathrm{SI}} (\Sp'(\T^d))$ via the relation
	\begin{equation} \label{eq:Gepsibetadef}
 	\mathrm{G}_\epsilon^\beta \{f\} = \sum_{k\in \Z} \frac{\widehat{f}[k]}{\widehat{g}_\epsilon^\beta[k]} e_k, \qquad \forall f \in \Sp'(\T).  
	\end{equation}
	Moreover,  $\mathrm{G}_\epsilon^\beta$ is spline-admissible with trivial null space. Its Green's function is $g_\epsilon^\beta$ and its spectral growth is $\gamma = 2(\beta - 1/2)$.
	\end{proposition}
	
	\begin{proof}
	The sequence $(1/ \widehat{g}_\epsilon^\beta[k])_{k\in \Z}$ is clearly slowly growing due to the left inequality in \eqref{tight_bound}, hence the operator $\mathrm{G}_\epsilon^\beta$ defined by \eqref{eq:Gepsibetadef} is   in $\mathcal{L}_{\mathrm{SI}} (\Sp'(\T^d))$ (see Proposition \ref{prop:firsttrivialstuff}). 
	The Fourier sequence of $\mathrm{G}_\epsilon^\beta$ is non vanishing, which implies, due to Proposition \ref{prop:finitedimNL}, that the null space of $\mathrm{G}_\epsilon^\beta$ is trivial. Moreover, the spectral growth is directly deduced from \eqref{tight_bound} in Lemma \ref{lemma:gia}. Finally, we readily see that, by construction, $\mathrm{G}_\epsilon^\beta g_\epsilon^\beta = \Sha$, so that $g_\epsilon^\beta$ is effectively the Green's function of $\mathrm{G}^\beta_\epsilon$. 
	\end{proof}

	\paragraph{Periodic Mat\'ern Splines.}

For $\beta > 1$, we define $G^\beta$ as the \emph{Mat\'ern function of order $(\beta - 1)$}, which is given by~\cite[Eq. (4.14)]{rasmussen2003gaussian}

\begin{equation} \label{eq:kernelmatern}
G^\beta(r)= S_{\beta-1}(r) =   \frac{2^{2-\beta}}{\Gamma(\beta - 1)}\left(\sqrt{2(\beta - 1)} r\right)^{\beta - 1} K_{\beta - 1}\left(\sqrt{2(\beta - 1)}r\right), \qquad \forall r \geq 0,
\end{equation}
where 
where $K_\nu$ denotes the \emph{modified Bessel function of the second kind of parameter $\nu> 0$} \cite[Section 9.6]{abramowitz1948handbook}. 
\begin{definition} \label{def:matern}
Let $\beta > 1$ and $\epsilon > 0$. The Mat\'ern operator $\mathrm{M}_{\epsilon}^{\beta}$ is the operator whose Green's function $g_\epsilon^\beta$ is given by \eqref{restrictions_rbf}, where the radial function satisfies \eqref{eq:kernelmatern}. 
\end{definition}

The next proposition, mostly based on known results, characterizes the properties of the Mat\'ern operators $\mathrm{M}_{\epsilon}^{\beta}$.

\begin{proposition} \label{prop:maternop}
For each $\beta > 1$ and $\epsilon > 0$, $\mathrm{M}_{\epsilon}^{\beta}$ is a spline-admissible operator with trivial null space, spectral growth $\gamma = 2(\beta - 1/2)$, and whose Fourier sequence $\widehat{g}_{\epsilon}^\beta$ satisfies \eqref{tight_bound}. 
Moreover, when $\beta \in \mathbb{N}_{\geq 1}+ 1/2$, the Green's function is the product of an exponential and a polynomial functions due to the relation, for the Mat\'ern function,
\begin{equation}
\label{matern_function}
 S_{k+1/2}(r)=\exp\left(-\sqrt{2k+1} r\right)\frac{k!}{(2k)!}\sum_{i=0}^k \frac{(k+i)!}{i!(k-i)!}\left(\sqrt{8k+4} r\right)^{k-i}, \qquad \forall r \geq 0.
 \end{equation}
 with $k \in \mathbb{N}$. 
\end{proposition}

\begin{proof}
The kernel $G^\beta$ defined by \eqref{eq:kernelmatern} reproduces the Sobolev space $\mathcal{H}_2^{\beta}(\R^2)$ according to \cite[Theorem 6.13]{wendland2004scattered}. We can therefore apply Lemma \ref{lemma:gia} to deduce \eqref{tight_bound}.
Moreover, the relation \eqref{matern_function} has been derived in \cite[Eq. (4.16)]{rasmussen2003gaussian} and \cite[Eq. (10.2.15)]{abramowitz1948handbook}.
The rest of the Proposition is then a direct application of Proposition \ref{prop:usefulcriterion} to the case of the Mat\'ern function. 
\end{proof}

Proposition \ref{prop:maternop} implies that the Green's function of the Mat\'ern operator admit a closed form expression when $\beta - 1/2 \in \mathbb{N}_{\geq 1}$. For instance, for $\beta = 3/2$, we get that $g_\epsilon^{3/2}(x) = S_{1/2}(\epsilon^{-1} \sqrt{2- 2\cos x})$  for every $x \in \T$, 
where $S_{1/2}(r)=\exp(-r)$ for all $r \geq 0$.
The Mat\'ern function converges to the Gaussian function~\cite[Chapter 4, p. 84]{rasmussen2003gaussian} and is practically indistinguishable from it when  $\beta \geq 9/2$.
Therefore, the periodic Mat\'ern Green's function $g_\epsilon^\beta$ in \eqref{restrictions_rbf} should resemble a bump function, sharply decaying away from zero. Using the Gaussian approximation above, it is moreover possible to approximate its effective support as $2\arcsin(5\epsilon/2)$, highlighting the role of the parameter $\epsilon$. 

On Figure \ref{fig:matern}, we plot the periodic Mat\'ern Green's function $g_\epsilon^\beta$ for 
$\beta\in\{3/2,5/2,7/2,9/2\}$ and $\epsilon\in\{1,3/10\}$. 
As we have seen, due to \eqref{tight_bound}, this corresponds to a spectral growth $\gamma = 2(\beta - 1/2) \in \{2, 4, 6, 8\}$. 
We use the parameter $\gamma$ in Figure \ref{fig:matern} to be consistent with other the other family of operators. 
We observe moreover that the function  $g_\epsilon^\beta$ is more localized for smaller scale parameters $\epsilon$.   

\begin{figure}[t!]
\centering
	\includegraphics[width=0.7\textwidth]{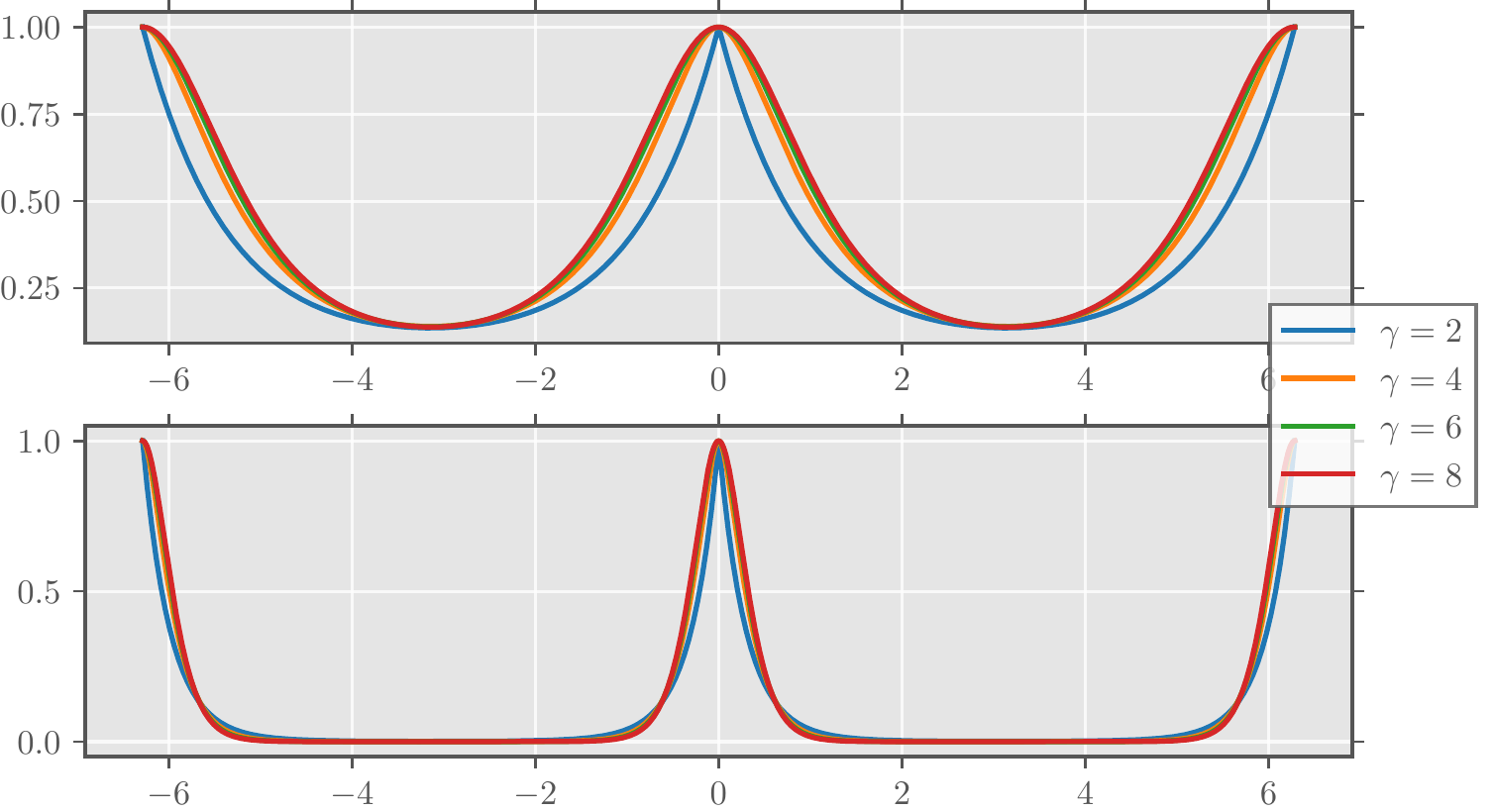}
	\caption{Periodic $\Lop$-splines associated to the Mat\'ern operators over two periods for various values of $\gamma > 0$ and $\epsilon = 1$ (top) or $\epsilon = 0.3$ (bottom). The splines are normalized so that the maximum value is $1$. }
	\label{fig:matern}
\end{figure}

\paragraph{Periodic Wendland Splines.}

While exhibiting good spatial localizations, the Green's functions of the Mat\'ern operators are nevertheless theoretically supported over the entire period. By considering compactly supported radial functions $G^\beta$ in \eqref{restrictions_rbf}, it is possible to construct a class of spline-admissible operators, namely the \emph{Wendland operators}, whose Green's functions are supported on a subregion of $\mathbb{T}$. 

\begin{definition}\label{def:wendland}
For $\mu \in \mathbb{N}$ and $\beta \in \mathbb{N}_{\geq 2}$, the \emph{missing Wendland function} is defined by
\begin{equation}\label{eq:Gmubeta}
W_{\mu}^{\beta}(r):=P_{\mu,\beta - 3/2}(r^2)\log\left(\frac{r}{1+\sqrt{1-r^2}}\right)+Q_{\mu,\beta - 3/2}(r^2)\sqrt{1-r^2},\qquad \forall r \geq 0,
\end{equation}
where $P_{\mu,\beta - 3/2}, Q_{\mu,\beta - 3/2}$ are polynomial functions defined in \cite[Eq. 3.12]{zhu2012compactly}  and \cite[Corollary 4.6]{hubbert2012closed}, respectively. 
The periodic Wendland operator $\mathrm{W}_{\epsilon,\mu}^\beta$ is then the operator whose Green's function $g_{\epsilon,\mu}^\beta$ is given by \eqref{restrictions_rbf} with $G^\beta = W_{\mu}^{\beta}$. 
\end{definition}

The next proposition characterizes the properties of the Mat\'ern operators $\mathrm{W}_{\epsilon,\mu}^{\beta}$.

\begin{proposition} \label{prop:wendlanderie}
Let $\mu \in \mathbb{N}$ and $\beta \in \mathbb{N}_{\geq 2}$. Then, the periodic Wendland operator $\mathrm{W}_{\epsilon,\mu}^\beta$ is in $\mathcal{L}_{\mathrm{SI}} (\Sp'(\T^d))$, is spline-admissible, and has a trivial null space. Moreover, the Fourier sequence of its Green's function satisfies \eqref{tight_bound} and its support is $[-2\arcsin(\epsilon/2),2\arcsin(\epsilon/2)]\subset [-\pi,\pi]$ for $0<\epsilon\leq 2$. Finally, the spectral growth of $\mathrm{W}_{\epsilon,\mu}^\beta$ is $2(\beta - 1/2)$. 
\end{proposition}

\begin{proof}
According to \cite[Proposition 3.5]{zhu2012compactly} , the missing Wendland function $G_{\mu,\beta}$ reproduces the Sobolev spaces $\mathcal{H}^\beta_2(\mathbb{R}^2)$. Lemma \ref{lemma:gia} therefore applies and \eqref{tight_bound} is shown. As for the Mat\'ern case, this means that the corresponding operator with Green's function $g_\epsilon^\beta$ is well-defined and spline-admissible with a trivial null space. 
Moreover, the missing Wendland functions are compactly supported on $[0,1]$ \cite[Theorem 2.2]{hubbert2012closed}, which implies that the support of $g_\epsilon^\beta$ is $[-2\arcsin(\epsilon/2),2\arcsin(\epsilon/2)]\subset [-\pi,\pi]$ for $0 < \epsilon \leq 2$. 
\end{proof}

Closed form expressions for the missing Wendland functions with parameters $\mu = \beta \in \{ 2, 3, 4, 5\}$ are listed in \cite[Table 4.2]{zhu2012compactly}. 
In Figure \ref{fig:wendland}, we represent the periodic Wendland Green's functions for $\mu = \beta \in \{2,3,4\}$. We again use the spectral growth   $\gamma = 2(\beta - 1/2)  \in \{3,5,7\}$. \\

\textit{Remark.} It is worth noting that the Mat\'ern splines admitting a closed form expression are associated to an even spectral growth $\gamma \in 2 \mathbb{N}_{\geq 1}$. In comparison, the Wendland operator has an odd spectral growth $\gamma \in 2\mathbb{N}_{\geq 1}+ 1$. 

	\begin{figure}[t!]
\centering
	\includegraphics[width=0.7\textwidth]{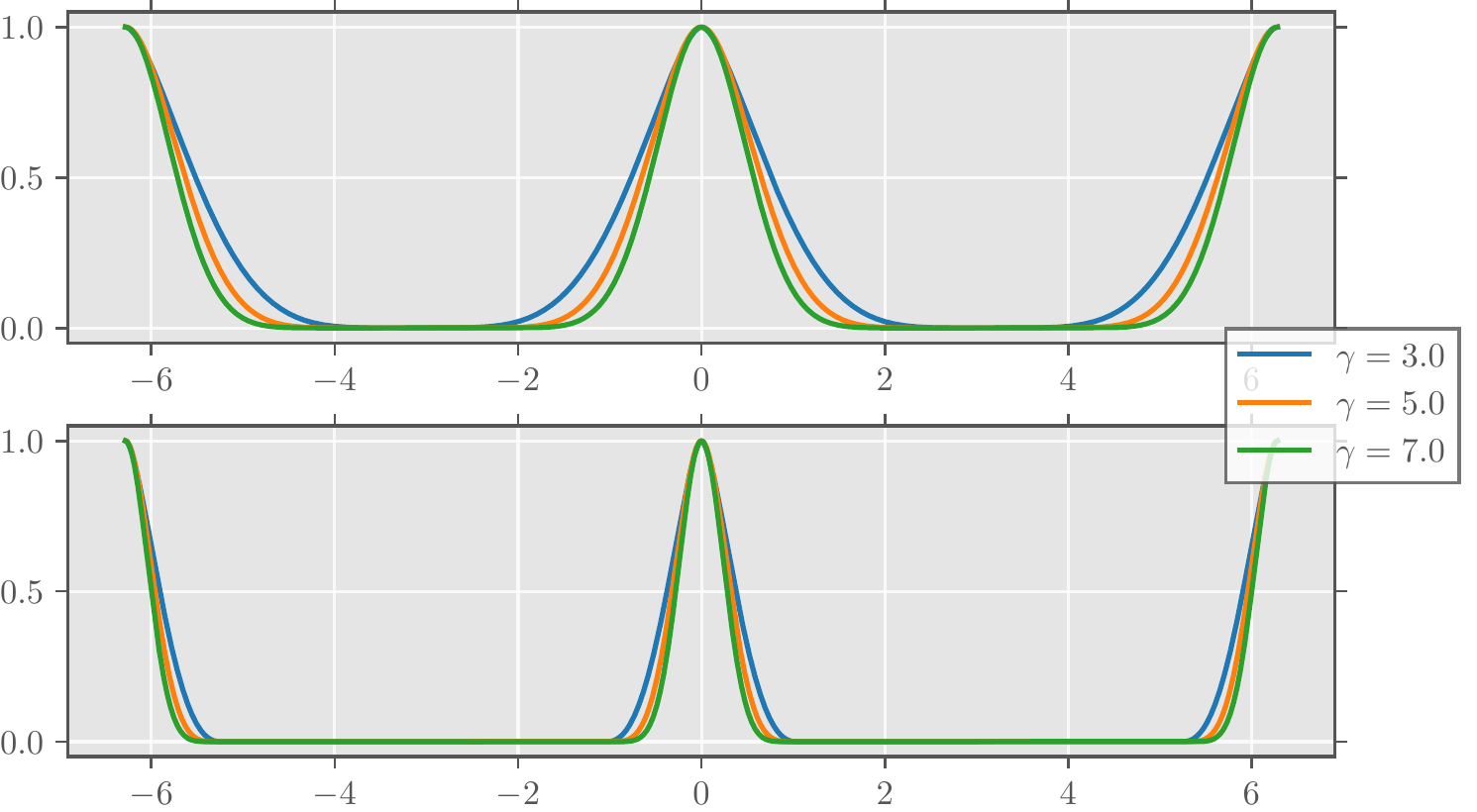}
	\caption{Periodic $\Lop$-splines associated to the Wendland operator over two periods for various values of $\gamma > 0$. The splines are normalized so that the maximum value is $1$.}
	\label{fig:wendland}
\end{figure}

	\subsection{Multivariate Splines-admissible Operators} \label{sec:multivariate}

	In any ambiant dimension $d \geq 1$, we consider two types of spline-admissible operators.
	First, we consider separable ones, based on univariate operators for each of the $d$ variables. Second, we introduce operators with isotropic Fourier sequences, which are nonseparable and have isotropic Green's functions.
	We plot families multivariate periodic splines for the ambiant dimensions $d = 2$ and $d=3$. 
	Table \ref{table:operators1d} provides the list of some of the multivariate spline-admissible operators introduced thereafter, together with their Fourier sequence and their null space via the set of null space frequencies $N_{\Lop}$. 
	
\begin{table*}[h!] 
\centering
\caption{Families of multivariate spline-admissible operators}
	\begin{tabular}{ccccccc} 
\hline
\hline 
Spline's type & Operator & Parameter & $\widehat{L}[\bm{k}]$ & $N_\Lop$  
\\
\hline\\[-1ex]
Separable splines  & $\prod_{i=1}^d (\Dop_i - \alpha_i \mathrm{Id})^{\gamma_i}$ & $\gamma_i >0, \ \alpha_i \in \R$ & $\prod_{i=1}^d ( \mathrm{i} k_i  - \alpha_i)^{\gamma_i}$ & $\emptyset$ \\
Polyharmonic spline & $\Delta^N$ & $N\in \N$ & $ (-1)^N \lVert \bm{k}\rVert^{2N}  $ & $0$   \\
& $\Delta +  \lVert \bm{k}_0 \rVert^2  \Id$ & $\bm{k}_0 \in \Z^d$ &  $  \lVert \bm{k}_0\rVert^2 -  \lVert \bm{k}\rVert^2$  &   $\{ \bm{k}, \  \lVert \bm{k}\rVert^2 =  \lVert \bm{k}_0\rVert^2 \} $   \\
Fractional polyharmonic splines &  $(- \Delta)^{\gamma/2}$ & $\gamma > 0$ &  $\lVert \bm{k}\rVert^{\gamma}$ & $0$   \\
Sobolev splines & $(\alpha^2 \Id - \Delta)^{\gamma/2}$ & $\alpha > 0$, $\gamma > 0$ &  $( \alpha^2 + \lVert \bm{k}\rVert^2)^{\gamma/2}$ &  $\emptyset$    \\
\hline
\hline
\end{tabular} \label{table:operatorsmultid}
\end{table*}	
	
	\paragraph{Periodic Separable Splines.}
	Let $\Lop_i$ be univariate spline-admissible operators for $i=1,\ldots , d$. We assume moreover that each $\Lop_i$ has a trivial null space, and therefore admits an inverse operator $\Lop_i^{-1}$. Then, the operator $\Lop$ with Fourier sequence 
	\begin{equation} \label{eq:Lopseparable}
	\widehat{L}[\bm{k}] = \prod_{i=1}^d \widehat{L}_i[k_i]
	\end{equation}
	for any $\bm{k} = (k_1, \ldots , k_d) \in \Z^d$ is spline-admissible with trivial null space and inverse $\Lop^{-1}$ with Fourier sequence $(1/\widehat{L}[\bm{k}])_{\bm{k}\in \Z^d}$.  
	
	We denote by $\Dop_i$ the derivative with respect to the $i$th coordinate. Applying the previous principle, we easily see that the operator  
	$\Lop = \prod_{i=1}^d (\Dop_i - \alpha_i \mathrm{Id})^{\gamma_i}$ with $\gamma_i > 0$, $\alpha_i \in \R$ is a periodic separable spline-admissible operator, using $\Lop_i = ( \Dop - \alpha_i \mathrm{Id})^{\gamma_i}$ for any $i=1, \ldots , d$ in \eqref{eq:Lopseparable}. 
	We represent the corresponding separable splines with a unique knot at $\bm{x} = \bm{0}$ for the ambiant dimension $d=2$ in Figure \ref{fig:exposeparable}.  \\
	
	\begin{figure}[t!]
\centering
	\includegraphics[width=0.70\textwidth]{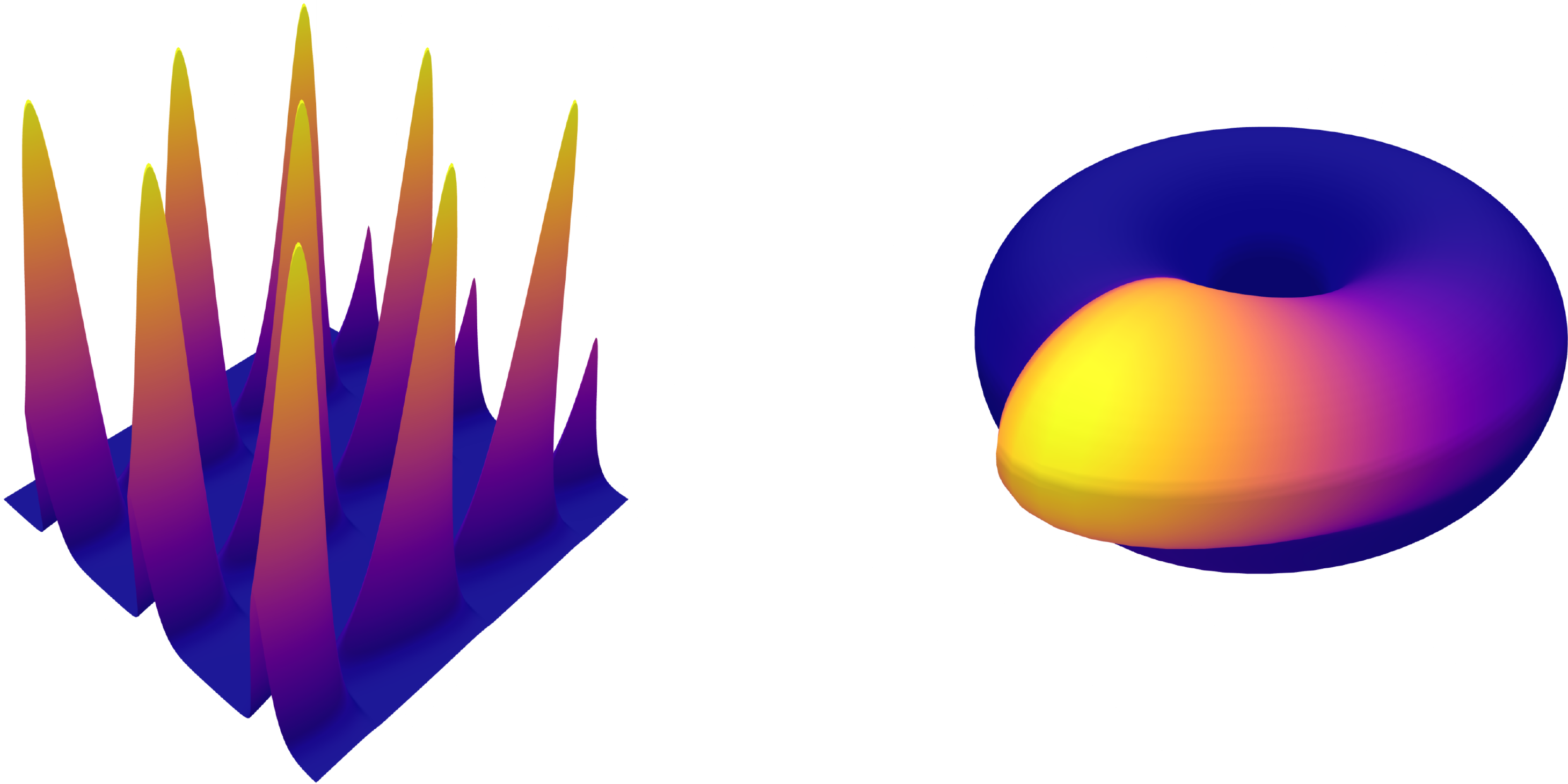}
	\caption{Periodic $\Lop$-spline with one knot $\bm{x}_1 = \bm{0}$ (Green's function)  associated to $\Lop = (\mathrm{D}_1 - \mathrm{Id})^2 (\mathrm{D}_2 - 3 \mathrm{Id})^2$ over $3 \times 3$ periods (left) or in the geometrical $2$d torus (right).}
	\label{fig:exposeparable}
\end{figure}
	
	\textit{Remark.} If one of the $\Lop_i$ has a non trivial null space, then the null space of $\Lop$ defined by \eqref{eq:Lopseparable} is infinite-dimensional, and the operator $\Lop$ is therefore not spline-admissible. Indeed, the null space contains any generalized function of the form $\bm{x} = (x_1, \ldots, x_d) \mapsto p(x_i) f(x_1, \ldots , x_{i-1}, x_{i+1} , \ldots x_d)$ for any $f \in \Sp'(\T^{d-1})$ and $p \in \mathcal{N}_{\Lop_i}$. As a typical example, the operator $\Dop_1 \ldots \Dop_d$ is not spline-admissible. 

	\paragraph{Multivariate Polyharmonic Splines.}
	The fractional Laplacian operator $(-\Delta)^{\gamma/2}$ introduced in the univariate setting has multivariate counterpart. It corresponds to the Fourier sequence $(\lVert \bm{k}\rVert^\gamma)_{\bm{k}\in \Z^d}$. Its null space is made of constant functions, and its pseudoinverse has Fourier sequence $\One_{\bm{k}\neq 0} / \lVert \bm{k} \rVert^\gamma$. It is therefore spline admissible. As we have seen for the derivative $\Dop$, a  non constant periodic $(-\Delta)^{\gamma/2}$-spline has at least two dots. We represents such periodic splines with knots in $\bm{0}$ and $(\pi , 0)$ in Figure \ref{fig:bilaplacien2d}. 
	They corresponds to the periodic counterpart of the polyharmonic splines considered for instance in \cite{Madych1990polyharmonic,rabut1992elementarym,VDV2005polyharmonic}.

\begin{figure}[t!]
	\centering
	\includegraphics[width=0.7\textwidth]{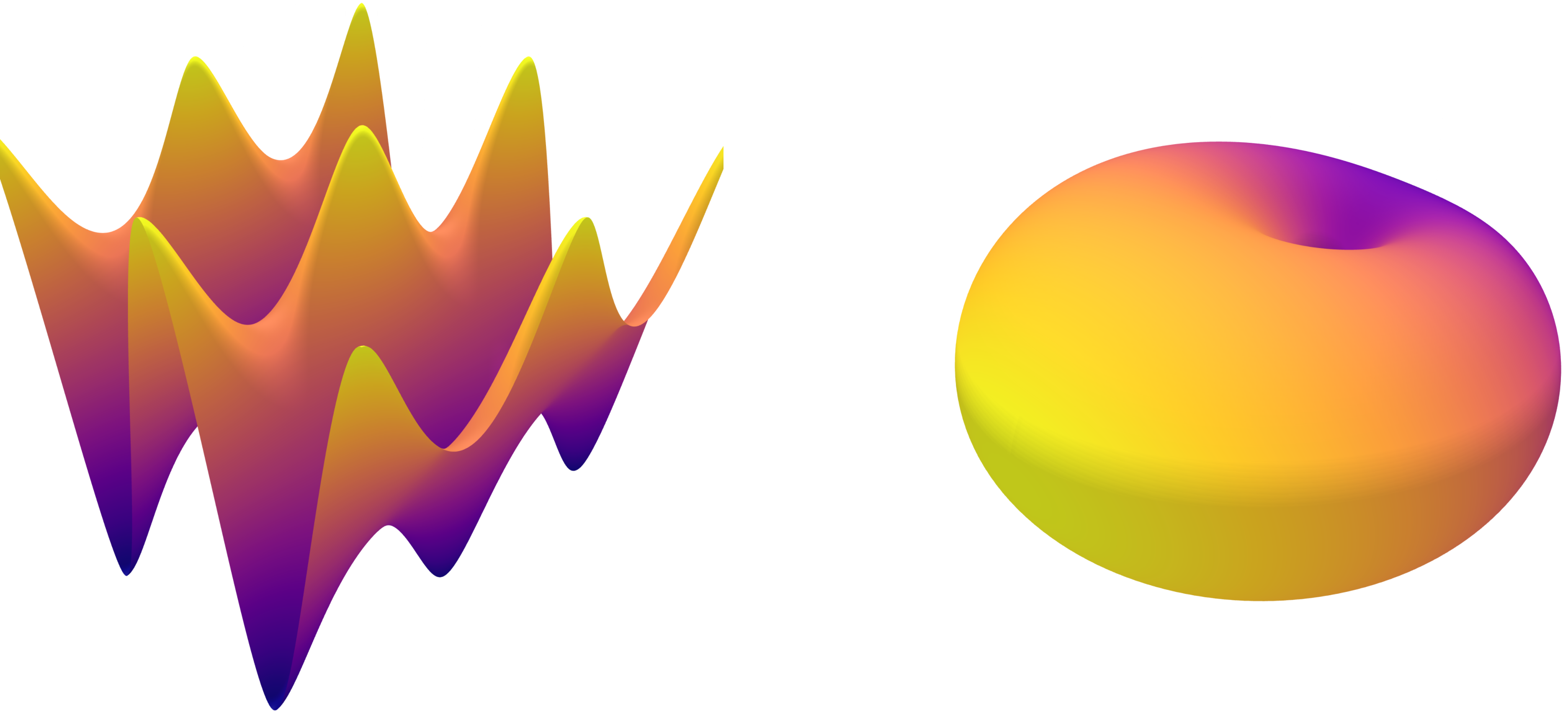}
	\caption{Periodic $\Lop$-spline with two knots at $\bm{x}_1 = (0,0)$ and $\bm{x}_2 = (0, \pi)$ associated to $\Lop = (-\Delta)^2$ over $3 \times 3$ periods (left) or in the geometrical $2$d torus (right).}
	\label{fig:bilaplacien2d}
\end{figure}	
	
	\paragraph{Multivariate Sobolev Splines.}
	Similarly as what we did for the fractional Laplacian, the Sobolev operator admits a multivariate generalization, the operator $(\alpha^2 \mathrm{Id} - \Delta)^{\gamma/2}$ having Fourier sequence $( \alpha^2 + \lVert \bm{k}\rVert^2)^{\gamma/2}$. It is an invertible operator with trivial null space, and is therefore spline-admissible. The corresponding splines with a unique knot at $\bm{x}=\bm{0}$ are represented in Figure \ref{fig:bisobolev2d}.

	\begin{figure}[t!]
\centering
	\includegraphics[width=0.7\textwidth]{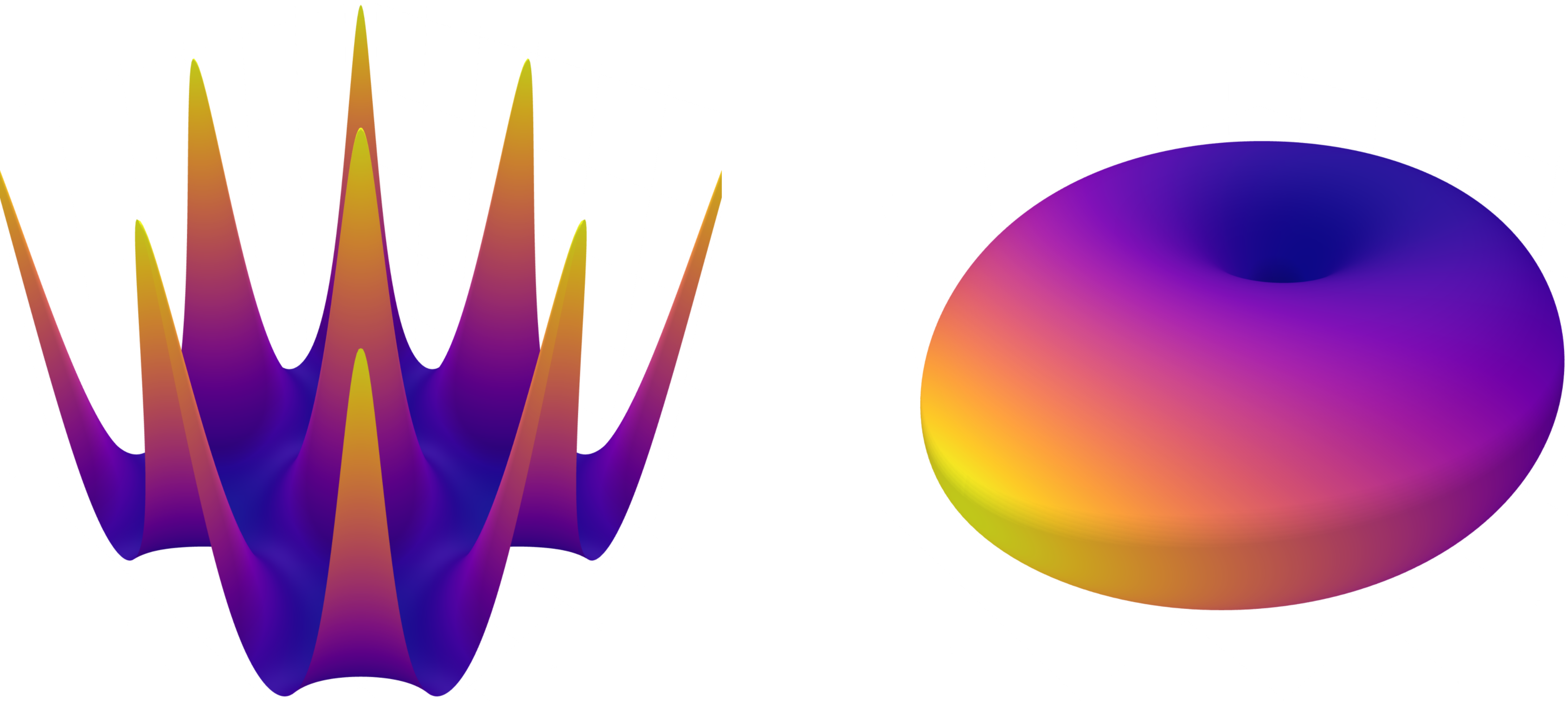}
	\caption{Periodic $\Lop$-spline with one knot $\bm{x}_1 = \bm{0}$ (Green's function)  associated to $\Lop = (4 \mathrm{Id} -\Delta)^2$ over $3 \times 3$ periods (left) or in the geometrical $2$d torus (right).}
	\label{fig:bisobolev2d}
\end{figure}
	
 	\paragraph{Periodic splines in dimension $d=3$.}
	To illustrate the versatility of our framework, we also represent periodic splines in ambiant dimension $d=3$ for separable exponential operators, Sobolev operators and separable Mat\'ern operators in Figure \ref{fig:exponential3d}. 
	
		\begin{figure}[t!]
\centering
	\includegraphics[width=0.3\textwidth]{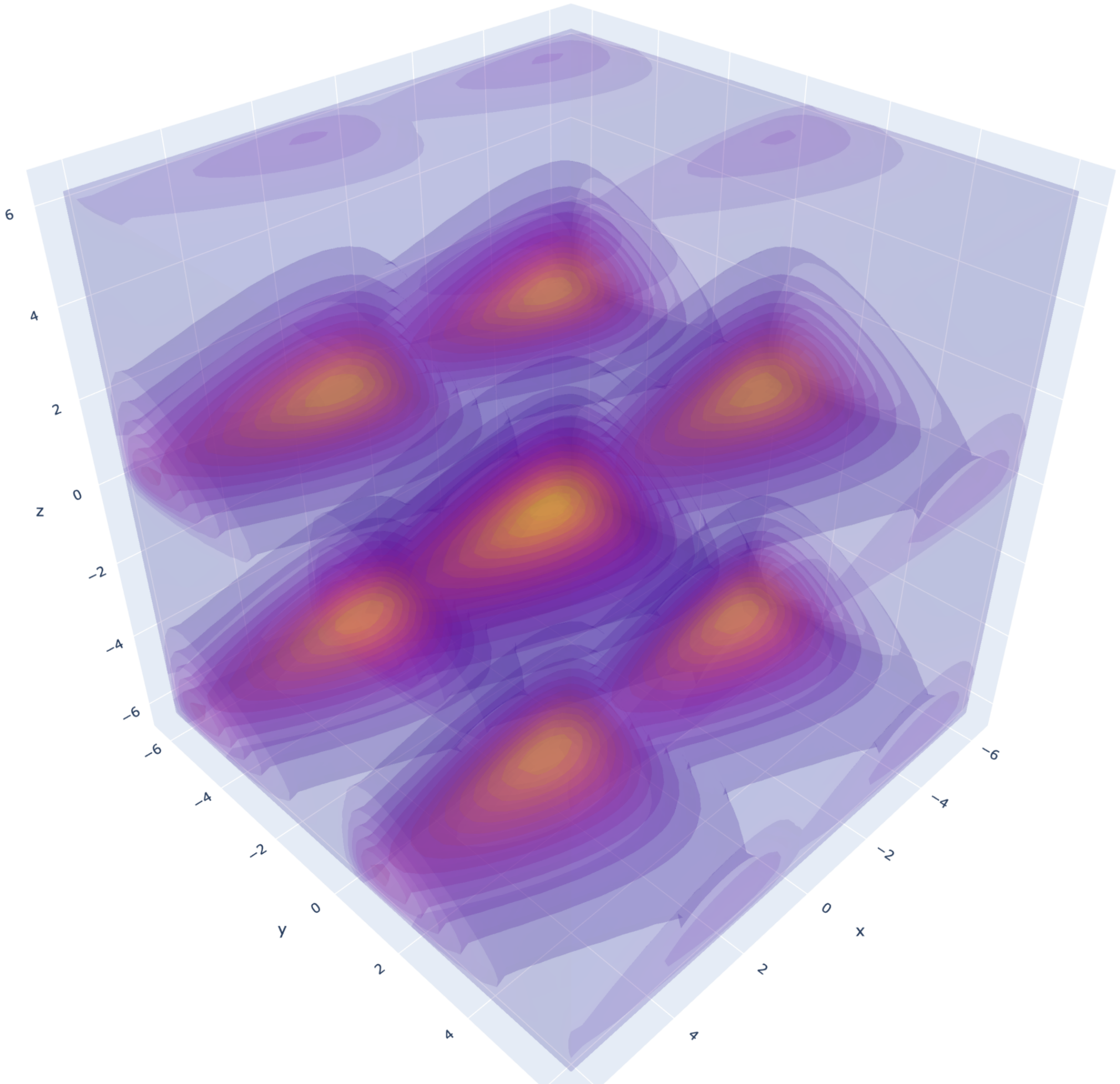}
	\includegraphics[width=0.3\textwidth]{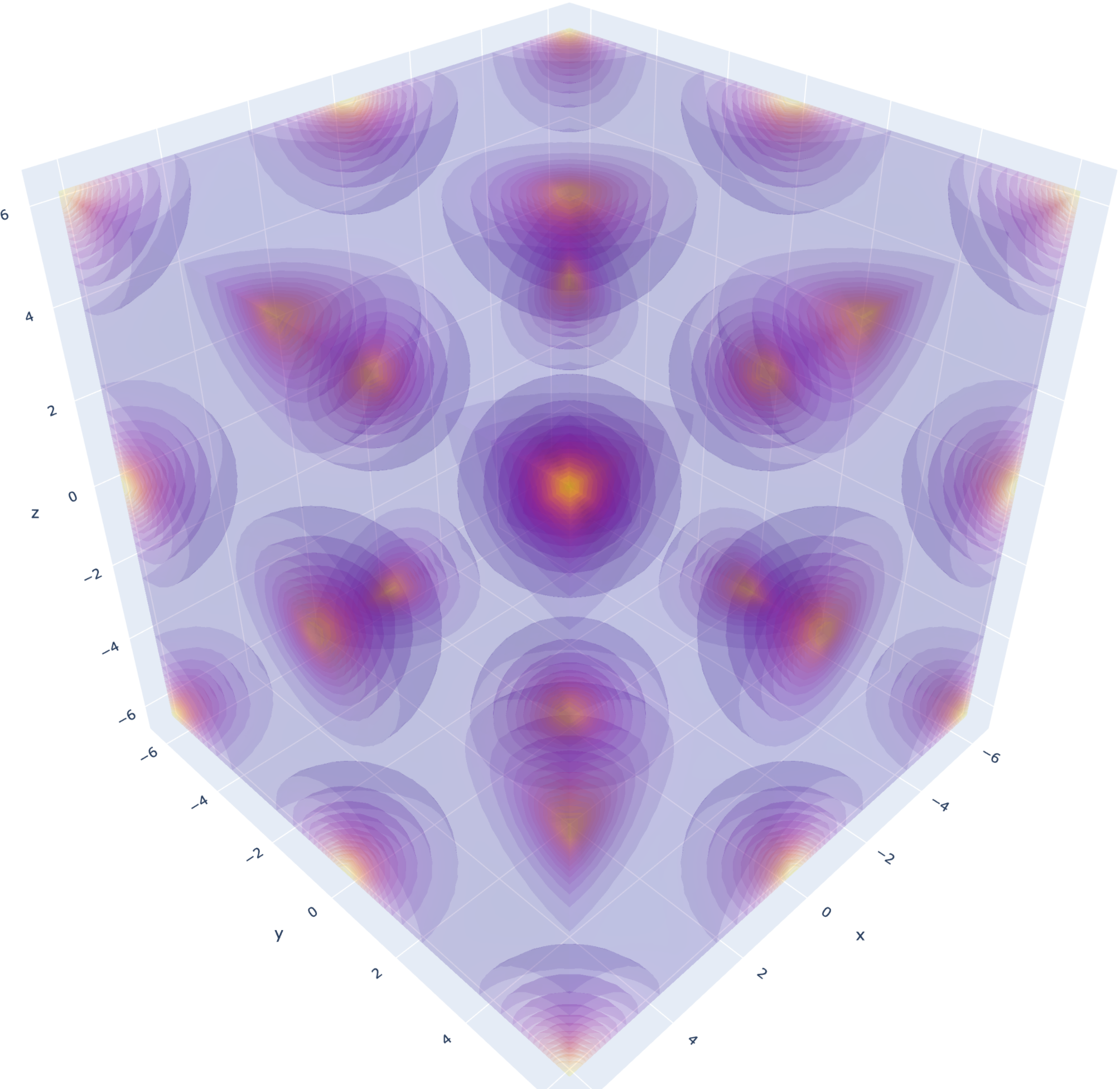}
	\includegraphics[width=0.3\textwidth]{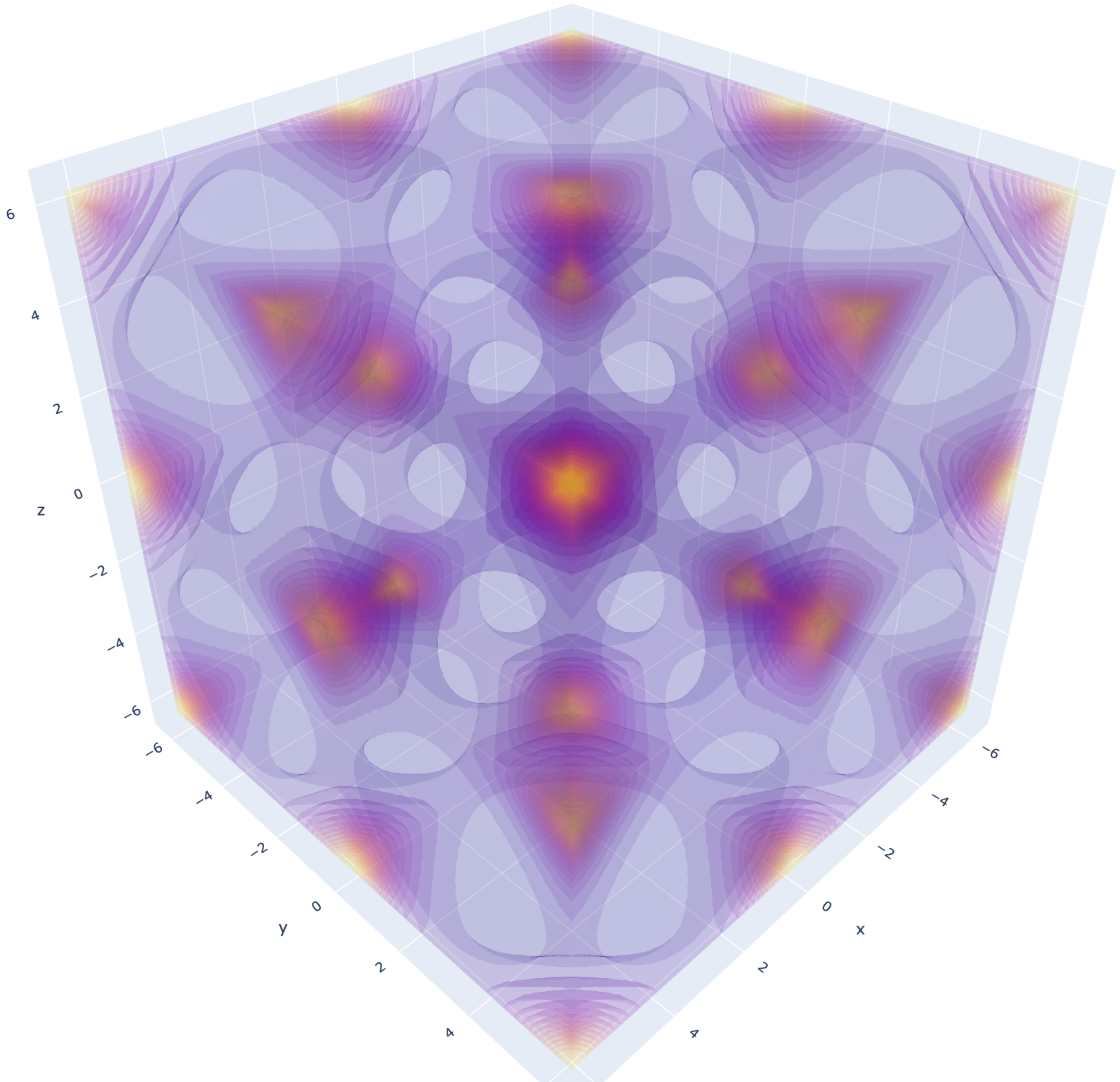}

	\caption{Periodic $\Lop$-spline with one knot $\bm{x}_1 = \bm{0}$ (Green's function)   associated to $\Lop = (0.5 \mathrm{Id} - \mathrm{D}_1)^2  (\mathrm{Id} - \mathrm{D}_2)^2  (1.5  \mathrm{Id} - \mathrm{D}_3)^2$ (left), $\Lop = (4 \mathrm{Id} - \Delta)^2$ (center) and $\Lop =(\mathrm{M}_{1}^{3.5})^3$ (right) over $2 \times 2 \times 2$ periods.}
	\label{fig:exponential3d}
\end{figure}

	\section{Measurement Space and Admissible Measurements} \label{sec:criteria}
		
	We have seen in Theorem \ref{theo:RT} that $\CL(\T^d)$ of a spline-admissible operator $\Lop$ is the exact function space from which the measurements can be taken such that the optimization problem is well-posed and the extreme points are characterizable as periodic $\Lop$-splines. 
	This is of practical significance: it delineates which measurements can be applied in order to keep the conclusions of Theorem \ref{theo:RT}.
		In this section, we provide conditions to identify if classical measurement procedures are applicable for a given spline-admissible operator $\Lop$.
		A special focus is given to Fourier sampling (Section \ref{sec:Fouriersampling}) and space sampling measurements (Section \ref{sec:spatialsampling}).
	


	\subsection{Fourier Sampling} \label{sec:Fouriersampling}
		
	The $\ek$ are infinitely differentiable, therefore in $\Sp(\T^d) \subseteq \CL(\T^d)$ for any spline-admissible operator $\Lop$ (the embedding is proven in Theorem \ref{theo:whatisCL}).
	One can therefore consider the measurement functional $\bm{\nu} = (e_{\bm{k}^1} , \ldots , e_{\bm{k}^M}) \in (\CL(\T^d))^M$ as linear measurements with distinct frequencies $\bm{k}^m$. 
	
	In order to apply Theorem \ref{theo:RT}, the only restriction is that  $\bm{\nu} : \NL \rightarrow \R^M$ should be injective over the finite dimensional null space of $\Lop$.
	Equivalently, we require that the frequencies $\bm{k}^m$ used for the Fourier sampling should include the null space frequencies $\bm{k}_1,\ldots,\bm{k}_{N_0}$ of $\Lop$. 
	For instance, with the $N$th order derivative operator $\Dop^N$ in dimension $d=1$, one should include $\nu_m = e_{0} = 1$ as a measurement functional.
		

	\subsection{Spatial Sampling} \label{sec:spatialsampling}
		
		In view of Theorem \ref{theo:RT}, classical sampling is an admissible measurement procedure if and only if  
		\begin{equation} \label{eq:RKBS}
		 \Sha(\cdot - \bm{x}_0)  \in \CL(\T^d), \qquad \forall \bm{x}_0 \in \T^d.
		\end{equation}
		We recover the classical notion of Reproducing Kernel Hilbert Spaces (RKHS), here in the context of (non reflexive) Banach spaces.
		Under the assumption \eqref{eq:RKBS} and selecting $\nu_m =  \Sha(\cdot - \bm{x}_m)$,  for any $f \in \ML(\R^d)$, we have $\bm{\nu}(f) = (f(\bm{x}_1),\ldots , f(\bm{x}_M) )$. 
		
		\begin{definition}[Sampling-Admissible Operator]
		A spline-admissible operator $\Lop$ is said to be \emph{sampling-admissible} if \eqref{eq:RKBS} holds.
		\end{definition}
				
		\begin{proposition} \label{prop:conditionsampling}
			Let $\Lop$ be a spline-admissible operator.
			Then, we have the equivalence
			\begin{equation} \label{eq:equivalenceCLC}
			\Lop \text{ is sampling-admissible} \quad \Longleftrightarrow  \quad \Kop \Sha \in \mathcal{C}(\T^d).
			\end{equation}
		\end{proposition}
		
		\begin{proof}
		First of all,  $\CL(\T^d)$ is shift-invariant because $\Lop$ is. Hence, $\Sha (\cdot - \bm{x}_0) \in \CL(\T^d)$ for every $\bm{x}_0$ if and only if $\Sha\in \CL(\T^d)$.
		Assume that $\Kop \Sha \in \mathcal{C}(\T^d)$.
		Then, $\Sha = \Lop \{ \Kop \Sha\} + \Proj_{\NL} \Sha \in \Lop (\mathcal{C}(\T^d)) + \NL = \CL(\T^d)$, as expected.
		If now $\Sha \in \CL(\T^d)$, then $\Sha = \Lop f + p $ with $f \in \mathcal{C}(\T^d)$ and $p\in \NL$.
		Therefore, we have that
		\begin{equation} 
			\Kop \Sha = \Kop \Lop f + \Kop p = f - \Proj_{\NL} f + \Kop p
		\end{equation}
		where we used that $\Kop \Lop f = f - \Proj_{\NL} f$ according to \eqref{eq:LopKop}.
		Moreover, $\Kop p \in \Sp(\T^d) \subset \mathcal{C}(\T^d)$ because $p \in \Sp(\T^d)$, $\Proj_{\NL} f \in \NL \subset \Sp(\T^d) \subset \mathcal{C}(\T^d)$, and $f \in \mathcal{C}(\T^d)$ by definition. Hence, $\Kop \Sha \in \mathcal{C}(\T^d)$. The equivalence \eqref{eq:equivalenceCLC} is established.
			\end{proof}

		
		Proposition~\ref{prop:conditionsampling} characterizes the validity of sampling measurements from the smoothness properties of the function $\Kop \Sha$, which plays a similar role to the one of the Green's function for differential operators in a non periodic setting.
		We now present some criteria based on the Fourier sequence of the pseudoinverse operator $\Kop$.
		
		
		\begin{proposition} \label{prop:conditionsamplingbis}
			Let $\Lop$ be a  spline-admissible operator with pseudoinverse $\Kop$.
			\begin{itemize}
			\item If $\sum_{\bm{k} \in \Z^d} \lvert \widehat{L^\dagger}[\bm{k}] \rvert = \sum_{\bm{k}\in \KL} \frac{1}{\lvert \widehat{L}[\bm{k}] \rvert} < \infty$, then $\Lop$ is sampling-admissible.
			\item If $\sum_{\bm{k} \in \Z^d} \lvert \widehat{L^\dagger}[\bm{k}] \rvert^2 = \sum_{\bm{k}\in \KL} \frac{1}{\lvert \widehat{L}[\bm{k}] \rvert^2} = \infty$, then $\Lop$ is not sampling-admissible.
 			\end{itemize}
			In particular, an spline-admissible operator admitting a spectral growth $\gamma > d$ is sampling-admissible. If $\gamma \leq d/2$, then, the operator is not sampling-admissible. 
		\end{proposition}
		
		\begin{proof}
		Recall that $( \widehat{L^\dagger}[\bm{k}])_{\bm{k}\in\Z^d}$ is the Fourier sequence of $\Kop \Sha$.
		The first condition means that Fourier sequence of $\Kop \Sha$ is in $\ell_1(\Z^d)$, from which we deduce the continuity of $\Kop \Sha$, and therefore \eqref{eq:RKBS}. The second condition means that $( \widehat{L^\dagger}[\bm{k}])_{\bm{k}\in\Z^d} \notin \ell_2(\Z^d)$, which is equivalent to $\Kop \Sha \notin \mathcal{L}_2(\T^d)$. Hence, $\Kop \Sha \notin \mathcal{C}(\T^d)$ and $\Lop$ is not sampling-admissible.
		When the operator admits a growth order $\gamma$, \eqref{eq:AGforglop} reveals the asymptotic behavior of $ \lvert \widehat{L^\dagger}[\bm{k}] \rvert  = \widehat{g}_\Lop[k]$ and the two last results follow.
		\end{proof}		
		
		\textit{Remarks.} If we only know that $( \widehat{L^\dagger}[\bm{k}])_{\bm{k}\in\Z^d} \in \ell_2(\Z^d)$ and $( \widehat{L^\dagger}[\bm{k}])_{\bm{k}\in\Z^d} \notin \ell_1(\Z^d)$, we cannot say anything in general. Indeed, we shall see in Proposition \ref{prop:1dwhoissamplingOK} that the fractional derivative $\mathrm{D}^\gamma$, which is typically in this regime for $\gamma \in (1/2,1]$ is not sampling-admissible. However, there exists sequences $(c_{\bm{k}})_{\bm{k}\in \Z^d}$ such that $\lvert c_{\bm{k}} \rvert \sim_{\infty} \lVert \bm{k} \rVert^{-\gamma}$, and such that $f = \sum_{\bm{k} \in \Z^d} c_{\bm{k}} e_{\bm{k}} \in \mathcal{C}(\T^d)$. An example for $d=1$ is given by the Hardy-Littlewood series, defined for $\gamma \in (1/2,1]$ by 
		\begin{equation}
		c_0 = 0 \text{ and } \forall k \neq 0, \ c_k = \frac{\mathrm{e}^{\mathrm{i} |k| \log |k|}}{|k|^\gamma}.
		\end{equation}
This Fourier series is known to converge uniformly to a continuous function~\cite[Section V-4]{zygmund2002trigonometric}. 
In that case, if $\Lop$ is defined by its Fourier sequence with $\widehat{L}[k]  = c_{k}$, $k \in \Z$, we have that $\lvert \widehat{L}[k] \rvert = \lvert \widehat{D^\gamma}[k] \rvert$, while $\Lop$ is sampling-admissible. Note that this behavior is based on strong phase oscillations of the coefficients.
We now visit the sampling-admissibility of the classes of operators introduced in Section \ref{sec:differentialop}. \\
		
		Moreover, we can easily generalize Propositions \ref{prop:conditionsampling} and \ref{prop:conditionsamplingbis} to other sampling measurements. For instance, sampling measurements on the derivative of the unknown function is allowed if and only if the derivative of the Dirac comb $\Sha' = \Dop \Sha$ is in the measurement space, with potential applications to spline-based reconstruction with tangent control~\cite{Uhlmann2016hermite}. 
		
		\subsubsection{Sampling-Admissibility of Univariate Operators}
		
		The ambiant dimension is $d=1$. We investigate the sampling-admissibility of classical differential operators, and their fractional counterparts.
		
		\begin{proposition} \label{prop:1dwhoissamplingOK}
		Let $\gamma \geq 0$, $\alpha \in \R$. We have the equivalences:
		\begin{align} \label{eq:1}
		\mathrm{D}^\gamma \text{ is sampling-admissible }  & \Longleftrightarrow (\mathrm{D} + \alpha \mathrm{Id})^\gamma  \text{ is sampling-admissible } 
		\Longleftrightarrow (-\Delta)^{\gamma/2} \text{ is sampling-admissible }  \\ \label{eq:2}
		&  \Longleftrightarrow  (\alpha \mathrm{Id} - \Delta)^{\gamma/2}  \text{ is sampling-admissible }    \Longleftrightarrow  \gamma > 1. 
		\end{align}
		Moreover, the Mat\'ern and Sobolev operators $\mathrm{M}_\epsilon^\beta$ and $\mathrm{W}_{\epsilon,\mu}^\beta$ are sampling-admissible for any $\epsilon, \beta, \nu$. 
		\end{proposition}
		\begin{proof}
		All the considered operators have a spectral growth $\gamma$. Moreover, a spline-admissible operator with spectral growth $\gamma > 1$ is sampling admissible according to Proposition \ref{prop:conditionsamplingbis}. Hence, the condition $\gamma > 1$ implies the sampling admissibility of all the considered operators in\eqref{eq:1} and \eqref{eq:2}. Similarly, the condition $\gamma \leq 1/2$ implies that the Green's function is not square-integrable, and therefore not continuous, and the operators are not sampling-admissible in this case. 
		
		For $\gamma = 1$, the function $\mathrm{D}^{\dagger} \Sha$ is  given by
$\mathrm{D}^{\dagger} \Sha(x) = \pi - x$ for $x \in \T$, and then periodized. Since $\mathrm{D}^{\dagger} \Sha(0^+) = 1/2 \neq \mathrm{D}^\dagger \Sha(2\pi ^-) = -1/2$, the function is discontinuous. Hence, $\mathrm{D}$ is not sampling-admissible.
		For the case $1/2 < \gamma < 1$, we refer to \cite[Eq. (8.10), Section XII-8]{zygmund2002trigonometric}, where the function $(\mathrm{D}^\gamma)^{\dagger} \Sha$---denoted by $\Psi_{\gamma}$ up to a rescaling---is shown to be such that 
		\begin{equation}
		\forall x \in (-2\pi,2\pi), \quad	(\mathrm{D}^\gamma)^{\dagger} \Sha (x) = \frac{1}{\Gamma(\gamma)} (x)_+^{\gamma -1} + r_\gamma(x),
		\end{equation}
		with $\Gamma$ the Gamma function and $r_\gamma$ a function that is infinitely differentiable on $(-2\pi,2\pi)$.
		The function $x \mapsto  (x)_+^{\gamma -1}$ being discontinuous at the origin for $\gamma \in (1/2, 1)$, we deduce that $(\mathrm{D}^\gamma)^{\dagger} \Sha$ is also discontinuous, and $\mathrm{D}^\gamma$ is not sampling-admissible.

		 For $\Lop = (\mathrm{D} + \alpha \mathrm{Id})^\gamma$, we simply observe that
		 \begin{equation} \label{eq:argufracandexpo}
		 	 \widehat{L^\dagger}[k] - \widehat{(D^\gamma)^{\dagger}}[k]  = \frac{1}{(\mathrm{i}  k)^\gamma} \left( \frac{1}{(1 + \frac{\alpha}{ \mathrm{i} k})^\gamma} - 1 \right) =   O\left( \frac{1}{k^{\gamma +1}}\right).
		 \end{equation}
		In particular, we deduce that $\Kop\Sha - \mathrm{D}^\gamma \Sha \in \mathcal{C}(\T)$, because its Fourier sequence is in $\ell_1(\Z)$. This means in particular that $\Kop\Sha \in \mathcal{C}(\T)$ if and only if $\mathrm{D}^\gamma \Sha \in \mathcal{C}(\T)$, and the result follows. 
		
		We show similarly that $ (-\Delta)^{\gamma/2}$ is spline-admissible if and only if $(\alpha \mathrm{Id} - \Delta)^{\gamma/2}  $ is.
		For $\gamma   = 1$, we actually have that
		\begin{equation}
			\widehat{((-\Delta)^{1/2})^{\dagger}}(x) = 2 \sum_{k\geq 1} \frac{e_k(x)}{k} =  \log (1 - \cos x) + \log 2, \qquad \forall x \in \T,
		\end{equation}
		which is clearly discontinuous (and unbounded) around $0$. Fix $\gamma \in (1/2,1)$. Assume that the periodic function $x \mapsto f (x) = \sum_{k\geq 1} \widehat{f}[k] \cos( k x) \in \mathcal{C}(\T)$ has positive Fourier coefficients. Then, for any $\alpha \in (0,1)$, we have that $\sum_{k \geq 1} k^{\gamma -1} \widehat{f}[n] < \infty$~\cite[Theorem 1]{boas1966fourier}. 
		Consider the function 
		\begin{equation}
		f(x) = \widehat{((-\Delta)^{\gamma/2})^{\dagger}}(x) = 2  \sum_{k\geq 1}\frac{e_k(x)}{k^\gamma}, \qquad \forall x \in \T.
		\end{equation}
		Then, $\sum_{k \geq 1} k^{\alpha -1} \widehat{f}[n] = \sum_{k \geq 1} k^{- (\gamma + 1-  \alpha)}$, which is infinite as soon as $\gamma \leq \alpha$. Applying the contraposition of \cite[Theorem 1]{boas1966fourier}, we therefore deduce that $f$ is discontinuous, hence $ (-\Delta)^{\gamma/2}$ is not spline-admissible.
		
		
		For the Mat\'ern and Wendland operators, we  remark that their growth order is at least equal to $2$ according to Propositions \ref{prop:maternop} and \ref{prop:wendlanderie}, implying the sampling-admissibility.
		
		\end{proof}
		
		\subsubsection{Sampling Admissibility of multivariate Operators}
		
		The ambiant dimension is $d\geq 1$. We evaluate the sampling admissibility of the separable and isotropic operators introduced above.
		
		\begin{proposition}\phantom{test}
		\begin{itemize}
\item		Let $\Lop_i$ be $d$ univariate spline-admissible operators with trivial null space. Then, the separable multivariate spline-admissible operator $\Lop = \prod_{i=1}^d \Lop_i$ is sampling-admissible if and only if each $\Lop_i$ is. In particular, $\prod_{i=1}^d (\Dop_i - \alpha_i \mathrm{Id})^{\gamma_i}$ is sampling-admissible if and only if $\gamma_i > 1$ for any $i= 1 , \ldots , d$. 
		
\item		Let $\gamma \geq 0$. Then, we have the relations: 
		\begin{equation}\label{eq:twoequivalences}
		\gamma > d \Longrightarrow (-\Delta)^{\gamma/2} \text{ is sampling-admissible } \Longleftrightarrow (\alpha^2 \mathrm{Id} -\Delta)^{\gamma/2} \text{ is sampling-admissible }. 
		\end{equation}
		\end{itemize}
		\end{proposition}		
		
		\begin{proof}
		The Green's function of the separable operator $\Lop$ is $g_{\Lop}(\bm{x}) = g_{\Lop_1} (x_1) \ldots g_{\Lop_d}(x_d)$ for every $\bm{x} = (x_1, \ldots, x_d) \in \T^d$. Then, $g_{\Lop} \in \mathcal{C}(\T^d)$ if and only if each $g_{\Lop_i} \in \mathcal{C}(\T)$ for $i=1, \ldots , d$. Applying this principle to $L_i = (\Dop - \alpha_i \mathrm{Id})^{\gamma_i}$, which is sampling-admissible if and only if $\gamma_i > 1$ according to Proposition \ref{prop:1dwhoissamplingOK}, gives the second result.
		
		Let $\gamma > d$. Then, $(-\Delta)^{\gamma/2}$ has a growth order $\gamma > d$ and is therefore sampling-admissible using Proposition \ref{prop:conditionsamplingbis}.
		The  equivalence in \eqref{eq:twoequivalences} follows from an argument identical to \eqref{eq:argufracandexpo}: we readily show that the difference of the two Green's functions has a summable Fourier series and is therefore continuous.  
		\end{proof}
		
		\textit{Remark.} 
		The case of multidimensional Fourier series is more evolved than for the univariate case~\cite[Chapter XVII]{zygmund2002trigonometric} and the literature on this subject is much less developed (see \cite{shapiro2011fourier} for an extensive discussion on these topics). In particular, the arguments we used in Proposition \ref{prop:1dwhoissamplingOK} for the equivalence between the sampling-admissibility of $(-\Delta)^{\gamma/2}$ and $\gamma > 1$ are not directly applicable. We conjecture however that the $d$-dimensional generalization of this result is true. That is, $\gamma > d$ if and only if $(-\Delta)^{\gamma/2}$ is sampling-admissible. 

	\subsection{Square-Integrable Measurement Functions} \label{sec:spatialsampling}
	
	We provide a simple characterization of the spline-admissible operators for which the space of square-integrable measurement functions is included in the measurement space in Proposition \ref{prop:squaremeasurement}. 
	
	\begin{proposition}
	\label{prop:squaremeasurement}
		Let $\Lop$ be a spline-admissible opeartor with pseudoinverse $\Kop$. Then, the following equivalences hold:		
		\begin{equation}
		\mathcal{L}_2(\T^d) \subseteq \CL(\T^d) \quad \Longleftrightarrow \quad  \Kop \Sha \in \mathcal{L}_2(\T^d) \quad \Longleftrightarrow  \quad (\widehat{L^{\dagger}}[\bm{k}])_{\bm{k}\in \Z^d} \in \ell_2(\Z^d). 
		\end{equation}	
		More generally, we have the following equivalences:
		\begin{equation}
		 \mathcal{H}_{2}^{-\tau}(\T^d) \subseteq \CL(\T^d) \quad \Longleftrightarrow \quad  \Kop \Sha \in  \mathcal{H}_{2}^{\tau}(\T^d) 
		 \quad \Longleftrightarrow  \quad  ( \lVert \bm{k} \rVert^\tau \widehat{L^{\dagger}}[\bm{k}])_{\bm{k}\in \Z^d} \in \ell_2(\Z^d)
		\end{equation}
		for any $\tau \in \R$, with $ \mathcal{H}_{2}^{\tau}(\T^d) $ the periodic Sobolev space of smoothness $\tau$ defined in \eqref{eq:sobolevspace}.
	\end{proposition}
	
	\begin{proof}
	The proof for periodic Sobolev spaces works identically, therefore we first prove the first part of Proposition \ref{prop:squaremeasurement}. 
	The second equivalence is simply due to the Parseval relation. We therefore focus on the first one. \\
	
	$\Longleftarrow$ Assume first that $\Kop \Sha \in \mathcal{L}_2(\T^d)$. 
	Then, we have that $\lVert \widehat{L^\dagger} \rVert_{\ell_2}^2 = \sum_{\bm{k}\in \Z^d} | \widehat{L^{\dagger}}[\bm{k}]|^2 < \infty$. Let $f \in \mathcal{L}_2(\T^d)$. Then,
	\begin{align} \label{eq:equationforcontinuity1}
		\lVert f \rVert_{\CL}^2 &= \lVert \Kop^* f \rVert_\infty^2 + \lVert \mathrm{Proj}_{\NL} f \rVert_2^2 \in [0,\infty].
	\end{align}
	Since $\mathrm{Proj}_{\NL}$ is an orthogonal projector, we have $\lVert \mathrm{Proj}_{\NL} f \rVert_2^2 \leq \lVert f \rVert_2^2$. Moreover, we have
	\begin{align} \label{eq:lasteqpkopf}
	  \lVert \Kop^* f \rVert_\infty &= \sup_{\bm{x}\in\T^d} \left\lvert \sum_{\bm{k}\in \Z^d} \overline{\widehat{L^\dagger} [\bm{k}] } \widehat{f}[\bm{k}] e_{\bm{k}} (\bm{x}) \right\rvert \leq \sum_{\bm{k}\in \Z^d} \lvert  \widehat{L^\dagger} [\bm{k}]  \rvert \lvert  \widehat{f}[\bm{k}] \rvert \leq \lVert \widehat{L^{\dagger}} \rVert_{\ell_2} \lVert \widehat{f} \rVert_{\ell_2} = \lVert \Kop \Sha \rVert_2 \lVert f \rVert_2 , 
	\end{align}	
	where we used Cauchy-Schwarz in the last inequality and the Parseval relation in the last equality in \eqref{eq:lasteqpkopf}. 
	This shows that $\Kop^* f \in L_{\infty}(\T^d)$. Let us show moreover that $\Kop^* f  \in \mathcal{C}(\T^d)$. 
	We define $g_K = \sum_{\lVert \bm{k} \rVert \leq K} \widehat{L^{\dagger}}[\bm{k}] \widehat{f} [\bm{k}] e_{\bm{k}}$, which is the truncated Fourier series of $\Kop^* f$. The functions $g_K$ are continuous (and even infinitely differentiable). Moreover, we have that, for any $\bm{x} \in \T^d$, 
	\begin{equation} \label{eq:partialsumforcontinuity}
	\lvert \Kop^* f (\bm{x}) - g_K(\bm{x}) \rvert 
	=
	\left\lvert \sum_{\lVert \bm{k} \rVert > K}  \overline{\widehat{L^{\dagger}}[\bm{k}] } \widehat{f} [\bm{k}] e_{\bm{k}} (\bm{x}) \right\rvert
	\leq 
	\sum_{\lVert \bm{k} \rVert > K}  \lvert \widehat{L^{\dagger}}[\bm{k}]  \rvert \lvert \widehat{f} [\bm{k}] \rvert
	\leq 
	\left(\sum_{\lVert \bm{k} \rVert > K}  \lvert \widehat{L^{\dagger}}[\bm{k}]  \rvert^2 \right)^{1/2} \left(\sum_{\lVert \bm{k} \rVert > K}  \lvert \widehat{f} [\bm{k}] \rvert^2\right)^{1/2}
	\end{equation}
	where we used that $\lvert e_{\bm{k}}(\bm{x}) \rvert \leq 1$ and the Cauchy-Schwarz inequality. Both $\sum_{\lVert \bm{k} \rVert > K}  \lvert \widehat{L^{\dagger}}[\bm{k}]  \rvert^2$ and $\sum_{\lVert \bm{k} \rVert > K}  \lvert \widehat{f} [\bm{k}] \rvert^2$ and \eqref{eq:partialsumforcontinuity} holds for any $\bm{x}\in \T^d$, hence $\lVert  \Kop^* f   - g_K\rVert_\infty \rightarrow 0$ when $K\rightarrow 0$. Then, $ \Kop^* f $ is the limit of the continuous functions $g_K$ for the uniform convergence, and is therefore continuous. In particular, $f \in \CL(\T^d)$ showing the set inclusion $\mathcal{L}_2(\T^d) \subset \CL(\T^d)$. 
	
	Using \eqref{eq:equationforcontinuity1} and \eqref{eq:lasteqpkopf}, we moreover deduce that
	\begin{equation}
	\lVert f \rVert_{\CL}  \leq (1 + \lVert \Kop \Sha \rVert_2^2)^{1/2} \lVert f \rVert_2, 
	\end{equation}
	which, together with the set inclusion $\mathcal{L}_2(\T^d) \subset \CL(\T^d)$, implies the topological embedding $\mathcal{L}_2(\T^d) \subseteq \CL(\T^d)$. \\
	
	 $\Longrightarrow$  Assume now that $\mathcal{L}_2(\T^d) \subseteq \CL(\T^d)$. Moreover, we know with Theorem \ref{theo:whatisCL} that $\Sp(\T^d)$ is dense in $ \CL(\T^d)$. This implies that the embedding $\mathcal{L}_2(\T^d) \subseteq \CL(\T^d)$ is also dense, from which we deduce the topological embedding $\ML(\T^d) \subseteq \mathcal{L}_2(\T^d)$, using that the $(\mathcal{L}_2(\T^d))' = \mathcal{L}_2(\T^d)$ and $(\CL(\T^d))' = \ML(\T^d)$ due to Theorem \ref{theo:RieszMarkovgeneralized}. Finally, since $\Kop \Sha \in \ML(\T^d)$ (because $\Lop \Kop \Sha = \Sha + \Proj_{\NL} \Sha \in \mathcal{M}(\T^d)$), we conclude that $\Kop \Sha \in \mathcal{L}_2(\T^d)$ as expected. 
	\end{proof}

	We say that a spline-admissible operator is 	\emph{$\mathcal{L}_2$-admissible} is its measurement space contains the square-integrable functions.
	In particular, a {$\mathcal{L}_2$-admissible} operator admits indicator functions as valid measurement functions.
	From the previous results, we deduce that a sampling admissible operator is necessarily $\mathcal{L}_2$-admissible. Indeed, the sampling admissibility implies that $\widehat{L}^\dagger \in \ell_2(\Z^d)$ (second part of Proposition \ref{prop:conditionsamplingbis}), which is equivalent to the $\mathcal{L}_2$-admissibility with Proposition \ref{prop:squaremeasurement}.
	The next corollary reveals which pseudo-differential operators introduced in Section \ref{sec:differentialop} are {$\mathcal{L}_2$-admissible}. 
		
	\begin{corollary} \label{coro:L2admonpseudodiff}
	Let $\gamma \geq 0$, $\alpha \in \C$. 
	Then, the univariate spline-admissible operators $\mathrm{D}^\gamma$ and $(\mathrm{D} + \alpha \mathrm{Id})^\gamma$ are $\mathcal{L}_2$-admissible if and only if $\gamma > 1/2$.
	In any ambiant dimension $d\geq 1$, the multivariate spline-admissible operators  $(-\Delta)^{\gamma/2}$ and $(\alpha^2 \mathrm{Id} -\Delta)^{\gamma/2}$ are $\mathcal{L}_2$-admissible   if and only if $\gamma > d/2$.
	Finally, Mat\'ern and Wendland operators are $\mathcal{L}_2$-admissible. 
	\end{corollary}
	
	\begin{proof}
		The proof is very simple using Proposition \ref{prop:squaremeasurement}, which implies that $\Lop$ is $\mathcal{L}_2$-admissible if and only if $\sum_{\bm{k}\in\Z^d} | \widehat{L^\dagger}[\bm{k}] |^2 < \infty$. For instance, the Sobolev operator $\Lop_{\gamma,\alpha} = (\alpha^2 \mathrm{Id} - \Delta)^{\gamma/2}$ is such that 
		\begin{equation}
		\sum_{\bm{k}\in\Z^d} | \widehat{L_{\gamma,\alpha}^\dagger}[\bm{k}] |^2 = 
		\sum_{\bm{k}\in\Z^d} \frac{1}{(\alpha^2 + \lVert \bm{k}\rVert^2)^{\gamma}},
	\end{equation}
	which is finite if and only if $2 \gamma > d$, as expected. Finally, the Mat\'ern and Wendland operators are $\mathcal{L}_2$-admissible as any sampling-admissible operators. 
	\end{proof}
	
	The results of Proposition \ref{prop:squaremeasurement} and Corollary \ref{coro:L2admonpseudodiff} are consistent with \cite[Proposition 8]{simeoni2020functionalpaper}, which obtains similar but partial results over the $d$-dimensional sphere $\mathbb{S}^d$. The two main differences are that \cite{simeoni2020functionalpaper} only provides a \emph{sufficient} condition for the set inclusion $\mathcal{L}_2(\mathbb{S}^d) \subset \CL(\mathbb{S}^d)$, and for a specific class of spline-admissible operators.

	\section{Discussion and Conclusion} \label{sec:discussion}
	For the sake of simplicity, we restrict our attention in this section to the case where the regularizing spline-admissible operator $\Lop$ has a trivial null space, that is, $\NL = \{0\}$. The pseudoinverse is then an inverse.
	
		\subsection{Practical Discretization Schemes} \label{sec:algo}
Theorem \ref{theo:RT} can be used to derive canonical discretization schemes for the optimization problem \eqref{eq:optiwellstated}. Indeed, the representer theorem tells us that the extreme point solutions of \eqref{eq:optiwellstated} take the form of periodic $\Lop$-splines with sparse innovations---\emph{i.e.}, less innovations then available data. One idea for solving \eqref{eq:optiwellstated} in practice consists then in discretizing it by replacing the function $f : \R^d \rightarrow \R$  by a non-uniform $\Lop$-spline with unknown knots and weights. The spline innovations must then be estimated from the data. While the spline weights can be recovered from the measurements using a convex optimization problem, the same is not true for the knots, which are consequently much harder to estimate.

To circumvent this issue, one strategy consists in considering overparametrised uniform splines with knots chosen over a very fine uniform grid to approximate the non-uniform splines with unknown knots. The weights are then recovered by solving a discrete penalized basis pursuit problem using state-of-the-art proximal algorithms such as the ones discussed in \cite[Section 5.1]{simeoni2020functionalpaper}. Such discretization schemes were investigated and analyzed in \cite[Section V.B]{gupta2018continuous} and \cite[Section 5.1]{simeoni2020functionalpaper} in the Euclidean and spherical setting respectively. Extensions to B-splines and multiresolution grids were also considered in \cite{Debarre2019}. 
While conceptually simple, this approach is computationally wasteful since the approximating uniform spline typically has much more innovations than the number of measurements.

As a potential cure to this issue, one could consider meshfree algorithms capable of directly recovering the non uniform knots in the continuum. Candidate reconstruction algorithms include the Cadzow Plug-and-Play Gradient Descent (CPGD) algorithm \cite{simeoni2020cpgd} as well as the Franck-Wolfe algorithm and its variants \cite{denoyelle2019sliding,flinth2020linear}. Both algorithms have been successfully used for the reconstruction of periodic Dirac streams. To the best of our knowledge however, they have not yet been tested for the purpose of reconstructing spline knots, and would therefore need to be adapted for this specific purpose. 


		\subsection{Comparison with Generalized Periodic Tikhonov Regularization}
		We compare here the solutions of the periodic TV-regularized problem \eqref{eq:optiwellstated}  to its analog with quadratic Tikhonov regularization considered in~\cite{Badou}. The latter takes the form
		\begin{equation} \label{eq:L2optipb}
		{\min} \ E(\bm{y}, \bm{\nu} (f ) )  +   \lambda \lVert \Lop f \rVert^2_{\mathcal{L}_2},
		\end{equation}
		where  $E$ is a cost function sharing the same properties as in Theorem \ref{theo:RT} and $\Lop$ is a spline-admissible operator. 			
According to \cite[Theorem 1]{Badou} the solution of \eqref{eq:L2optipb} is unique  and of the form
					\begin{equation}  \label{eq:uniquedudeL2}
					f_{\mathrm{opt}} (x) = \sum_{m=1}^M a_m ( h_{\Lop} * \nu_m (x))  ,
					\end{equation}
					where $h_{\Lop} =(\Lop^* \Lop)^{-1}\{\Sha\}= \sum_{k\in \Z} \frac{e_k}{\lvert \widehat{L}[k]|^2}$. 
			The main differences between the two settings are then as follows:
			\begin{itemize}
\item			 For Tikhonov regularization the solution is unique, which is not the case in general for \eqref{eq:optiwellstated}, as revealed by  Theorem \ref{theo:RT}. Tikhonov regularization	is hence a more effective regularization strategy when it comes to enforcing uniqueness of the solution.
\item		The unique solution of \eqref{eq:L2optipb} lives in a finite dimensional space of dimension $M$ generated by the functions $\{h_{\Lop} * \nu_m, \,1 \leq m \leq M\}$. This is reminiscent of kernel methods: the function to reconstruct lies in a possibly infinite dimensional Hilbert space, but the regularization procedure enforces the solution to lie in a finite-dimensional space determined by the choice of the kernel and the measurement functionals.
		In contrast, TV regularization benefits from an infinite union of finite-dimensional subsets, given by periodic $\Lop$-splines with less than $M$ knots at any possible locations.
		This is known to improve the adaptiveness of the method and yield higher accuracy estimates with sharp variations~\cite{lu2008theory,eldar2009robust}. 
		\item		Finally, for the Tikhonov optimization problem \eqref{eq:L2optipb}, the measurement functionals $\nu_m$ directly impact the form of the estimate \eqref{eq:uniquedudeL2}. For instance, with Fourier measurements, this conducts to the well-known Gibbs phenomenon and the presence of oscillations in the reconstruction (see \cite[Figure 4(a)]{gupta2018continuous}). In contrast, the form of the solutions of \eqref{eq:optiwellstated} is agnostic to the measurement process and  depends only on the chosen regularizing operator $\Lop$.  Solutions to the TV regularized optimisation problem \eqref{eq:optiwellstated} are hence less sensitive to Gibbs-like phenomena.
 \end{itemize}

		\subsection{Comparison with TV Regularization over $\R^d$}
		
		As we have seen in Section \ref{sec:intro}, many recent works have investigated the reconstruction of continuous-domain functions $f : \R^d \rightarrow \R$ from finitely many noisy measurements.
		Our paper contributes to  this global effort by considering the use of TV-based regularization norms for the reconstruction of periodic functions.
		We believe that the periodic setting has several remarkable advantages that greatly facilitate the construction of the framework. 		
				
		First, Schwartz functions over $\R^d$ mix smoothness and decay properties, while periodic Schwartz functions must only be smooth. This significantly simplifies the construction of the native space and the measurement space, as can be appreciated when comparing to the derivations in \cite{unser2019native}. In the periodic setting, we are moreover able to provide complete characterizations  of spline-admissible operators from their Fourier symbol and  of sampling-admissible operators with concrete criteria applicable to classical families of (possibly fractional) pseudo-differential operators. In both cases and to the best of our knowledge, similar results are only partially known in the non periodic setting.
		
		Second, even if splines play a central role for both the periodic and non periodic settings, the construction of the splines differs. Consequently, the form of the extreme points solutions differs. Consider for instance the univariate operator $\Lop = \Dop^N$. In the non periodic setting, an extreme-point solution has at most $(M-N_0)$-knots~\cite[Theorem 2]{Unser2017splines}. For the periodic case, extreme points solutions has at most $M$-knots whose weights satisfy the linear condition $\bm{\mathrm{M}} \bm{a} = \bm{0}$ (see Theorem \ref{theo:RT}). 
		Finally, it is worth noting that the dimension $N_0$ of the null space of $\Lop$ depends on the chosen setting: for the periodic case, $N_0=1$ when $\Lop=D^{N}$, while $N_0 = N+1$ over the real line. 
				
	\subsection{Conclusions}
	
	We presented a general framework for the reconstruction of sparse and periodic functions defined over a continuum from finitely many noisy linear measurements.
	This is achieved by using total variation based regularizations in addition to a data fidelity term.
	The main novelty of our work was to address the problem in full generality for periodic functions in a self-contained manner.
	In particular, we characterized the complete class of periodic operators and periodic measurement functionals for which a periodic representer theorem for the solution of \eqref{eq:optipb} can be obtained. 
	In a future work, we plan to work on practical aspects of the proposed periodic framework, including discretization procedures, reconstructions algorithms, and practical applications to signal processing tasks.
	
	\section*{Acknowledgments} 
	The authors are grateful to Michael Unser, Thomas Debarre, and Quentin Denoyelle for interesting discussions at the early stage of this research project. 
	Julien Fageot is supported by the Swiss National Science Foundation (SNSF) under Grant P2ELP2\_181759. For this work, Matthieu Simeoni was in part supported by the Swiss National Science Foundation grant number 200021 181978/1, “SESAM - Sensing and Sampling: Theory and Algorithms”. 

\small
\bibliographystyle{IEEEtran}
\bibliography{references}

\end{document}